\documentclass[a4paper, 11pt]{article}
\usepackage{mypackage}
\usepackage[en-GB]{datetime2}\DTMlangsetup[en-GB]{ord=raise}
\usepackage{ragged2e}
\usepackage{csquotes}
\usepackage{bbm}
\usepackage{bm}
\usepackage{subcaption}

\usepackage{blindtext}
\usepackage{caption}
\usepackage{verbatim}
\usepackage[scr=boondoxo, scrscaled=.98]{mathalfa}
\usepackage[ruled,linesnumbered]{algorithm2e}
\usepackage{pgfplotstable}
\pgfplotsset{compat=1.18}
\usepackage{multirow}

\allowdisplaybreaks

\newcommand{\mfn}{\mathfrak{N}}

\newcommand{\mcd}{\mathcal{D}}
\newcommand{\mcl}{\mathcal{L}}
\newcommand{\mcn}{\mathcal{N}}
\newcommand{\mcp}{\mathcal{P}}
\newcommand{\mct}{\mathcal{T}}
\newcommand{\mcx}{\mathcal{X}}
\newcommand{\mcy}{\mathcal{Y}}

\newcommand{\mvert}{\,|\,}

\newcommand{\atwo}{\langle2\rangle}

\newcommand{\Eta}{\mathrm{H}}
\newcommand{\Tau}{\mathrm{T}}
\newcommand{\tx}{\tilde{x}}
\newcommand{\ty}{\tilde{y}}
\newcommand{\tE}{\tilde{E}}
\newcommand{\tL}{\widetilde{L}}
\newcommand{\tM}{\tilde{M}}
\newcommand{\tR}{\tilde{R}}
\newcommand{\tX}{\tilde{X}}

\newcommand{\tPi}{\widetilde{\Pi}}
\newcommand{\tSigma}{\widetilde{\Sigma}}
\newcommand{\tXi}{\widetilde{\Xi}}
\newcommand{\inlawto}{\stackrel{\mcd}{\longrightarrow}}
\newcommand{\G}{\mathbb{G}}
\newcommand{\eqinlaw}{\stackrel{\mcd}{=}}

\xdefinecolor{dgreen}{rgb}{0,0.5,0}
\xdefinecolor{dgreenish}{rgb}{0,0.5,0.2}
\xdefinecolor{dred}{rgb}{0.65,0,0}
\xdefinecolor{dpurple}{rgb}{0.6,0,0.6}

\begin{document}


\title{Assignment Based Metrics for Attributed Graphs}

\author{Dominic Schuhmacher and Leoni Carla Wirth\footnote{Supported by Deutsche Forschungsgemeinschaft GRK 2088.}\\[2mm]
Institute for Mathematical Stochastics\\
University of G\"ottingen
}
\maketitle

\begin{abstract}
We introduce the Graph TT (GTT) and Graph OSPA (GOSPA) metrics based on optimal assignment, which allow us to compare not only the edge structures but also general vertex and edge attributes of graphs of possibly different sizes. We argue that this provides an intuitive and universal way to measure the distance between finite simple attributed graphs.
Our paper discusses useful equivalences and inequalities as well as the relation of the new metrics to various existing quantifications of distance between graphs. 

By deriving a representation of a graph as a pair of point processes, we are able to formulate and study a new type of (finite) random graph convergence and demonstrate its applicability using general point processes of vertices with independent random edges.

Computational aspects of the new metrics are studied in the form of an exact and two heuristic algorithms that are derived from previous algorithms for similar tasks. As an application, we perform a statistical test based on the GOSPA metric for functional differences in olfactory neurons of \textit{Drosophila} flies.

\par\medskip
{\footnotesize

}
\end{abstract}

\section{Introduction}

Graphs belong to the fundamental structures in mathematics. From a theoretical point of view, they offer a diversity of sometimes very challenging problems based on a simple set of rules. From an applied point of view, they serve as flexible models for any kind of scientific question studying ``things'' and the relations between them. The study of random graphs and the use of graphs as a data structure has surged in the 21st century (\cite{grimmett2018}; \cite{frieze2023}; \cite{vanderhofstad2016,vanderhofstad2023}; \cite{crane2018}; \cite{hamilton2020}) and various new fields have arisen such as \emph{network science}, \emph{graph machine learning} and \emph{graph neural networks}.

In the present paper we are interested in graph metrics that allow for an intuitive quantification of discrepancies, not only between the connectivity structures of the graphs, but also between their vertex and edge attributes (labels, weights, \ldots). This is achieved by assigning the vertices of one graph to those of the other as well as possible with respect to (pseudo)metrics on the underlying spaces of attributes. As part of the vertex attributes, we are especially interested in coordinates localizing the vertices in Euclidean space or on a Riemannian manifold, but the presented metrics are very general and work well for many other kinds of (vertex and edge) attributes and their associated underlying (pseudo)metrics. Our goal is to establish an intuitive metric structure on a suitable graph space. From the point of view of probability theory, this structure allows to express random graphs concisely as random elements (pairs of point processes in fact) in a complete separable metric space, gives a suitable topology for studying convergence in distribution of attributed graphs and allows to quantify the rate of convergence in terms of Wasserstein metrics; see Subsection~\ref{ssec: convergence considerations} for details. From the point of view of statistics and data science, the metric structure allows to apply a host of modern techniques for data on metric spaces; see Section~\ref{sec:realdata}. 



Various quantifications of the discrepancy between graphs have been studied before. Often the goal was to find graph isometries and solve pattern recognition problems or to develop statistical methods for more general problems. \textcite{emmert2016} and \textcite{evans2019} give an overview of the developments in this area. In what follows, we list and compare some of these measures for quantifying the (dis)similarity between two graphs. We emphasize that many of them do not satisfy metric properties or it has at least not been established if they do by their inventors.

An important approach to measure graph (dis-)similarities is the graph edit distance introduced by \textcite{bunke1983}. The graph edit distance measures the dissimilarity of two graphs by calculating the minimal cost required to transform one graph into the other using edit operations, such as deleting, inserting and substituting vertices and edges. Various modifications of the graph edit distance have been studied, e.g. \textcite{marzal1993} proposed to apply a cost function normalized by the number of edit operations to the graph edit distance while \textcite{Bunke1997} considered the graph edit distance with a certain cost function showing a correspondence of the resulting graph match and the maximum common subgraph.

Another class of (dis-)similarity measures are graph kernels that are used in machine learning to apply kernel methods on graph structure data. Graph kernels are symmetric, positive semi-definite functions designed to capture the similarity of a set of graphs. In the last years several approaches for the construction of suitable graph kernels emerged, see e.g. \textcite{kriege2020} and \textcite{nikolentzos2021} for an overview. One of these approaches is the optimal assignment method, proposed e.g. by \textcite{frohlich2005}, that suggests to measure the similarity of two graphs by searching for a matching of components that maximizes the overall similarity of these components. \textcite{frohlich2005} showed that this approach does in general not lead to positive semi-definite kernels. However, \textcite{vert2008} found that strong base kernels, i.e.\ kernels that are induced by a so-called hierarchy, yield positive semi-definite optimal assignment kernels. Further, there are several approaches for machine learning with indefinite similarity functions. For example, \textcite{johansson2015} proposed a similarity measure between two graphs that is constructed by embedding graphs into a geometric space using different methods like the adjacency matrix or the Graph Laplacian and then applying the maximum-weight-matching to these embeddings. They further use landmarks and this similarity measure to solve graph classification problems.  
A different approach by \textcite{luss2007} modifies the Support Vector Machine algorithm such that it is applicable to indefinite similarity functions by viewing the indefinite similarity function as a noisy sample of an unknown positive semi-definite function. This approach is, for example, applied by \textcite{Nikolentzos2017} to an indefinite similarity function that is constructed by applying the Earth Mover's distance to the optimal assignment kernel. 

In \textcite{Memoli2011} the Gromov--Wasserstein distance between metric measure spaces has been introduced. \textcite{MemoliNeedham2022} discuss a Gromov--Monge variant, which avoids the splitting of mass. For graphs with (possibly) weighted vertices that have spatial positions or can otherwise be embedded in a metric space of features, these serve as (pseudo-)metrics, which however do not take positions (or other features) of the vertices into account, but only their distances. Since for a graph metric this is usually too restrictive, \textcite{VayerEtAl2019} introduced the Fused Gromov--Wasserstein (FGW) distance, which is a convex combination of the Wasserstein distance between (metric embeddings of) the vertices and the Gromov--Wasserstein distance.

A different, spectral approach called GOT, which is also based on optimal-transport, is proposed in \textcite{PetricEtAl2019got}. In this approach, deterministic graphs are characterised by centred multivariate normal distributions whose covariance matrix is the generalized inverse of the graph Laplacian. A graph discrepancy is then defined by minimizing the squared Wasserstein distance of these normal distributions, for which there is a simple explicit expression, over all vertex matchings. For unequal cardinalities of the vertex sets one-to-many matchings are admitted, based on which the Laplacian of the larger graph is transformed to the lower dimensional space, see \textcite{PetricEtAl2022wasserstein}. An alternative approach that uses more general couplings of counting measures, i.e.\ allows for splitting of mass, was proposed in \textcite{DongSawin2020}. 
In all of these approaches, positions or other features of vertices are not taken into account. 

In the present article we propose three metrics based on (incomplete) matchings of vertices and their associated edges. 
This approach allows us to compare not only the edge structure, but also vertex attributes of graphs. Furthermore, we are able to compare (and match) graphs of (possibly) different size without the splitting of  ``masses'' of vertices or edges. This is achieved by inserting auxiliary vertices and adding a certain penalty to the distance. The handling of edges incident to auxiliary vertices differs for the metrics proposed in this article. We consider two different approaches on how different sizes of graphs (cardinalities of their vertex sets) are accounted for: \emph{absolute metrics}, where each vertex contributes up to a fixed amount, and \emph{relative metrics}, where this contribution is normalized with the total number of vertices. In this way we are able to capture the intuition that a missing edge in a large graph does barely contribute to the distance, while a missing edge in a small graph is more influential. Although additional characteristics such as differences in the global graph structure are by default not considered by our metrics, it is possible to capture such influences by modeling them explicitly through the vertex and edge features.

In the following we study examples to compare the metrics proposed in this article with the fused Gromov--Wasserstein (FGW) distance and the Wasserstein distance applied to graph signals (GOT approach) to emphasize main differences and shed light on cases where the use of such metrics is advantageous. For simplicity we limit ourselves to the comparison with the GOSPA2 metric (see Definition~\ref{def: graph-OSPA}b)) and comment on differences to the other metrics. 

\begin{example}[vertex features]\ 
\par \noindent
As it can be observed in Figure~\ref{Fig:IntroductionExampleFeatures} 
both the FGW distance and the metrics proposed in this article are able to take into account vertex attributes (here, the positions of the vertices) while the GOT distance is not able to detect the vertex differences. In many applications vertices contain additional information. For example, in many microbiological applications, cells are modeled as nodes while edges represent cell interactions. In such cases vertices capture positional and other cell information that is vital to the problem.
\begin{figure}[ht!]
    \centering
    \begin{subfigure}[t]{.4\linewidth}
    \centering
    \captionsetup{format=hang, justification=raggedright, singlelinecheck=off}
    \fbox{\includegraphics[width=0.8\linewidth]{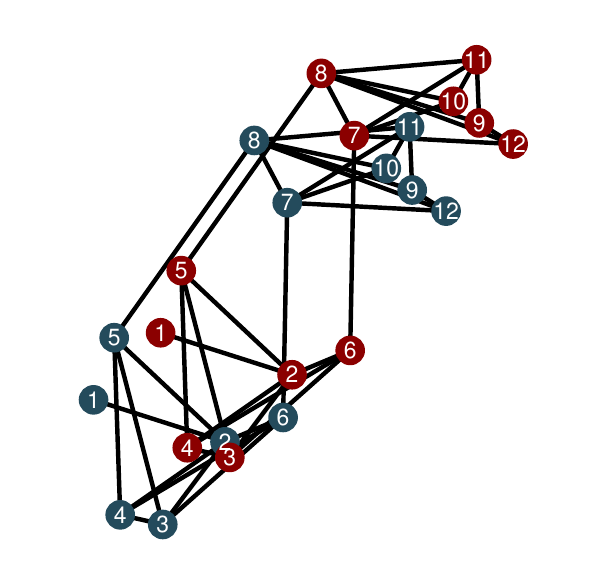}}
    \caption{$G_1$ and $G_1^{\textit{shift}}$\\
    $d_{GOT}(G_1,G_1^{\textit{shift}}) = 0$\\ $d_{FGW}(G_1,G_1^{\textit{shift}})= 0.12$\\ $d_{\mathbb{G},{R_2}}(G_1,G_1^{\textit{shift}})=0.424$}
    \end{subfigure}
    \caption{How do vertex features affect the metric? Translation of vertex positions by 0.424.}
  \label{Fig:IntroductionExampleFeatures}
\end{figure}
\end{example}

\begin{example}[splitting of mass]\ 
\par\noindent
Figure~\ref{Fig:IntroductionExampleMassSplit} illustrates the two different approaches of treating the ``masses'' of the vertices. For the FGW distance, the mass of any vertex in one graph is distributed freely to several vertices in the other graph. This can result in a splitting of mass for vertices and their incident edges. The approach taken in this article assigns mass $1$ to each vertex and (thereby) each edge  and suppresses the splitting of vertices and edges. 
This is essential when comparing fingerprint graphs for example, because the presence of doubled minutiae (and ridge lines) is an important distinctive feature for fingerprint identification. 
\begin{figure}[h!]
    \centering
    \begin{subfigure}[t]{.45\linewidth}
    \centering
    \captionsetup{format=hang,justification=raggedright, singlelinecheck=off}
    \includegraphics[width=.9\linewidth]{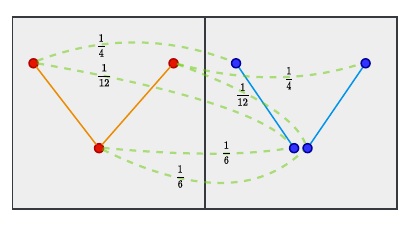}
    \caption{fused Gromov--Wasserstein distance:\\ mass is split freely.}
    \end{subfigure}
    \begin{subfigure}[t]{.45\linewidth}
    \centering
    \captionsetup{format=hang,justification=raggedright, singlelinecheck=off}
    \includegraphics[width=.9\linewidth]{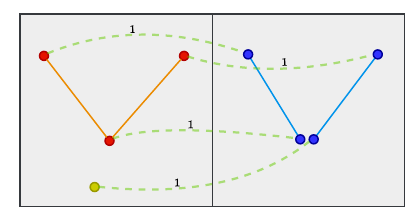}
    \caption{GOSPA2 metric: an auxiliary vertex is added and no mass is split.}
    \end{subfigure}
    \caption{To split or not to split?}
  \label{Fig:IntroductionExampleMassSplit}
\end{figure}
\end{example}

\begin{example}[global structure]\ \label{ex: Graph Structure}
\par \noindent
In Figure~\ref{Fig:IntroductionExampleNegative} the influence of the global structure is depicted. It can be seen that the GOT approach leads to a global comparison, 
i.e.\ edges that are important to the overall graph structure contribute more to the distance than edges that have a small effect on the global structure. 
The FGW distance in this example exhibits a similar behaviour which can be traced back to the underlying mass-split approach of this metric. Contrary to this, the metrics introduced in this paper are by default not able to detect this additional characteristic. 
Ignoring the influence on the global structure can be unfavorable for optimal routing problems in traffic networks for example.
\begin{figure}[ht!]
    \centering
    \begin{subfigure}[t]{0.3\linewidth}
    \centering
    \captionsetup{format=hang,justification=raggedright, singlelinecheck=off}
    \fbox{\includegraphics[width=.8\linewidth]{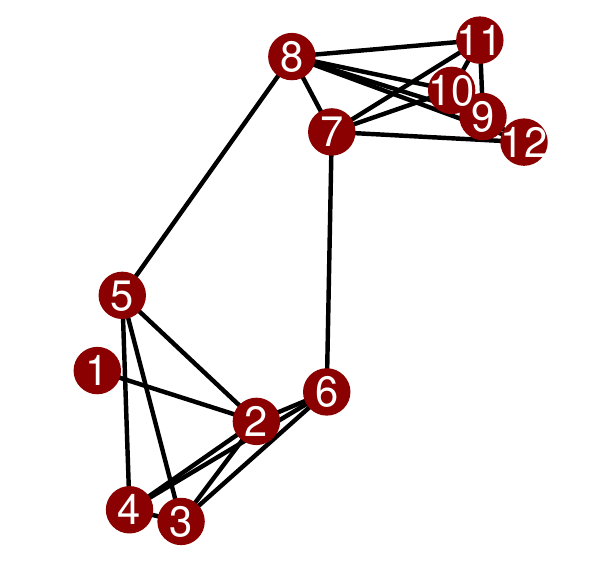}}
    \caption{$G_1$}
    \end{subfigure}
 \begin{subfigure}[t]{0.3\linewidth}
    \centering
    \captionsetup{format=hang,justification=raggedright, singlelinecheck=off}
    \fbox{\includegraphics[width=.8\linewidth]{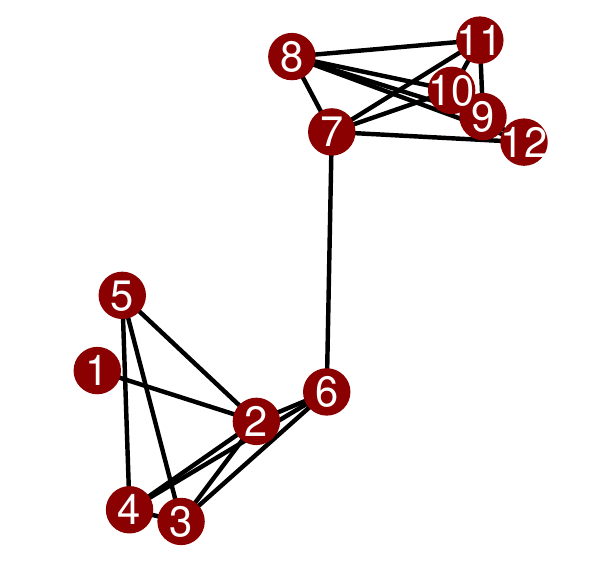}}
    \caption{$G_2$\\ $d_{GOT}(G_1,G_2) = 0.511$\\ $d_{FGW}(G_1,G_2)= 0.181$\\ $d_{\mathbb{G},{R_2}}(G_1,G_2)=0.008$}
    \end{subfigure}
 \begin{subfigure}[t]{0.3\linewidth}
    \centering
    \captionsetup{format=hang,justification=raggedright, singlelinecheck=off}
    \fbox{\includegraphics[width=.8\linewidth]{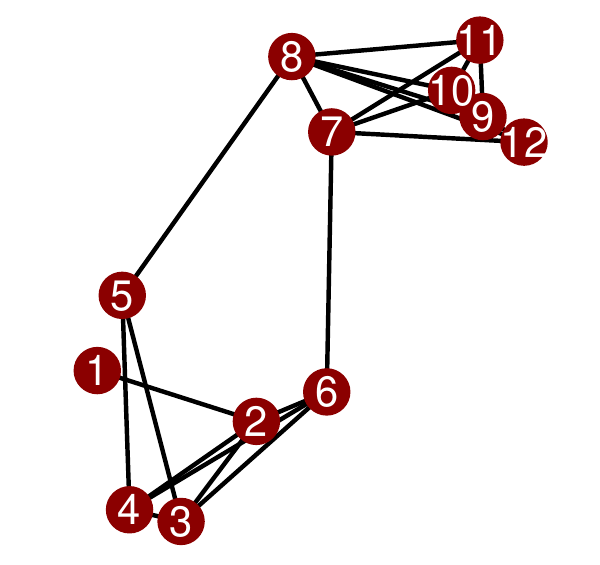}}
    \caption{$G_3$\\ $d_{GOT}(G_1,G_3) = 0.029$\\ $d_{FGW}(G_1,G_3)= 0.028$\\ $d_{\mathbb{G},{R_2}}(G_1,G_3)=0.008$}
    \end{subfigure}
    \caption{How does the global structure affect the metric? Omitting an edge that is vital for the overall connectivity versus omitting an edge that is not.}
  \label{Fig:IntroductionExampleNegative}
\end{figure}
\end{example}

\begin{example}[graph size]\ \label{ex: graph size}
\par \noindent
The impact of the graph size is demonstrated in Figure~\ref{Fig:IntroductionExampleRelative}, where two subgraphs of the graphs compared in Figure~\ref{Fig:IntroductionExampleNegative}
are studied. On a purely phenomenological level, spotting the difference between graphs $G_1$ and $G_3$ is much harder than spotting the difference between graphs $G_1'$ and $G_3'$. This observation is detected by all three metrics as they yield a higher dissimilarity between $G'_1$ and $G'_3$ than between $G_1$ and $G_3$. However, the change in GOT and FGW distance is small, while the GOSPA2 metric shows a significant increase in dissimilarity. This can be attributed to the nature of the influence of graph size on the metrics. 
While for the GOT and FGW distance the graph size has a (small) indirect influence that could be related to the fact that in a small graph, a single edge is more important for global structure than in a large graph, the influence of the graph size on the GOSPA2 metric is direct.

\begin{figure}[ht!]
    \centering
    \begin{subfigure}[t]{0.3\linewidth}
    \raggedleft
    \fbox{\includegraphics[width=.8\linewidth]{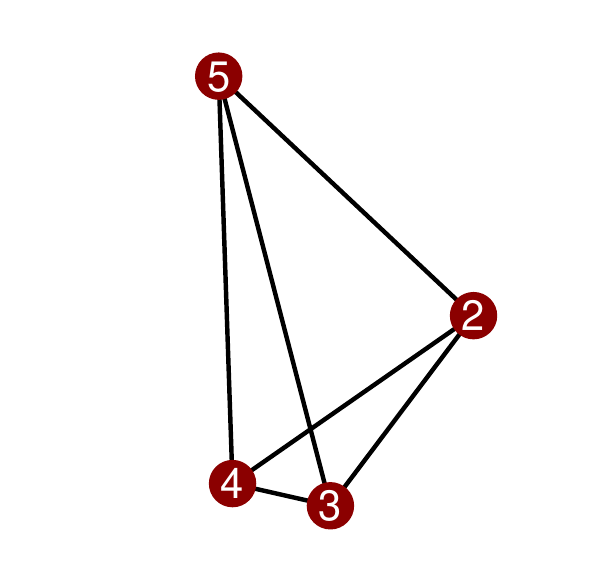}}
    \caption{$G'_1$}
    \end{subfigure}
    \hspace{8mm}
 \begin{subfigure}[t]{0.3\linewidth}
    \raggedright
    \hspace{5mm}
    \captionsetup{justification=raggedright, singlelinecheck=off}
    \fbox{\includegraphics[width=.8\linewidth]{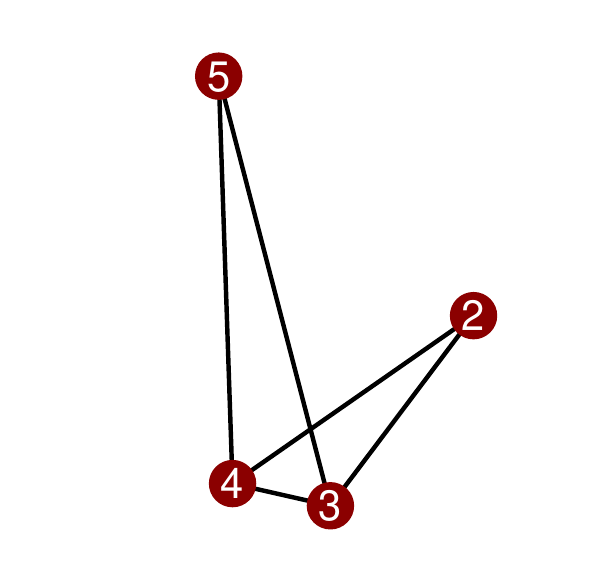}}
    \caption{$G'_3$\\ $d_{GOT}(G'_1,G'_3) = 0.043$\\ $d_{FGW}(G'_1,G'_3)= 0.041$\\ $d_{\mathbb{G},{R_2}}(G'_1,G'_3)=0.083$}
    \end{subfigure}
    \caption{How does the graph size affect the metric? Omitting an edge in a small graph versus omitting the same edge in a large graph (see Figure~\ref{Fig:IntroductionExampleNegative}).}
  \label{Fig:IntroductionExampleRelative}
\end{figure}
\end{example}

The plan of the paper is as follows. In Section~\ref{sec:theory} we define our new distance functions, prove that they are metrics and show a number of results about their properties and interrelations. In addition, we derive a new notion of random graph convergence and provide an application to point processes of vertices with independent edges. 
Section~\ref{sec:computation} gives computational considerations. In addition to a very slow exact method, we present an adaptation of the FAQ algortihm by \textcite{vogelstein2015} and a new auction algorithm with externalities. Section~\ref{sec:eval} evaluates these procedures in two simulated scenarios for various problem sizes. Finally, Section~\ref{sec:realdata} performs hypothesis tests for different structure of projection neurons across different areas (glomeruli) of the olfactory bulb in \textit{Drosophila} flies.

\section{Theory}
\label{sec:theory}

We first introduce a more formal framework for treating attributed undirected graphs without self-loops. It is convenient to consider pseudometric spaces of vertex and edge attributes in what follows, which allows to cover more general situations (see e.g.\ Remark~\ref{rem:why_pseudometrics}) but in general only leads to graph pseudometrics. For simplicity we continue to refer to these as graph metrics in the generic setting, where the attribute spaces are not further specified, although the metric property is strictly only true if the attribute spaces are metric as well.

Let $(\mcx,d_{\mcx})$ be a pseudometric space of vertex attributes, i.e.\ $d_{\mcx} \colon \mcx \times \mcx \to \R_{+}$ is symmetric and satisfies the triangle inequality and $d_{\mcx}(x,x) = 0$ for every $x \in \mcx$. Let $(\mcy,d_{\mcy})$ be another pseudometric space of edge attributes, with a distinguished element $y_0 \in \mcy$ signifying ``no edge''. We model a simple graph (i.e.\ undirected without loops) with vertex and edge attributes as a triple $([n], v, e)$ of a vertex set (identified with) $V := [n] := \{1,\ldots,n\}$ for some $n \in \N$, a map of vertex attributes $v \colon [n] \to \mcx$ and a symmetric map of edge attributes $e \colon [n]^2 \to \mcy$ satisfying $e(i,i) = y_0$ for every $i \in [n]$. The properties of $e$ reflect the fact that we consider undirected graphs without loops. We obtain the edge set as $E = \{ \{i,j\} \mvert e(i,j) \neq y_0 \}$, and set $v_i = v(i)$ and $e_{ij} = e(i,j)$ to lighten the notation. Furthermore write $\mathbb{G}$ for the set of all attributed simple graphs $([n],v,e)$ for given spaces $(\mcx,d_{\mcx})$ and $(\mcy,d_{\mcy})$.

We use the vertex set $V=[n]$ for notational convenience. The enumeration of vertex and edge attributes induced by this choice has no real importance. In particular, denoting by $S_n$ the set of permutations on $[n]$, we call two attributed simple graphs $([n],v,e)$, $([n],w,f)\in\mathbb{G}$ \emph{equal} and write $([n],v,e) =([n],w,f)$ if there exists a permutation $\pi\in S_n$ such that $v_i = w_{\pi(i)}$ and $e_{ij}=f_{\pi(i)\pi(j)}$ for all $i,j\in[n]$. In other words, our graphs are equivalence classes modulo relabeling of vertices, but come with a prespecified set of labels for convenience.


\subsection{Absolute metric}
\label{ssec:absolut_metric}

We build our metric from additive contributions of the individual vertex and edge distances based on an optimal (vertex) matching. This is similar to the idea used for the definition of the transport-transform (TT) metric between point patterns (finite counting measures), which was recently introduced in \textcite{MuellerEtAl2020}.

\begin{definition} \label{def:gttmetric}\ 
Let $(\mathcal{X}, d_{\mcx})$ and $(\mcy, d_{\mcy})$ be (pseudo-)metric spaces. For two attributed simple graphs $([m],v,e),([n],w,f)\in\mathbb{G}$
with $m \leq n$ (w.l.o.g.) we define the \textit{graph transport-transform (GTT) (pseudo-)metric} $d_{\mathbb{G},A}$ of order $p \geq 1$ and vertex penalty $C>0$ by
    \begin{align*}
    &d_{\mathbb{G},A}(([m],v,e),([n],w,f))\\
    &= \min\limits_{I,\pi} \, \biggl[ (m+n-2\abs{I}) C^p + \sum\limits_{i \in I} d_{\mcx}(v_i,w_{\pi(i)})^p\\
    &\hspace*{4mm} + \frac{1}{2} \sum\limits_{(i,i')\in I^2} d_{\mcy}(e_{i i'},f_{\pi(i) \pi(i')})^p + \frac{1}{2} \sum_{(i,i')\in [m]^2\setminus I^2 } d_{\mcy}(e_{ii'},y_0)^p + \frac{1}{2} \sum_{(j,j')\in [n]^2\setminus \pi(I)^2 } d_{\mcy}(y_0,f_{jj'})^p
    \biggr]^{1/p},
    \end{align*}
where the minimum is taken over $I\subset [m]$ and $\pi\in S_n$ and $\abs{I}$ denotes the cardinality of $I$.
\end{definition}

We defer the justification for calling $d_{\mathbb{G}}$ a (pseudo-)metric to Theorem~\ref{thm: graph-TT is metric} below. First we show that we may express Definition~\ref{def:gttmetric} equivalently in terms of an optimal matching problem between the two graphs enlarged to size $m+n$. This equivalence is also used for computing the metric in~Section~\ref{sec:computation}. 
\begin{proposition}\label{prop: GTT as optimal matching (filled to n+m)}
Let $([m],v,e),([n],w,f)\in\mathbb{G}$.  Consider auxiliary vertices $x_*\notin \mcx$ and extend the (pseudo-)metric by $d_{\mcx}(x,x_*)^p = d_{\mcx}(x_*,x)^p = C^p$ for all $x\in\mcx$. Enlarge each graph to size $m+n$ by setting $v_i = x_*$ for $m+1\leq i\leq m+n$ and $e_{ii'} = y_0$ for $m+1\leq i \vee i' \leq m+n$ and analogously $w_i = x_*$ for $n+1\leq i\leq m+n$ and $f_{ii'} = y_0$ for $n+1\leq i \vee i' \leq m+n$. Then
  \begin{align*}
   &d_{\mathbb{G},A}(([m],v,e),([n],w,f))= \!\min\limits_{\pi\in S_{m+n}} \biggl[ \sum_{i\in[m+n]}\! d_{\mcx}(v_i,w_{\pi(i)})^p +\frac{1}{2} \sum\limits_{(i,i')\in [m+n]^2} d_{\mcy}(e_{ii'},f_{\pi(i)\pi(i')})^p\biggr]^{1/p}.
  \end{align*}
\end{proposition}
Due to their technical nature, we defer this and most other proofs to the appendix. Instead we provide here a constructive approach to the GTT metric that relates to the proposition.

\begin{figure}[t!]
\hspace{3mm}
  \includegraphics[width=1\textwidth]{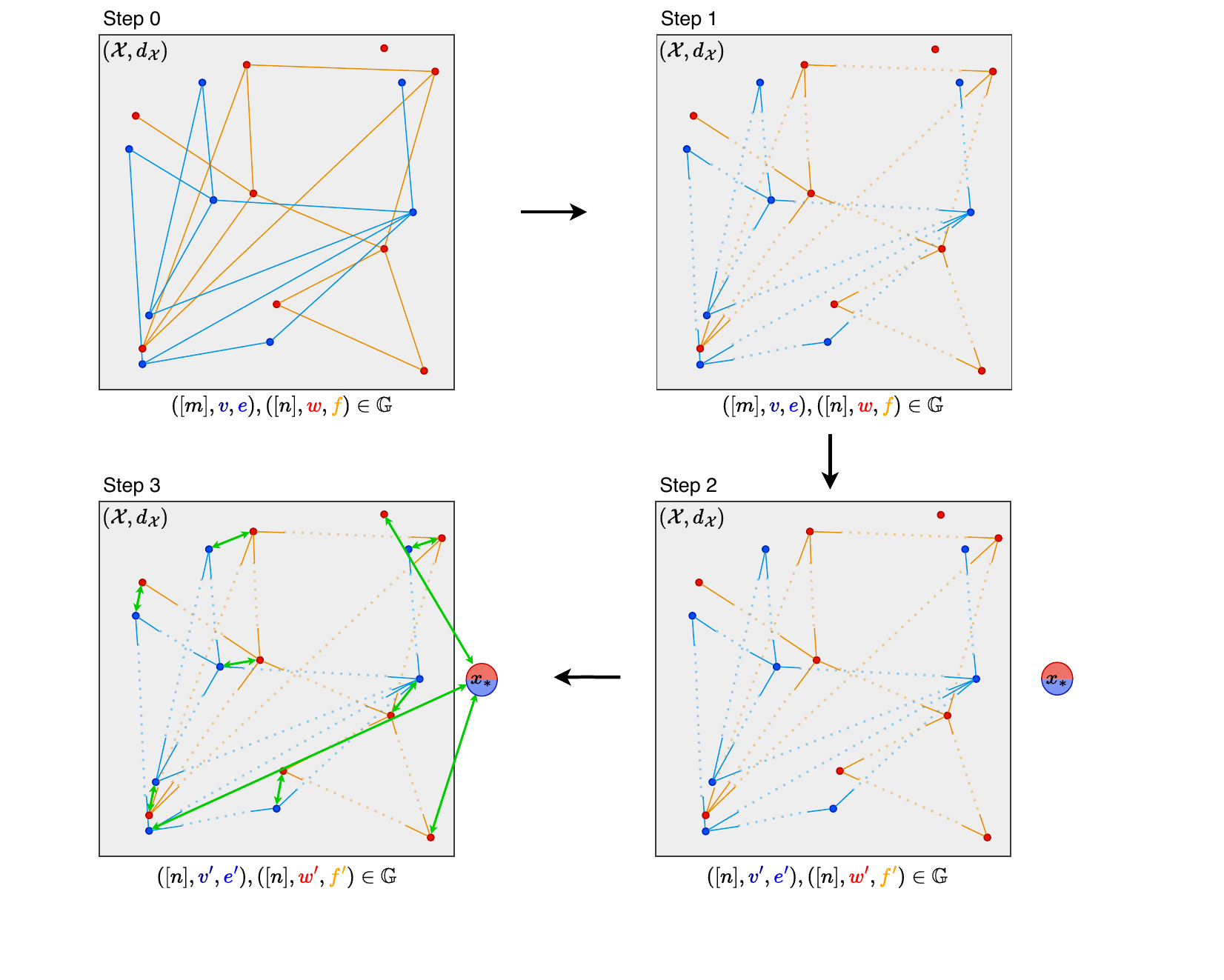}
  \vspace*{-12mm}
  
  \caption{Metric construction with reduced dependencies: Step 1: Split the edges into half-edges. Step 2: Add auxiliary vertices at $x_*$ (without edges) to both graphs. Step 3: Perform an optimal matching between the vertices taking into account their half-edge structure.}
  \label{Fig: Intuition Graph-TT}
\end{figure}

In Figure~\ref{Fig: Intuition Graph-TT} the construction of the metric in terms of the optimal matching in Proposition~\ref{prop: GTT as optimal matching (filled to n+m)} is illustrated. Having two graphs of possibly different sizes we start by replacing each edge by two half-edges. We add auxiliary vertices to both graphs, increasing their size to $m+n$. As no further edges are added, the auxiliary vertices are isolated. We can now perform an optimal matching where the cost of a vertex matching depends on the vertex location and the (half-)edge structure of the considered vertices. The enlargement of both graphs to size $m+n$ yields that two vertices can both be matched to an auxiliary vertex instead of being matched with each other. Such vertices correspond to the ones whose indices are not contained in the set $I$ from Definition~\ref{def:gttmetric}. Under the optimal $\pi$ in Proposition~\ref{prop: GTT as optimal matching (filled to n+m)}, the cost of a single vertex matching is naturally bounded by two times the cost of matching with an auxiliary vertex, e.g.\ for $p=1$, $\mcy := \{0,1\}$ and $d_\mcy(e,f) = \mathbbm{1}\{e \neq f\}$, by two times the cost $C$ of comparing to an auxiliary vertex plus the number of half-edges of both vertices (as auxiliary vertices are isolated).

Using Proposition~\ref{prop: GTT as optimal matching (filled to n+m)}, we show (in the appendix) that $d_{\mathbb{G},A}$ is a (pseudo-)metric.
\begin{theorem}\label{thm: graph-TT is metric}
The map $d_{\mathbb{G},A}$ is a pseudometric on $\mathbb{G}$. If $d_{\mcx}$ and $d_{\mcy}$ are metrics, then $d_{\mathbb{G},A}$ is also a metric.
\end{theorem}

The next two results relate the GTT metric to previous metrics.
\begin{proposition}\label{prop: GTT is special case fo GED }
For $p=1$ the GTT metric is a natural special case of a graph edit metric proposed by \textcite{bunke1983}.
\end{proposition}

\begin{proposition} \label{prop:gtt_vs_tt}
  Suppose that $p \geq 1$ and $d_{\mcy} \equiv 0$ with $C_{\mcy}=0$. For $([m],v,e),([n],w,f)\in\mathbb{G}$, we have
  \begin{equation*}
      d_{\mathbb{G},A}(([m],v,e),([n],w,f)) = \tau(\xi,\eta),
  \end{equation*}
  where $\tau$ denotes the TT metric between finite counting measures with penalty $C_i \geq C_{\mcx}$ and ${\xi = \sum_{i=1}^m \delta_{v_i}}$, $\eta = \sum_{j=1}^n \delta_{w_j}$ are the counting measures corresponding to the node sets $([m],v)$ and $([n],w)$, respectively.
\end{proposition}

\begin{remark} \label{rem:why_pseudometrics}
  The fact that we allow $d_{\mcx}, d_{\mcy}$ to be pseudometrics rather than metrics allows us to ignore 
  the distinction of elements in $\mcx$ and $\mcy$ or even the distinction of whole components of multivariate attributes when computing graph distances. In particular we obtain for $d_{\mcy} \equiv 0$
  once more the TT metric between the vertices of the graphs and for $d_{\mcx} \equiv 0$ a purely edge-based metric, e.g.\ for $d_{\mcy}(e_{ij}, f_{ij}) = \abs{d_{\mcx}(v_i,v_j)-d_{\mcx}(w_i,w_j)}$ the Gromov--Monge pseudometric, a variant of the Gromov--Wasserstein pseudometric that prohibits splitting of mass, see~\textcite{MemoliNeedham2022}.
\end{remark}

\subsection{Relative metrics}
As illustrated in the introduction there are contexts where normalization by the total number of points (and edges) is more appropriate (see Example~\ref{ex: graph size}). Among other things this leads to a weaker (pseudo-)metric that still metrizes weak convergence of spatial graphs, see Theorem~\ref{thm: Metrization of weak convergence} below. For differing cardinalities it seems natural to divide by the maximal number of points (and edges), as illustrated e.g.\ by the extension of the smaller to the larger graph when adding dummy points. We first note that this is not readily possible based on Definition~\ref{def:gttmetric}. 
\begin{remark}
  Suppose that $\mcx$ contains two points $x_1,x_2$ at distance $d_{\mcx}(x_1,x_2) \geq 2 C$ for $C>0$. Let $m=n=1$ and $l=2$ and choose $v_1 =x_1$, $w_1 = x_2$, $u_i= x_i$ for $i=1,2$ as well as $e\equiv y_0$, $f\equiv y_0$, $g\equiv y_0$. It is easily checked that 
  \begin{align*}
    \frac{1}{\max\{m,n\}} &d_{\mathbb{G},A}(([m],v,e),([n],w,f)) = 2 C \\
    &> \frac12 C + \frac12 C \\
    &= \frac{1}{\max\{m,l\}} d_{\mathbb{G},A}(([m],v,e),([l],u,g))
    + \frac{1}{\max\{l,n\}} d_{\mathbb{G},A}(([l],u,g),([n],w,f)). 
  \end{align*}
  This example shows that simple dividing of the GTT metric by the maximal number of vertices will in general not result in a metric again. Since our example did not involve edges, this implies the same statement for the original (point pattern) TT-metric. In particular it follows with Proposition~\ref{prop:gtt_vs_tt} that the map $\overline{\tau}$ proposed in \textcite{MuellerEtAl2020} is not a metric on the space $\mfn_{\mathrm{fin}}$ of finite counting measures
  if $C$ is too small with respect to the diameter of $\mcx$. The above counterexample remains valid for order $p \in [1,2)$.
\end{remark}

A convenient approach is obtained by enforcing the maximal number of vertex matches between the graphs, which does result in metric as long as we suitably bound the underlying metrics between vertices and edges. Ignoring the edges, this would result in the Optimal SubPattern Assignment (OSPA) metric introduced in \textcite{SXia2008} and popularized in \textcite{SVoVo2008}, see Proposition~\ref{prop:gospa_vs_ospa} below. Taking edges into account, there are (at least) two ways how we may penalize the (non-)edges of unmatched vertices in the larger graph.

A cruder option leading to a stronger metric is obtained by adding a fixed penalty $C_{\mcy}$ (in addition to the vertex penalty $C_{\mcx}$) for each other vertex that an unmatched vertex could have an edge with, i.e.\ regardless of the actual edge structure of the unmatched vertex. We refer to the resulting map as \emph{GOSPA1 metric}. 
From a modelling point of view this may sometimes be the right choice if the focus lies on penalizing the loss of information incurred by missing individual vertices and their edge information. 

A refined option leading to a weaker metric is obtained by penalizing only the actual edges of the unmatched vertices according to the specified cost incurred for matching an edge with a non-edge. 
We refer to the resulting map as \emph{GOSPA2 metric}. 
In this approach a lone unmatched vertex without any edges will be penalized much more lightly than an ``influential'' unmatched vertex with a lot of edges.

We refer to these metrics jointly as \emph{Graph OSPA (GOSPA) metrics} and 
again defer justification of the term  (pseudo-)metric to later, see Theorem~\ref{thm: graph-OSPA are metric}.

\begin{definition} \label{def: graph-OSPA}
Let $(\mathcal{X}, d_{\mcx})$ and $(\mcy, d_{\mcy})$ be (pseudo-)metric spaces satisfying $\mathrm{diam}(\mcx) \leq C_{\mcx}$ and $\mathrm{diam}(\mcy) \leq C_{\mcy}$. For $([m],v,e),([n],w,f)\in\mathbb{G}$ with $n\geq \max\{m,1\}$ (w.l.o.g.), define
\begin{itemize}
    \item[a)] the \textit{GOSPA1 (pseudo-)metric} $d_{\mathbb{G},R_1}$ of order $p \geq 1$ and penalty $C_1^p \geq C_{\mcx}^p + \frac12 C_{\mcy}^p$ by 
      \begin{align*}
  d_{\mathbb{G},R_1}(&([m],v,e),([n],w,f))  \\[1mm] 
  &=\frac{1}{n^{1/p}}\min\limits_{\pi \in S_n} \biggl[ (n-m) C_1^p + \sum_{i\in [m]} d_{\mcx}(v_i, w_{\pi(i)})^p\\
  &\hspace*{14mm}+ \frac12 \frac{1}{n-1} \sum_{i\in [m]}  \biggl(\sum\limits_{i'\in [m] } d_{\mcy}(e_{ii'},f_{\pi(i)\pi(i')})^p +  (n-m) C_{\mcy}^p \biggr) \biggr]^{1/p}.  
       \end{align*}
    \item[b)] the \textit{GOSPA2 (pseudo-)metric} $d_{\mathbb{G},R_2}$ of order $p \geq 1$ and penalty $C_2^p \geq C_{\mcx}^p + C_{\mcy}^p$, additionally requiring for $p>1$ that ${d_{\mcy}(y_1,y_2) \leq \max\{d_{\mcy}(y_1,y_0),d_{\mcy}(y_0,y_2)\}}$ for all $y_1,y_2\in\mcy$, by 
    \begin{align*}
  d_{\mathbb{G},R_2}(&([m],v,e),([n],w,f)) \\[1mm] 
  &= \frac{1}{n^{1/p}}\min\limits_{\pi \in S_n} \biggl[ (n-m) C_2^p + \sum_{i\in[m]} d_{\mcx}(v_i, w_{\pi(i)})^p  \\
&\hspace*{10mm} + \frac12  \frac{1}{n-1} \biggl(\sum_{(i,i')\in[m]^2}   d_{\mcy}(e_{ii'},f_{\pi(i)\pi(i')})^p +  \sum_{(i,i')\in [n]^2\setminus[m]^2} d_{\mcy}(y_0,f_{\pi(i)\pi(i')})^p \biggr)\biggr]^{1/p}.  
  \end{align*}
\end{itemize}
  We set $0/0 := 0$. So for $m \leq n=1$ we obtain $d_{\mathbb{G},R_i}(([m],v,e),([n],w,f)) = \min\limits_{\pi \in S_n} \bigl[ (1-m) C_i^p+ \sum_{i=1}^m  d_{\mcx}(v_i, w_{\pi(i)})^p  \bigr]^{1/p}$, $i=1,2$. For $m=n=0$, we set $d_{\mathbb{G},R_i}(([0],v,e),([0],w,f)) := 0$.
\end{definition}

\begin{remark}
It may seem desirable to relax the condition on $C_2$ 
to $C_2\geq C_{\mcx} +\frac12 C_{\mcy}$ or even $C_2\geq C_{\mcx}$. However, this is not possible as the following counterexample shows:\\ 
Set $p=1$ and $C_2= C_{\mcx} + \frac{1}{2} C_{\mcy}$ and assume there exist $x_1,x_2,x_3\in \mcx$ and $y_1\in\mcy$ such that $d_{\mcx}(x_3,x_i) = C_{\mcx}$ for $i=2,3$, as well as $d_{\mcy}(y_0,y_1) = C_{\mcy}$. Consider the three graphs $([1],v,e),([2],w,f),([3],u,g)\in\mathbb{G}$ where $v_1 = x_3$, $e_{11}=y_0$ and $w_i = x_i$ for $i\in[2]$, $ f_{12}=f_{21} = y_1$ and $f_{ij} = y_0$ otherwise as well as 
$u_i = x_i$ for $i\in[3]$, $g_{12}=g_{21} = y_1$ and $g_{ij} = y_0$ otherwise. Then 
\begin{align*}
    d_{\mathbb{G},R_2}(([m],v,e),([n],w,f)) &= \frac12 ( C_2 + C_{\mcx} + d_{\mcy}(y_0,y_1)) = C_2 + \frac{1}{4}C_{\mcy} \\
    &> C_2 + \frac{1}{6}C_{\mcy}= \frac{1}{3} \biggl(2C_2 + \frac12 d_{\mcy}(y_0,y_1)\biggr) + \frac{1}{3}C_2\\ &=d_{\mathbb{G},R_2}(([m],v,e),([l],u,g)) +d_{\mathbb{G},R_2}(([l],u,g),([n],w,f)),
\end{align*}
which contradicts the triangle inequality.
\end{remark}

\begin{remark}
The additional condition on the underlying edge metric $d_{\mcy}$ is only required for $p > 1$ as a consequence of the proof strategy. It is fulfilled whenever two edges (with possibly different attributes) are closer together than the maximum of the distance of each of these edges to the no-edge element $y_0$, for example for $d_{\mcy}(e_{ij},f_{ij}) = \abs{e_{ij}-f_{ij}}$ for $e_{ij},f_{ij}\in \mcy = \R_+$, $y_0=0$.
\end{remark}


The GOSPA metrics are equivalent to solving an optimal matching problem of two graphs of the same size. Again, this correspondence is used in Section~\ref{sec:computation} to calculate the GOSPA distances of two graphs. 
\begin{proposition} \label{prop: proof of GOSPA1 and GOSPA2 filling up to same size does not change} 
Let $([m],v,e),([n],w,f)\in\mathbb{G}$. Consider auxiliary vertices $x_*\notin \mcx$ and auxiliary edges $y_*\notin \mcy$.
\begin{itemize}
    \item[a)] For GOSPA1 extend the attribute (pseudo-)metrics by $d_{\mcx}(x,x_*)^p = d_{\mcx}(x_*,x)^p = C_1^p-\frac{1}{2}C_{\mcy}^p$ for all $x\in\mcx$ and $d_{\mcy}(y,y_*) = d_{\mcy}(y_*,y) = C_{\mcy}$ for all $y\in\mcy$, respectively. Enlarge the smaller graph (w.l.o.g.) $([m],v,e)$ to the size of the larger graph by setting $v_i = x_*$ and $e_{ii} = y_0$ for $m+1\leq i\leq n$ and $e_{ii'} = y_*$ for $m+1\leq i \vee i' \leq n$ with $i\neq i'$.  Then
    \begin{align*}
    d_{\mathbb{G},R_1}(&([m],v,e),([n],w,f)) \\
    &=\frac{1}{n^{1/p}}\min\limits_{\pi \in S_n} \biggl[\sum_{i\in[n]} d_{\mcx}(v_i, w_{\pi(i)})^p+ \frac12 \frac{1}{n-1}\sum\limits_{(i,i')\in[n]^2} d_{\mcy}(e_{ii'},f_{\pi(i)\pi(i')})^p \biggr]^{1/p}\\
         &= d_{\mathbb{G},R_1}(([n],v,e),([n],w,f)).
    \end{align*}
    \item[b)] For GOSPA2 extend the vertex attribute (pseudo-)metric by $d_{\mcx}(x,x_*)^p = d_{\mcx}(x_*,x)^p = C_2^p$ for all $x\in\mcx$. Enlarge the smaller graph (w.l.o.g.) $([m],v,e)$ to the size of the larger graph by setting $v_i = x_*$ for $m+1\leq i\leq n$ and $e_{ii'} = y_0$ for $m+1\leq i \vee i' \leq n$. Then
    \begin{align*}
    d_{\mathbb{G},R_2}(&([m],v,e),([n],w,f)) \\
    &=\frac{1}{n^{1/p}}\min\limits_{\pi \in S_n} \biggl[\sum_{i\in [n]} d_{\mcx}(v_i, w_{\pi(i)})^p+ \frac12 \frac{1}{n-1}\sum\limits_{(i,i')\in[n]^2} d_{\mcy}(e_{ii'},f_{\pi(i)\pi(i')})^p \biggr]^{1/p}\\
         &= d_{\mathbb{G},R_2}(([n],v,e),([n],w,f)).
\end{align*}
\end{itemize}
\end{proposition}
As usual the proof can be found in the appendix. In what follows we present a constructive approach to the GOSPA metrics that relates to the proposition.

\begin{figure}[t!]
\hspace{-7mm}
    \includegraphics[width=1\textwidth]{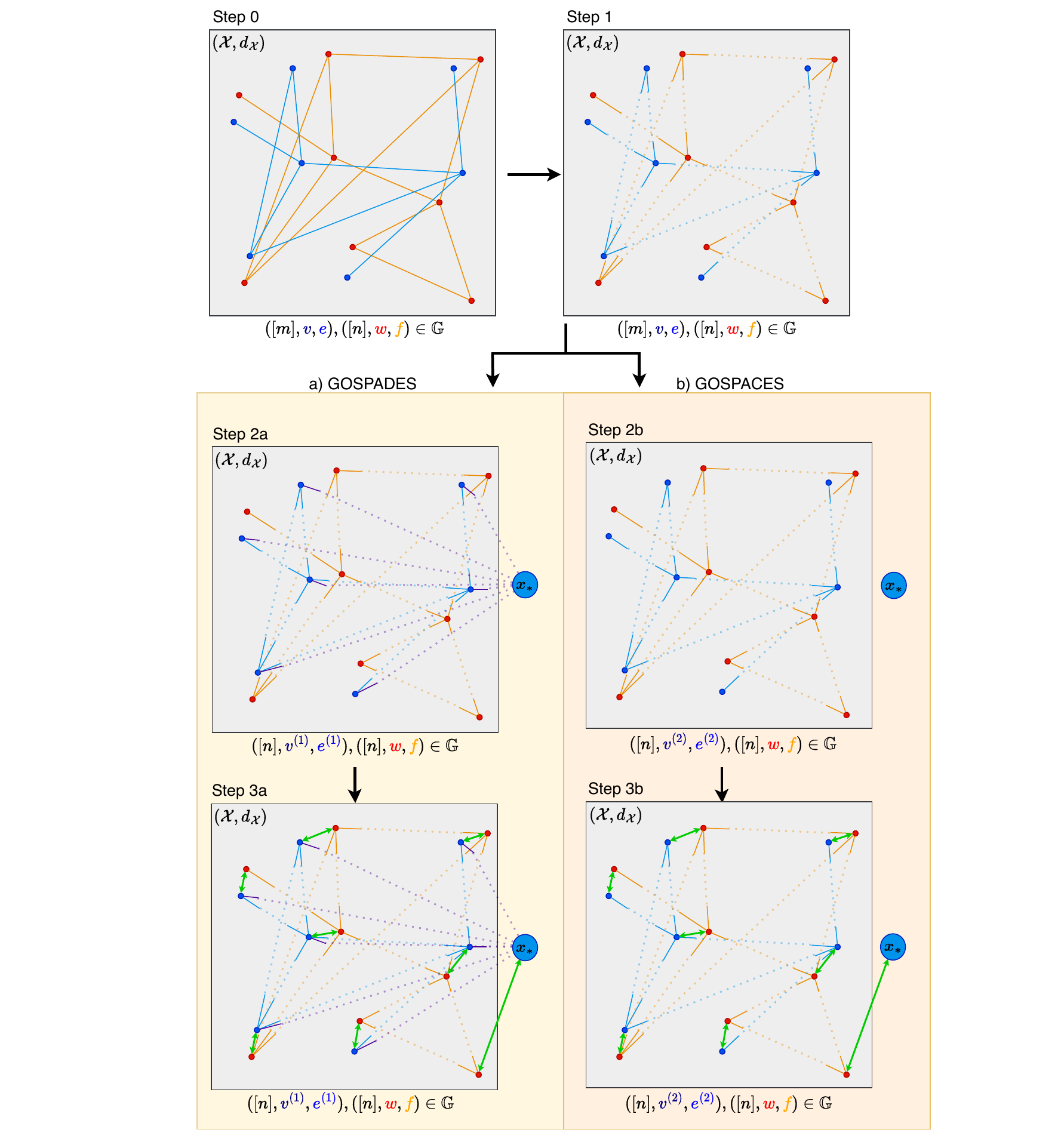}
    \caption{Metric construction for GOSPA metrics: Step 1: Split the edges into half-edges. a) GOSPA1: Step 2a: Add auxiliary vertices at $x_*$ and connect each auxiliary vertex to all other vertices of the graph using auxiliary edges. Step 3a: Perform an optimal matching between the vertices taking into account their half-edge structure. b) GOSPA2: Step 2b: Add auxiliary vertices at $x_*$ (without edges). Step 3b: Perform an optimal matching between the vertices taking into account their (half-)edge structure.}
  \label{Fig: Intuition Graph-OSPA}
\end{figure}

In Figure~\ref{Fig: Intuition Graph-OSPA} the enlargement construction of the GOSPA1 and the GOSPA2 metric is illustrated. Starting with two graphs of possibly different sizes, we divide each edge into two half-edges. Then we fill up the smaller graph using auxiliary elements.
This enlargement differs for the GOSPA1 and the GOSPA2 metric. In the case of the GOSPA1 metric we fill up the graph $([m],v,e)$ to a graph $([n],v,e)$ by adding $n-m$ auxiliary vertices (at cost $C_1^p-\frac12 C_{\mcy}^p$) and connecting each of these vertices with every other vertex using auxiliary edges (see Step~2a). 
We remark that the properties of the auxiliary edge (i.e. $d_{\mcy}(y,y_*) = C_{\mcy}$ for all $y\in\mcy$) correspond to the disregarding of the edge-structure of auxiliary vertices in the GOSPA1 metric, since every matching of a vertex to an auxiliary vertex is penalized with constant cost $C_1^p\geq C_{\mcx}^p +\frac12 C_{\mcy}^p$ independent of the (half-)edge structure of this (non-auxiliary) vertex.\\
For the GOSPA2 metric we fill up the graph $([m],v,e)$ to a graph $([n],v,e)$ by again adding $n-m$ auxiliary vertices (at cost $C_2^p$) but then not adding any edge connecting these auxiliary vertices (see Step 2b). 
In this way we take into account the edge structure of vertices matched with auxiliary vertices, since for each such match there is a constant cost $C_2^p\geq C_{\mcx}^p + C_{\mcy}^p$ and an additional cost proportional to the number of (half-)edges this (non-auxiliary) vertex has.\\
Both enlargements yield two graphs of the same size. Hence, we are able to perform an optimal matching between the vertices of the two graphs. The cost of this optimal matching depends on the way we filled up and corresponds to the penalties in the definition of the GOSPA metrics.\\ 

Using Proposition \ref{prop: proof of GOSPA1 and GOSPA2 filling up to same size does not change}, we show (in the appendix) that $d_{\mathbb{G},R_i}$ are (pseudo-)metrics.
\begin{theorem}\label{thm: graph-OSPA are metric}
The maps $d_{\mathbb{G},R_1}$ and $d_{\mathbb{G},R_2}$ are pseudometrics on $\mathbb{G}$. If $d_{\mcx}$ and $d_{\mcy}$ are metrics, then $d_{\mathbb{G},R_1}$ and $d_{\mathbb{G},R_2}$ are also metrics.
\end{theorem}

The following is a direct consequence of the definitions.
\begin{proposition} \label{prop:gospa_vs_ospa}
  Suppose that $p \geq 1$ and $d_{\mcy} \equiv 0$ with $C_{\mcy}=0$. Then we have, for $([m],v,e),([n],w,f)\in\mathbb{G}$,
  \begin{equation*}
      d_{\mathbb{G},R_i}(([m],v,e),([n],w,f)) = \varrho(\xi,\eta), \quad \text{$i=1,2$},
  \end{equation*}
  where $\varrho$ denotes the OSPA metric between finite counting measures with penalty $C_i \geq C_{\mcx}$ and $\xi = \sum_{i=1}^m \delta_{v_i}$, $\eta = \sum_{j=1}^n \delta_{w_j}$ are the counting measures corresponding to the node sets $([m],v),([n],w)$, respectively. 
\end{proposition}

\subsection{Bounds and interrelations}\label{ssec: further metric results}

The following bounds are simple consequences of the definitions. We list them here for easy reference.

\begin{proposition}[Bounds]\label{prop: bound of graph metrics}
Let $(\mathcal{X}, d_{\mcx})$ and $(\mcy, d_{\mcy})$ be (pseudo-)metric spaces such that $\mathrm{diam}(\mcy) \leq C_{\mcy}$. Choose $p\geq 1$.
  \begin{itemize}
      \item[a)] Assume further that $\mathrm{diam}(\mcx) \leq C_{\mcx}$ and consider $C_1^p\geq C_{\mcx}^p + \frac12 C_{\mcy}^p$ and $C_2^p\geq C_{\mcx}^p + C_{\mcy}^p$. Then, for $i= 1,2$ and any $([m],v,e),([n],w,f)\in \mathbb{G}$
      $$d_{\mathbb{G},R_i}([m],v,e),([n],w,f)\leq C_i.$$
      \item[b)] For any $C>0$ and any $([m],v,e),([n],w,f)\in \mathbb{G}$
      $$d_{\mathbb{G},A}([m],v,e),([n],w,f)\leq \bigl((n+m)C^p + (E +F) C_{\mcy}^p\bigr)^{1/p},$$
      where $E$ and $F$ denote the number of edges in $([m],v,e)$ and $([n],w,f)$, respectively.
  \end{itemize}
\end{proposition}

The following comparisons of the metrics are straightforwardly obtained from Propositions~\ref{prop: GTT as optimal matching (filled to n+m)} and~\ref{prop: proof of GOSPA1 and GOSPA2 filling up to same size does not change}.

\begin{proposition}[Relation of metrics]
Let $(\mathcal{X}, d_{\mcx})$ and $(\mcy, d_{\mcy})$ be (pseudo-)metric spaces satisfying $\mathrm{diam}(\mcx) \leq C_{\mcx}$ and $\mathrm{diam}(\mcy) \leq C_{\mcy}$. Choose $p\geq 1$.
\begin{itemize}
    \item[a)] For any $([n],v,e),([n],w,f)\in \mathbb{G}$ $$d_{\mathbb{G},R_1}([n],v,e),([n],w,f)= d_{\mathbb{G},R_2}([n],v,e),([n],w,f).$$
    \item[b)] For $C_1^p \leq C_2^p - \frac{1}{2}C_{\mcy}^p$ and any $([m],v,e),([n],w,f)\in \mathbb{G}$ $$d_{\mathbb{G},R_1}([m],v,e),([n],w,f)\leq d_{\mathbb{G},R_2}([m],v,e),([n],w,f).$$ 
    \item[c)] For $C_2^p \leq C_1^p - \frac{1}{2}C_{\mcy}^p$ and any $([m],v,e),([n],w,f)\in \mathbb{G}$ $$d_{\mathbb{G},R_1}([m],v,e),([n],w,f)\geq d_{\mathbb{G},R_2}([m],v,e),([n],w,f).$$ 
    \item[d)] For $C^p\geq C_i^p$, $i=1,2$, and any $([m],v,e),([n],w,f)\in \mathbb{G}$
    $$d_{\mathbb{G},A}([m],v,e),([n],w,f)\geq d_{\mathbb{G},R_i}([m],v,e),([n],w,f).$$
\end{itemize}
\end{proposition}
\begin{remark}[Differences between GTT and GOSPA metrics]
Besides the scaling in the GOSPA metrics, we obtain further differences between the GTT metric and the GOSPA metrics. For instance, unlike the GTT metric, the GOSPA metrics require that the number of (non-auxiliary) vertex matchings is maximal.
Furthermore, the GTT metric can be defined for any underlying vertex and edge metrics and is in general unbounded. In contrast to this the GOSPA metrics require bounded underlying metrics which yields bounded graph metrics (see Proposition~\ref{prop: bound of graph metrics}). 
\end{remark}

\subsection{Convergence of random graphs}\label{ssec: convergence considerations}

We show that the (relative) graph metrics introduced metrize weak convergence of graphs in a sense that we are going to make precise now. 

Assume that $(\mcx, d_{\mcx})$, $(\mcy, d_{\mcy})$ are compact metric spaces. We first represent $([m], v, e)$ equivalently as a pair of finite counting measures $(\xi,\sigma)$. The vertex information of the graph is expressed by the measure $\xi$ living on an extension $\mcx'$ of $\mcx$ that serves for identifying vertices uniquely. The edge information is expressed by the measure $\sigma$ living on an extension $\mcy'$ of $\mcy$ that serves for keeping the edge attributes continuously associated with their pairs of vertices. We equip $\mcx'$ and $\mcy'$ with appropriate metrics such that they are compact metric spaces. Denote by $\mfn(\mcx')$ and $\mfn(\mcy')$ the spaces of finite counting measures on $\mcx'$ and $\mcy'$ respectively. The constructions of $\xi \in \mfn(\mcx')$ and $\sigma \in \mfn(\mcy')$ are as follows.

Consider a graph $([m],v,e)\in\mathbb{G}$.
If $v$ is injective, i.e.\ if the vertex attributes are unique, set $(\mcx',d_{\mcx'}) = (\mcx,d_{\mcx})$ and $v' = v$. This is the case, e.g., if the vertices are obtained as realization of a simple point process in $\mcx \subset \R^d$ (altered on a null set to be strictly without multipoints).
Otherwise, if $v$ is not injective, use $\mcx' = \mcx \times [0,1]$ and set $d_{\mcx'}((x_1,u_1),(x_2,u_2)) = d_{\mcx}(x_1,x_2) + \abs{u_1-u_2}$. Let $v' \colon [m] \to \mcx'$, $v'(i) = v'_i = (v_i,1/i)$ be the new map of unique vertex attributes. Define then $\xi = \sum_{i=1}^{m} \delta_{v'_i} \in \mfn(\mcx')$.

Next, we write ${\mcx'}^{\atwo} = \{\{x'_1,x'_2 \} \mvert x'_1, x'_2 \in\mcx', x'_1 \neq x'_2 \}$, on which we define a metric by $d_{{\mcx'}^{\atwo}}\big(\{x'_1, x'_2\}, \{\tx'_1, \tx'_2\} \big) = \min\{d_{\mcx'}(x'_1,\tx'_1) + d_{\mcx'}(x'_2,\tx'_2), d_{\mcx'}(x'_1,\tx'_2) + d_{\mcx'}(x'_2,\tx'_1) \}$.

If $e_{ii'} \neq y_0$ allows to recover $\{v'_i,v'_{i'}\}$ for all $i,i' \in [m]^2$ continuously, i.e.\ if there exists a continuous function $h:\mcy \setminus \{y_0\} \to {\mcx'}^{\atwo}$ that maps an edge attribute~$e_{ii'}\neq y_0$ to its unique vertex set~$\{v'_i,v'_{i'}\}$, set $(\mcy',d_{\mcy'}) = (\mcy,d_{\mcy})$ and $e' = e$. Otherwise we equip all $e_{ii'} \neq y_0$ with the additional attribute $\{v'_i,v'_{i'}\}$ (or some simpler information that uniquely identifies the two vertices, such as $\{i,i'\}$ if $v'_i, v'_{i'}$ were constructed as in the previous paragraph).
For this, use $\mcy' =  \bigl\{(y, \{x_1, x_2\}) \bigm| y \in \mcy,\; x_1,x_2 \in \mcx',\; x_1 \neq x_2 \bigr\}$. While we include the elements $(y_0, \{x_1, x_2\})$, $x_1,x_2 \in \mcx',\; x_1 \neq x_2$ for formal simplicity, we think of them as a single equivalence class, denoted by $y'_0$. Set then
\begin{align*}
  d_{\mcy'}\bigl( (y, \{x'_1, x'_2\}), (\ty, \{\tx'_1, \tx'_2\}) \bigr) &= d_{\mcy}(y,\ty) + d_{{\mcx'}^{\atwo}}\big(\{x'_1, x'_2\},\, \{\tx'_1, \tx'_2\} \big) \1\bigl\{y,\ty \neq y_0 \bigr\},
\end{align*}
which is consistent with the elements that make up $y'_0$. 
Let $e'\colon [m]^2 \to \mcy'$,
\begin{align*}
  e'(i,i') = e'_{ii'} = \begin{cases}
    y_0 &\text{if $e(i,i') = y_0$}; \\
    \bigl( e_{ii'}, \{v'_i,v'_{i'}\} \bigr) &\text{otherwise}
  \end{cases}
\end{align*}
be the new map of edge attributes that allows to recover $\{v'_i,v'_{i'}\}$ continuously from images $e_{ii'} \neq y_0$. Define $\sigma = \sum_{\substack{i,i'=1\\ i<i'}}^{m} \delta_{e'_{ii'}} \in \mfn(\mcy')$, which concludes the construction.


Write $\mfn'_{\mathbb{G}}$ for the space of $(\xi,\sigma) \in \mfn(\mcx')\times\mfn(\mcy')$ that can be obtained from elements of $\mathbb{G}$. By the above construction, the map $\mathbb{G} \ni ([m],v,e) \mapsto (\xi, \sigma) \in \mfn'_{\mathbb{G}}$ is a bijection.

We equip $\mfn(\mcx')$ and $\mfn(\mcy')$ with their weak topologies and refer to the topology $\mct'_w$ on $\mfn'_{\mathbb{G}}$ that is induced by the product topology on $\mfn(\mcx')\times\mfn(\mcy')$ as the \emph{weak topology} on $\mfn'_{\mathbb{G}}$. This slight abuse of terminology is justified additionally by the fact $\mfn(\mcx')\times\mfn(\mcy')$ is homeomorphic to $\mfn(\mcx'\times \mcy')$ equipped with the weak topology.

The following theorem states that $(\mfn'_{\mathbb{G}}, \mct'_w)$ is homeomorphic to $(\mathbb{G}, d_{\G})$, where we write $d_{\G}$ as a placeholder for any of the metrics $d_{\G,A}, d_{\G,R_1}$ or $d_{\G,R_2}$, all of which (invisibly) for any order $p \geq 1$. 
\begin{theorem}[Metrization of weak convergence] \label{thm: Metrization of weak convergence}
Let $(\mathcal{X}, d_{\mcx})$ and $(\mcy, d_{\mcy})$ be compact metric spaces. Consider  $([m_n],v_n,e_n),([m],v,e)\in \mathbb{G}$ with representations $(\xi_n,\sigma_n)$, $(\xi,\sigma)\in\mfn'_{\mathbb{G}}$, respectively. Then the following statements are equivalent
\begin{enumerate}
    \item[(i)] $(\xi_n,\sigma_n) \to (\xi,\sigma)$ weakly in $\mfn'_{\mathbb{G}}$;
    \item[(ii)] $d_{\mathbb{G}}(([m_n],v_n,e_n),([m],v,e))\to 0 $.
\end{enumerate}
Furthermore, interpreting $d_{\mathbb{G}}$ as metric on $\mfn'_{\mathbb{G}}$, the space $(\mfn'_{\mathbb{G}},d_{\mathbb{G}})$ is a complete separable metric space (c.s.m.s.).
\end{theorem}

To the best of our knowledge this defines a new type of (random) graph convergence. In the following we compare this new graph convergence with the notion of local weak convergence that was introduced independently by \textcite{aldous2004} and \textcite{benjamini2011} for deterministic graphs and later extended to random graphs by \textcite{vanderhofstad2023}.
Roughly speaking, a graph sequence $(G_n)_{n}$ converges locally weakly to a limit graph $G$ if the $k$-hop neighbourhood of a uniformly chosen root in $G_n$ looks like the $k$-hop neighbourhood of a (possibly random) root in $G$ as $k \to \infty$. Thus, local weak convergence can be used to study the local behaviour and structure of graphs when the number of vertices goes to infinity. If the limit graph is finite, as it is the case for our metrics, local weak convergence is rather peculiar in the sense that it does not even require that the number of vertices eventually matches with that of the limit graph, as the following example shows. 
\begin{example}[edge structure]
Consider the graph sequence $(G_n)_{n\in\N}$  where $G_n = (\xi_n,\sigma_n)$ is the graph consisting of $2n$ vertices (without vertex attributes) and $n$ edges, each connecting two vertices in such a way that every vertex has degree $1$. Further let $G = (\xi,\sigma)$ be the graph containing two vertices and one edge connecting these vertices. Choose one of the two vertices as the deterministic root of $G$ and note that other choices of the root would yield an isomorphic rooted graph. Clearly, $G_n$ converges to $G$ in the local weak sense. However, setting $d_{\mcx} \equiv 0$, we obtain $d_{\mathbb{G}, R_i}((\xi_n,\sigma_n),(\xi,\sigma))\to C_i$ for $i=1,2$ and $d_{\mathbb{G},A}((\xi_n,\sigma_n),(\xi,\sigma))\to \infty$ as $n\to \infty$.
\end{example}

In a second example, we illustrate how convergence in the graph metrics takes the vertex attributes into account quite naturally, whereas locally weak convergence ignores any additional information about the vertices.
\begin{example}[vertex attributes]
Fix $m\in\N$ and $\kappa \geq 0$. For every $n\in\N$, let the vertices of the graph $G_n  = (\xi_n,\sigma_n)$ be given by $m$ equidistant points on a circle of radius $1+\kappa+\varepsilon_n$, where $\varepsilon_n \to 0$ as $n\to\infty $. Furthermore, connect two vertices in $G_n$ if their distance is $\leq r_m$, where $r_m$ is a distance between equidistant points on a fixed circle whose radius is slightly larger than $1+\kappa$. Analogously, let $G = (\xi,\sigma)$ be the graph consisting of $m$ equidistant points on the unit circle that are connected to form a convex polygon. 
Then $G_n$ is isomorphic to $G$ as soon as the distance of two neighbouring points on a circle of radius $1+\kappa+\varepsilon_n$ is smaller or equal to $r_m$. In particular, $G_n$ converges to $G$ in the local weak sense. As in the previous example, the choice of the root in $G$ has no impact on this result. On the other hand, $d_{\mathbb{G}, R_i}((\xi_n,\sigma_n),(\xi,\sigma))\to \kappa$ for $i=1,2$ and $d_{\mathbb{G},A}((\xi_n,\sigma_n),(\xi,\sigma))\to m \kappa$ as $n\to \infty$. In particular, we obtain convergence w.r.t.\ the graph metrics only if $\kappa=0$, i.e.\ only if the vertex locations converge. 
\end{example}

Based on the c.s.m.s.\ $(\mfn'_{\G},d_{\G})$, we can now define convergence in distribution for random graphs, i.e.\ weak convergence on the space $\mathcal{M}_1(\G)$ of probability measures on $\mathbb{G}$ in the usual way. By the fact that on the compact set $\mcx$ the weak and vague topologies agree, we obtain that if $\Sigma_n = \Sigma = \emptyset$, then the weak convergence above is exactly the notion of weak convergence of point processes. 

We may use Wasserstein metrics for metrization. For any real $p \geq 1$, let 
$$\mathcal{M}_{1,p}(\G) = \biggl\{ P \in  \mathcal{M}_1(\G) \biggm\vert \int d_{\G}((\xi,\sigma), \boldsymbol{0})^p \, P(d(\xi,\sigma)) < \infty \biggr\},$$
where $\boldsymbol{0}$ denotes the empty graph.

\begin{definition}
The \emph{Wasserstein metric} $W_{\mathbb{G},p}$ on $\mathcal{M}_{1,p}(\G)$ of order $p$ based on $d_{\mathbb{G}}$ is defined as
$$W_{\G,p}(P,Q) = \Biggl( \inf_{\substack{(\Xi,\Sigma) \sim P \\ (\Eta,\Tau) \sim Q}} \E \bigl[ d_{\mathbb{G}} \bigl( (\Xi,\Sigma), (\Eta,\Tau) \bigr)^p \bigr] \Biggr)^{1/p}, \quad P,Q \in \mathcal{M}_{1,p}(\G).$$
\end{definition}
Then $W_{\mathbb{G},p}(P_n,P) \to 0$ is equivalent to $P_n \to P$ weakly and $\int d_{\G}((\xi,\sigma), \boldsymbol{0})^p \, P_n(d(\xi,\sigma)) \to \int d_{\G}((\xi,\sigma), \boldsymbol{0})^p \, P(d(\xi,\sigma))$ for all three underlying metrics; see \textcite[Theorem~6.9]{villani2009}. By the fact that the GOSPA metrics are bounded, we thus obtain that $W_{\mathbb{G},R_i,p}$, the Wasserstein metrics of order $p$ based on $d_{\G,R_i}$, $i=1,2$, metrize weak convergence.

Note that the underlying metric $d_{\G}$ has a separate order parameter (so far also denoted by $p$), which a priori does not have to agree with the order parameter of the Wasserstein metric. However, in what follows, we will assume that both are the same.

The following result may be seen as a sanity check that convergence in distribution and the above Wasserstein metrics are reasonable concepts. For more advanced results we refer to the forthcoming paper \textcite{sw2}. 
\begin{theorem} \label{thm: Convergence Example}
Consider a metric space $(\mcx,d_{\mcx})$ and $\mcy = \{0,1\}$ with $d_{\mcy}(e,f) = C_{\mcy} \, \1\{e \neq f\}$ for all $e,f \in \mcy$.
Let $\Xi = \sum_{i=1}^{M} \delta_{X_i}, \Xi_n = \sum_{i=1}^{M_n} \delta_{X_{n,i}}$, $n \in \N$, be simple point processes in $\mcx$ with $\Xi_n \inlawto \Xi$. Let $(E_{n,ii'})_{i,i' \in \N, i<i'}$ be i.i.d.\ $\mathrm{Be}(q_n)$-variables that are independent of $\Xi_n$ and $(E_{ii'})_{i,i' \in \N, i<i'}$ i.i.d.\ $\mathrm{Be}(q)$-variables that are independent of $\Xi$, and suppose that $q_n \to q$ in $[0,1]$. Set
  $$
    \Sigma_n = \sum_{\substack{i,i'=1\\ i<i'}}^{M_n} \delta_{(E_{n,{ii'}}, \{X_{n,i}, X_{n,i'}\})}, \quad \Sigma = \sum_{\substack{i,i'=1\\ i<i'}}^{M} \delta_{(E_{ii'}, \{X_{i}, X_{i'}\})}.
  $$
Then we have for the two GOSPA metrics and any $p \geq 1$
  \begin{equation*}
      W_{\mathbb{G},R_k,p}(\mcl(\Xi_n, \Sigma_n),\mcl(\Xi, \Sigma))^p \leq 
      \begin{cases}
        2 \, W_{\hspace*{-1.5pt}\varrho,C_1,p}(\mcl(\Xi_n),\mcl(\Xi))^p + \frac{C_{\mcy}^p}{2} \abs{q_n-q}
          &\text{\hspace*{-0.5mm}\text{if $k=1$}}, \\[2mm]
        \hspace*{-2.5pt}(1\hspace*{-1.5pt}+\hspace*{-1pt}\max\{q_n,q\}) \, W_{\hspace*{-1.5pt}\varrho,C_2,p}(\mcl(\Xi_n),\mcl(\Xi))^p + \frac{C_{\mcy}^p}{2} \abs{q_n-q}
          &\text{\hspace*{-0.5mm}\text{if $k=2$}}, \\
      \end{cases}
  \end{equation*}
where $W_{\varrho,C,p}$ denotes the Wasserstein metric w.r.t.\ the OSPA metric $\varrho$ with constant $C$ and order $p$. In particular,
  \begin{equation*}
      (\Xi_n, \Sigma_n) \inlawto (\Xi, \Sigma).
  \end{equation*}
\end{theorem}

\section{Computation} \label{sec:computation} 

Based on the equivalent expressions in Propositions~\ref{prop: GTT as optimal matching (filled to n+m)} and \ref{prop: proof of GOSPA1 and GOSPA2 filling up to same size does not change}, the computation of any of the three metrics presented in Section \ref{sec:theory} amounts to solving the following quadratic assignment problem:
\begin{equation} \label{eq:ap}
  \min_{\pi \in S_n} \; \biggl[ \frac{1}{2} \sum_{(i,i') \in [n]^2} d^{(E)}_{i,i',\pi(i),\pi(i')} + \sum_{i \in [n]} d^{(V)}_{i,\pi(i)} \biggr],\tag{QAP}
\end{equation}
where $d^{(V)}_{i,j} = d_{\mcx}(v_i, w_j)^p$ and $d^{(E)}_{i,i',j,j'} = c_n \cdot d_{\mcy}(e_{ii'}, f_{jj'})^p$. The constant $c_n$ depends on the metric: we set $c_n = 1$ for the GTT metric and $c_n = \frac{1}{n-1}$ for the GOSPA metrics. For simplicity, we use $n$ for the total number of indices in all three cases (rather than $m+n$ in case of the GTT metric).

Quadratic assignment problems have been introduced in \textcite{Koopmans1957} in a more special case, where $d^{(E)}_{i,i',j,j'}$ is of the form $g_{ii'} h_{jj'}$ for matrices $(g_{ii'})$ and $(h_{jj'})$. In our setting this form is obtained e.g.\ if $\mcy \subset \R$ and $d_{\mcy}(e_{ii'}, f_{jj'})^p = (e_{ii'}-f_{jj'})^2$ since the terms $e_{ii'}^2$ and $f_{jj'}^2$ have no influence on the minimizing permutation. The more general form above has first been studied in \textcite{Lawler1963}. See \textcite{Burkard1998} for a comprehensive review.

The quadratic assignment problems are known for their complexity. \textcite{sahni1976} showed that both, the problem itself and the search for an approximate solution are NP-hard. Hence, besides some exact algorithms, mainly (meta-)heuristics have been proposed to solve (versions of) the quadratic assignment problem also for larger sizes, see \textcite{Burkard1998} and \textcite{loiola2007} for an overview. Recent approaches try to improve the accuracy of the solution and the run-time by using parallelization (see for example \textcite{AbdelkafiEtAl2019}) or by combining the best performing meta-heuristics (see \textcite{dokeroglu2016}). We refer to \textcite{abdel2018} and \textcite{silva2021} for reviews of recent developments.

In what follows we first describe a basic method to find an exact solution, which is only reasonably applicable for (very) small graphs. Then we present the fast approximating quadratic programming (FAQ) algorithm for graph matching based on the square difference of edge-weights introduced by \textcite{vogelstein2015} in a variant that includes a vertex-based cost term. Finally we give a new heuristic based on the auction algorithm for the case that the vertex-based cost term is reasonably large.

\subsection{Exact solution} \label{ssec:exactsol}

In order to compute an exact solution of \eqref{eq:ap}, we identify $\pi \in S_n$ with the permutation matrix $\tPi = (\tilde{\pi}_{i,j})_{i,j \in [n]} \in \{0,1\}^{n \times n}$ given by $\tilde{\pi}_{i,j} = \delta_{\pi(i), j}$. This matrix has row and column sums equal to one and satisfies $\tPi x = (x_{\pi(i)})_{i \in [n]}$ for any $x = (x_i)_{i \in [n]}$. To simplify notation we linearize the space $[n]^2$ of pairwise indices, i.e.\ apply the transformation $\varphi \colon [n]^2 \to [n^2]$, $(i,j) \mapsto n (j-1) + i$. Define then
\begin{equation*}
\begin{split}
  z &= (z_k)_{k \in [n]} \quad \text{with } z_{k} = \pi_{\varphi^{-1}(k)}; \\
  r &= (r_{k})_{k \in [n]} \quad \text{with } r_{k} = d^{(V)}_{\varphi^{-1}(k)}; \\
  Q &= (q_{kl})_{k,l \in [n]} \quad \text{with } q_{kl} = \tilde{d}^{(E)}_{\varphi^{-1}(k),\,\varphi^{-1}(l)},
\end{split}
\end{equation*}
where $\tilde{d}^{(E)}_{(i,j),\,(i',j')} = d^{(E)}_{i,i',j,j'}$.

This allows us to cast \eqref{eq:ap} into the binary quadratic program
\begin{equation} \label{eq:bqp}
  \min_{z \in \{0,1\}^{[n^2]}} \, \bigl[ \tfrac12 z^{\top} Q z + r^{\top} z \bigr]
  \quad \text{subject to $Az = b$}, \tag{BQP}
\end{equation}
where the side constraint expresses the fact that $z$ has to be a linearized permutation matrix; more precisely, $b = \boldsymbol{1} \in \R^{2n}$ and $A = (a_{st}) \in \{0,1\}^{(2n) \times (n^2)}$ with
\begin{equation*}
  a_{st} = \begin{cases}
    1 &\text{if $1 \leq s \leq n$ and $t = n (u-1) + s$ for some $u \in [n]$;} \\
    1 &\text{if $n+1 \leq s \leq 2n$ and $t = n (s-n-1) + u$ for some $u \in [n]$;} \\
    0 &\text{otherwise.}
           \end{cases}
\end{equation*}

This binary quadratic program can be transformed in several ways into a mixed integer linear program (MILPs), see \textcite{Burkard1998} and \textcite{ForresterHunt2020}. The later paper also illustrates how different transformation of similar BQP make a substantial difference for time efficiency. In our simulation experiments in Section~\ref{sec:eval} we let the CPLEX solver deal directly with the BQP. Depending on the concrete example computation times become prohibitive for as little as 12--15 vertices per graph. 


\subsection{The FAQ algorithm}

The original FAQ algorithm by \textcite{vogelstein2015} computes heuristic solutions to quadratic assignment problems of the Koopmans--Beckmann type without linear terms, i.e., using matrix notation,
\begin{align*}
    &\text{min } \;\; -\text{tr}(EPF^{\top}P^{\top}) \\
    &\text{s.t. } \;\; P\in\mathcal{P},
\end{align*}
where $E,F\in \R^{n\times n}$ are two matrices and $\mathcal{P} = \{ P \in \{0,1\}^{n \times n} \mvert \1^{\top} P = P \1 = 1 \}$ is the set of permutation matrices. The algorithm first finds a local solution to the continuous (indefinite) quadratic program obtained by relaxation of the constraint to the set $\mathcal{D} = \{ D \in [0,1]^{n \times n} \mvert \1^{\top} D = D \1 = 1 \}$ of doubly stochastic matrices. This is achieved by employing the Frank--Wolfe algorithm \parencite{frankwolfe1956} from a given initial position. Then this local solution is projected to the set $\mathcal{P}$. Both the Frank--Wolfe steps and the final projection require solving linear assignment problems. We refer to \textcite{vogelstein2015} for a more thorough description of the algorithm.

In the following we describe how the FAQ algorithm can be applied to \eqref{eq:ap} (including the linear term) in the special case where $\mcy \subset \R$ and $d_{\mcy}(e_{ii'}, f_{jj'})^p = (e_{ii'}-f_{jj'})^2$. Let $E = (e_{ij})_{i,j}$ and $F= (f_{ij})_{i,j}$ be the (possibly weighted) adjacency matrices of the filled up graphs $([n],v,e)$ and $([n],w,f)$, respectively, and denote by $L = (\ell_{ij})_{i,j} = \bigl( d_{\mcx}(v_i, w_j)^p \bigr)_{ij}$ the matrix of the corresponding vertex dissimilarities. Then, for any $\pi \in S_n$ and $P = (\delta_{\pi(i),j})_{ij}$, 
\begin{align} 
    \frac{1}{2} \sum_{(i,i') \in [n]^2} d^{(E)}_{i,i',\pi(i),\pi(i')} + \sum_{i \in [n]} d^{(V)}_{i,\pi(i)} 
    &= \frac{1}{2} \, c_n \sum_{(i,i') \in [n]^2} (e_{i\pi^{-1}(i')}-f_{\pi(i)i'})^2 + \sum_{i \in [n]} d_{\mcx}(v_i, w_{\pi(i)})^p \notag\\
    &= \frac{1}{2} \, c_n\,\norm{EP-PF}^2_F + \text{tr}(LP^{\top}), 
    \label{eq:graphmetric_to_matrixterm}
\end{align}
where $\norm{\cdot}_F$ is the Frobenius norm. As in \textcite{vogelstein2015}, we argue by orthogonality of $P$ and the trace formula that
$$
  \norm{EP-PF}^2_F = \text{tr}\bigl((EP-PF)^{\top} (EP-PF) \bigr) = \text{tr}(E^{\top} E) + \text{tr}(F^{\top} F) - 2 \, \text{tr}(E P F^{\top} P^{\top}).
$$
Since the first two summands do not depend on $P$, the optimal permutation minimizing (QAP) can be obtained by solving the equivalent quadratic assignment problem
\begin{align*}
    &\text{arg\,min } \;\:\,-\,c_n\,\text{tr}(EPF^{\top}P^{\top}) + \text{tr}(LP^{\top}) \\
    &\text{s.t. } \hspace*{1cm}P\in\mathcal{P}.
\end{align*}
We may then follow the concrete steps of the FAQ algorithm to obtain a heuristic solution again. Some of the computations needed for the Frank--Wolfe search change a little due to the linear term, but all the essential properties of subproblems and functions remain.  

The following proposition extends Proposition~1 of \textcite{vogelstein2015} in that it allows for a linear term and general matrices $E \in \R^{n \times n}$ rather than just unweighted adjacency matrices in $\{0,1\}^{n \times n}$. It shows that the relaxation to the set of doubly stochastic matrices $\mathcal{D}$ is a reasonable heuristic approach to solve the previously considered (QAP) problem.
\begin{proposition} \label{prop: FAQ adapted}
Let $([n],v,e)$ be a simple graph with vertex distance matrix $L$ and (weighted) adjacency matrix $E \in \R^{n \times n}$. Assume that $([n],v,e)$ has no symmetries, i.e.\ there is no $\pi \in S_n \setminus \{\text{id}\}$ such that both $d_{\mcx}(v_i,v_{\pi(i)}) = 0$  and $e_{ii'} = e_{\pi(i)\pi(i')}$ for all $i,i' \in [n]$. 
Then
\begin{align} \label{eq:relaxed_sol_unique}
    \argmin_{D\in\mathcal{D}} \;\,-\,c_n\,\text{tr}(EDE^{\top}D^{\top}) + \text{tr}(LD^{\top}) = \{I\},
\end{align}
where $I$ denotes the identity matrix.
\end{proposition}
Note that the prerequisite of the proposition is already satisfied if all pairwise distances between the vertex attributes of $([n],v,e)$ are positive. If there are symmetries, the corresponding permutation matrices are in the set of solutions along with $I$ and all the convex combinations of these matrices. 

For the proof we use the following refinement of the Hardy--Littlewood--P\'olya rearrangement inequality, which we could not locate in the literature (see \cite[Theorem 368]{hardy1934} for the original version). Its prove by induction is straightforward.
\begin{lemma} \label{lem:hlp}
  Let $n \in \N$ and let $a_1 \leq \ldots \leq a_n$ and $b_1 \leq \ldots \leq b_n$ be real numbers. For any $\pi \in S_n$ we have 
  \begin{equation*}
    \sum_{i=1}^{n} a_i b_{\pi(i)} \leq \sum_{i=1}^{n} a_i b_i.
  \end{equation*}
  The inequality is strict iff there is an $i \in [n]$ with $a_{\pi^{-1}(i)} \neq a_i$ and $b_{\pi(i)} \neq b_i$,
\end{lemma}

\begin{proof}[Proof of Proposition~\ref{prop: FAQ adapted}]
For permutations $\pi,\sigma \in S_n$ and their matrices ${P_{\pi} = (\delta_{\pi(i),j})_{i,j}}$, $P_{\sigma} = (\delta_{\sigma(i),j})_{i,j} \in \mcp$, we obtain
$$
  \text{tr}(EP_{\pi}E^{\top}P_{\sigma}^{\top}) = \sum_{(i,i') \in [n]^2} e_{ii'} e_{\pi(i)\sigma(i')}  \quad \text{and} \quad \text{tr}(LP_{\pi}^{\top}) =  \sum_{i \in [n]} \ell_{i,\pi(i)}.
$$ 
Trivially,
$$
\sum_{i \in [n]} \ell_{i,\pi(i)} \geq 0 = \sum_{i \in [n]} \ell_{ii}
$$
and the inequality is strict unless $\ell_{i,\pi(i)} = 0$ for all $i$.
Lemma~\ref{lem:hlp} applied to the doubly indexed sum (sort $e_{ii'}$ and linearize the index set) using as permutation $(i,i') \mapsto (\pi(i), \sigma(i'))$ yields 
\begin{equation} \label{eq:hlp_for_edges}
\sum_{(i,i') \in [n]^2} e_{ii'} e_{\pi(i)\sigma(i')} \leq \sum_{(i,i') \in [n]^2} e_{ii'}^{2} 
\end{equation}
and the inequality is strict unless for all $i,i'$ we have $e_{\pi(i)\sigma(i')} = e_{ii'}$ (and hence also $e_{\pi^{-1}(i)\sigma^{-1}(i')} = e_{ii'}$).

As in the proof of Proposition~1 in \textcite{vogelstein2015}, we use the Birkhoff--von Neumann Theorem to represent $D \in \mcd$ as a convex combination of projection matrices, $D = \sum_{\pi \in S_n} \alpha_\pi P_{\pi}$ with constants $\alpha_\pi \geq 0$ (depending on $D$) satisfying $\sum_{\pi \in S_n} \alpha_\pi = 1$. For $D \in \mcd \setminus \{I\}$ there is a $\pi_0 \in S_n \setminus \{\text{id}\}$ with $\alpha_{\pi_0} > 0$. Thus denoting the term on the left hand side of \eqref{eq:relaxed_sol_unique} by $f(D)$,
\begin{align*}
  f(D) &= - \,c_n\,\text{tr}(EDE^{\top}D^{\top}) + \text{tr}(LD^{\top}) \\[1mm]
  &= - \,c_n \sum_{\pi,\sigma \in S_n} \alpha_\pi \, \alpha_\sigma \, \text{tr}(EP_{\pi}E^{\top}P_{\sigma}^{\top}) +  \sum_{\pi \in S_n} \alpha_\pi \, \text{tr}(LP_{\pi}^{\top}) \\
  &= \sum_{\pi \in S_n} \alpha_\pi \biggl( - \alpha_\pi\,c_n \sum_{(i,i') \in [n]^2} e_{ii'} e_{\pi(i)\pi(i')} + \sum_{i \in [n]} \ell_{i,\pi(i)} \biggr) - \sum_{\substack{\pi,\sigma \in S_n \\ \pi \neq \sigma}} \alpha_\pi \, \alpha_\sigma\,c_n  \sum_{(i,i') \in [n]^2} e_{ii'} e_{\pi(i)\sigma(i')} \\
  &> - \sum_{\pi,\sigma \in S_n} \alpha_\pi \, \alpha_\sigma \,c_n \sum_{(i,i') \in [n]^2} e_{ii'}^2 \\[1mm]
  &= - \,c_n\,\text{tr}(EE^{\top}) = f(I),
\end{align*}
where the strict inequality holds by \eqref{eq:hlp_for_edges}, $\alpha_{\pi_0} > 0$ and the fact that
$$
  - \alpha_\pi \sum_{(i,i') \in [n]^2} e_{ii'} e_{\pi(i)\pi(i')} + \sum_{i \in [n]} \ell_{i,\pi(i)}
  > - \alpha_\pi \sum_{(i,i') \in [n]^2} e_{ii'}^2 ,
$$
for $\pi = \pi_0 \in S_n \setminus \{\text{id}\}$, since by the prerequisite we cannot have both $d_{\mcx}(v_i,v_{\pi(i)}) = 0$ and $e_{ii'} = e_{\pi(i)\pi(i')}$ for all $i,i' \in [n]$.
\end{proof}

\textcite{Lyzinski2016} compares the original FAQ algorithm with a number of other prominent relaxation algorithms and shows theoretically and practically that the FAQ algorithm performs best. We remark that the theoretical result (see \textcite[Theorem 1a]{Lyzinski2016}) for the QAP (without linear term) can be extended to the QAP with linear term by constructing the $p$-correlated random graphs with respect to the optimal vertex solution. 

For the practical performance, we refer to Section~\ref{sec:eval}. We use a slight adaptation of the FAQ algorithm implemented in the \textsf{R} package \textsf{iGraphMatch} (\cite{igraphmatch}). 
It stands to reason that, at least for the GOSPA metrics, where the additional factor $c_n$ makes the edge terms an order of magnitude smaller, the FAQ algorithm (and other methods) perform better than if there is no linear term. Also one might expect that starting the FAQ algorithm in the optimal solution of the linear vertex matching problem (ignoring the edges) rather than in the barycentre $\frac{1}{n}\1 \1^{\top} \in \mcd$ or at least in a proper convex combination of the two is beneficial, but, as it turns out, performance is rather worse. We therefore always start the FAQ algorithm in the barycentre. 


\subsection{An auction algorithm with externalities}

For the linear assignment problem, the auction algorithm by~\textcite{Bertsekas1988} has been used with great success, see e.g.\ Remark~1 in \textcite{MuellerEtAl2020}. We propose a heuristic extension of the algorithm for our quadratic assignment problem with substantial linear term. Unlike the FAQ algorithm, this algorithm is able to solve the QAP with linear term for arbitrary choices of the underlying metric $d_{\mcy}$.

The main idea of the auction algorithm is to think of the vertices of one graph as bidders and the vertices of the other graph as objects in an auction. We assume here that these roles are fixed, but it can be advantageous to switch the roles from time to time during the algorithm. The bidders take turns in bidding for an object that currently has the highest personal value to them, which is computed as benefit (a constant minus the cost of all terms in the distance computation that involve the bidder) minus price for the object. Bidding assigns the object to the bidder, cancelling any previous assignments involving this bidder or this object. The object prices all start at zero. Each bid increases the price of the object bid for by the difference of the personal value of the object for the bidder and the personal value of the second-most coveted object plus a small quantity~$\varepsilon$. To improve the accuracy of the auction algorithm, this constant~$\varepsilon$ should be chosen in relation to the differences in the bidders' personal values for different objects. 

For the classic auction algorithm (without edge costs), bidding stops automatically once all objects are assigned to bidders. This can be shown to be an overall optimal assignment. For our extended auction algorithm, we include the costs of incident (non-)edge attributes in the benefit, assuming that edges between the current bidder and unassigned other bidder can be matched in an optimal way for the current bidder (which in general is not true, but practical). Furthermore the personal value is discounted by a compensation fee the bidder has to pay for the external costs her bid generates. More concretely, if bidder $i$ places a bet for object $j$ any of the two following things may happen (individually or both together) that require such a compensation:\ (1) a bidder $i_0$ previously bidding for $j$ is unassigned, so bidder $i$ pays for the loss (or receives for the gain) in benefit of any assigned third-party bidder-object pair $(k,l)$, which is $d_{\mcy}(e_{ik},f_{jl}) - d_{\mcy}(e_{i_0k},f_{jl})$;\ (2) an object $j_0$ previously bid for by $i$ is unassigned, so bidder $i$ pays for the loss (or receives for the gain) in benefit of any assigned third-party bidder-object pair $(k,l)$, which is $d_{\mcy}(e_{ik},f_{jl}) - d_{\mcy}(e_{ik},f_{j_0l})$. 

The changes in our extended algorithm imply that bidding usually no longer stops once all objects are assigned and that our algorithm becomes purely heuristic. We stop the algorithm when either no more changes occur, when a full assignment has been reached for the $\textsf{stop\_at}$-th time or after a fixed number $\textsf{iter}$ of iterations, whichever happens first.

We give pseudocode for the heuristic auction algorithm in Appendix~\ref{sec: algorith}.
Algorithm~\ref{algo:AuctionExt} describes the main routine, which except for the additional stopping criteria is precisely the vanilla version of the classic auction algorithm. Algorithm~\ref{algo:compute_persvals} provides pseudocode for computing the vector of personal values of objects for a given bidder. The largest part is the computation of \DataSty{edgedists} and \DataSty{compensations}, which is our extension. 


\section{Simulation based evaluation of the algorithms} \label{sec:eval}

We perform a simulation study to evaluate the FAQ algorithm and the auction algorithm presented in Section~\ref{sec:computation} in terms of accuracy and run-time. 
Our study is limited to the GOSPA2 metric, but a partial analysis of the GOSPA1 metric indicates similar results. We omit testing the algorithms with respect to the GTT metric, among other things because the fill-up procedure (see Proposition~\ref{prop: GTT as optimal matching (filled to n+m)}) leads to graphs of size $m+n$ and thus to a larger BQP problem. 

For the input data we restrict ourselves to independent and dependent Erd\H{o}s--R\'{e}nyi graphs with independent $\mathrm{Unif}([0,1]^2)$ vertex attributes (that may be thought of as spatial positions of the vertices) and no additional edge attributes, i.e.\ using just the i.i.d.\ 1-0 information of ``edge'' or ``no edge'' from the Erd\H{o}s--R\'{e}nyi construction.

In the first scenario, we fix possibly different numbers of vertices $m$ and $n$ and edge probabilities $p$ and $q$ and generate the two graphs independently of one another. In the second scenario we create one graph of $m$ vertices with edge probability $p$ first and obtain the other one by adding i.i.d.\ $\mcn_2(0,\sigma^2 I)$-distributed deviations to the positions and flipping \mbox{(non-)edges} with probability~$r$.

For the underlying metrics we use $d_{\mcx}(x_1,x_2) = \min\{K, \norm{x_1-x_2}\}$ on the vertex space $\mcx = \R^2$ and $d_{\mcy}(y_1,y_2) = K \1\{y_1 \neq y_2\}$ on the edge space $\mcy = \{0,1\}$, with the same constant $K>0$. We choose the GOSPA2 penalty minimally as $C_2 = 2K$.

In what follows, we first consider small graphs of size $n\leq 11$, for which an exact computation with CPLEX is still easily possible, and then analyze the performance on larger graphs of sizes $20\leq n\leq 100$, for which we can usually not obtain an optimal solution. All simulations were performed on a HP ProLiant server with two Intel Xeon X5660 processors in \textsf{R}~$4.2.0$ \parencite{R2023}, using IBM ILOG CPLEX~$12.7.1$ and \textsf{iGraphmatch}~$2.0.1$. The experiments can be reproduced with the \texttt{gmspat} function in the \textsf{R}-package \textsf{graphmetrics}, obtainable at \url{https://github.com/dschuhmacher/graphmetrics}.

\subsection{Performance on small graphs}

We only present the first scenario (independent graphs) in detail, because the second scenario results in similar accuracies and similar (uniformly somewhat better) run-times, which would lead to the same conclusions.

Thus we consider pairs of independent graphs of (possibly different) sizes $m,n \in \{4,8,11\}$ constructed with edge probabilities $p = 0.3$ and $q = 0.4$, respectively. As the order of the input graphs has no influence on the performance of the algorithms, it suffices to consider the cases $(m,n) \in \{(4,4),(4,8),(4,11),(8,8),(8,11),(11,11)\}$. Furthermore, we vary the free metric parameter $K$ in $\{0.1,0.4,0.8\}$. This yields $6 \times 3 = 18$ different settings in total, for each of which we generate $100$ samples.

We fix the parameters of the auction algorithm to $\varepsilon = 0.01$, $\textsf{stop\_at} = 3$ and $\textsf{maxiter} = 100$ and start the FAQ algorithm in the barycentre matrix $\frac{1}{n}\1\1^{\top}$.
The accuracy results are depicted in Table~\ref{tab: S1O2Res}. The exact graph distance was computed by solving the corresponding BQP with CPLEX, see Subsection~\ref{ssec:exactsol}. 

We observe that the considered algorithms perform well overall, i.e.\ do usually not deviate more than $2.5\%$ from the optimal solutions, and that the auction algorithm performs better than the FAQ algorithm. There appears to be a tendency that results are most accurate for the intermediate parameter choice of $K=0.4$. 
\begin{table}[ht!]
    \centering
    \begin{filecontents*}{SimulationData/Simulation_Res_S1_O2.csv}
"(4,4)";"0.1";"0";"(0, 0)";"0.014";"(0, 0.143)"
" ";"0.4";"0";"(0, 0)";"0.005";"(0, 0.005)"
" ";"0.8";"0.001";"(0, 0)";"0.01";"(0, 0.058)"
"(4,8)";"0.1";"0";"(0, 0)";"0.001";"(0, 0.01)"
" ";"0.4";"0";"(0, 0)";"0";"(0, 0)"
" ";"0.8";"0";"(0, 0)";"0.001";"(0, 0.01)"
"(4,11)";"0.1";"0";"(0, 0)";"0.001";"(0, 0.01)"
" ";"0.4";"0";"(0, 0)";"0";"(0, 0)"
" ";"0.8";"0";"(0, 0)";"0";"(0, 0.002)"
"(8,8)";"0.1";"0.001";"(0, 0.006)";"0.018";"(0, 0.067)"
" ";"0.4";"0.001";"(0, 0.007)";"0.008";"(0, 0.035)"
" ";"0.8";"0.003";"(0, 0.022)";"0.017";"(0, 0.071)"
"(8,11)";"0.1";"0.002";"(0, 0.014)";"0.007";"(0, 0.026)"
" ";"0.4";"0";"(0, 0.003)";"0.003";"(0, 0.016)"
" ";"0.8";"0.001";"(0, 0.005)";"0.004";"(0, 0.02)"
"(11,11)";"0.1";"0.006";"(0, 0.022)";"0.016";"(0, 0.038)"
" ";"0.4";"0.002";"(0, 0.013)";"0.008";"(0, 0.032)"
" ";"0.8";"0.004";"(0, 0.022)";"0.019";"(0, 0.064)"
\end{filecontents*}
\pgfplotstabletypeset[
    col sep= semicolon,
    string type,
    column type=l,
    string replace*={"}{},
    header=false,
    every head row/.style={ 
    output empty row,
    before row={%
        \hline
         \multicolumn{1}{|c|}{$(m,n)$} &\multicolumn{1}{c|}{$K$} & \multicolumn{2}{c|}{Auction}& \multicolumn{2}{c|}{FAQ}\\
        },
    after row=\hline},
    columns/0/.style={
        assign cell content/.code={%
        \pgfmathparse{int(Mod(\pgfplotstablerow,3)}%
        \ifnum\pgfmathresult=0%
            \pgfkeyssetvalue{/pgfplots/table/@cell content}%
            {\multirow{3}{*}{##1}}%
        \fi%
        },
    },
    every row no 2/.style={after row = \hline},
    every row no 5/.style={after row = \hline},
    every row no 8/.style={after row = \hline},
    every row no 11/.style={after row = \hline},
    every row no 14/.style={after row = \hline},
    every row no 17/.style={after row = \hline},
    every first column/.style={column type/.add={|}{} },
    columns/1/.style = {column type/.add={|}{}},
	every odd column/.style={column type/.add={}{|}},
    ]{SimulationData/Simulation_Res_S1_O2.csv}

    \caption{Relative deviations from the optimal solution in the first scenario. Shown are the means as well as the $5\%$- and $95\%$-quantiles (in brackets) observed over $100$ samples.}
    \label{tab: S1O2Res}
\end{table}

The run-time results for the three algorithms are shown in Table~\ref{tab: S1TimeO2}. To save space, we omit the run-times for graph pairs of different sizes $m<n$ and note that they are similar to the run-times of graph pairs of the same larger size~$n$. 
We find that both the auction algorithm and the FAQ algorithm have significantly lower run-times than the exact algorithm and that the auction algorithm is about ten times faster than the FAQ algorithm. For the exact algorithm, we observe 
that the run-times generally increase with the maximal graph size, while for the other two algorithm such an effect is only marginal, indicating that this is still a problem size where the general overhead of the two algorithms dominates the computation.

\begin{table}[ht!]
    \centering
    \begin{filecontents*}{SimulationData/Simulation_Time_S1_O2.csv}
"(4,4)";"0.1";"0.86";"(0.05, 2.01)";"0.02";"(0.02, 0.03)";"0.29";"(0.27, 0.32)"
" ";"0.4";"0.87";"(0.05, 2.16)";"0.02";"(0.02, 0.03)";"0.3";"(0.27, 0.36)"
" ";"0.8";"1.47";"(0.05, 2.69)";"0.02";"(0.02, 0.03)";"0.3";"(0.27, 0.35)"
"(8,8)";"0.1";"2.26";"(1.01, 4.8)";"0.02";"(0.02, 0.02)";"0.32";"(0.26, 0.41)"
" ";"0.4";"1.55";"(0.97, 2.27)";"0.02";"(0.02, 0.02)";"0.3";"(0.26, 0.36)"
" ";"0.8";"1.65";"(1.05, 2.6)";"0.02";"(0.02, 0.02)";"0.31";"(0.26, 0.36)"
"(11,11)";"0.1";"10.49";"(2.47, 27.85)";"0.02";"(0.02, 0.02)";"0.36";"(0.29, 0.53)"
" ";"0.4";"2.86";"(2.06, 4.08)";"0.02";"(0.02, 0.03)";"0.33";"(0.27, 0.42)"
" ";"0.8";"2.85";"(2.13, 4.72)";"0.02";"(0.02, 0.03)";"0.34";"(0.26, 0.44)"
\end{filecontents*}
\pgfplotstabletypeset[
    col sep= semicolon,
    string type,
    column type=l,
    string replace*={"}{},
    header=false,
    every head row/.style={ 
    output empty row,
    before row={%
        \hline
         \multicolumn{1}{|c|}{$(m,n)$} &\multicolumn{1}{c|}{$K$} & \multicolumn{2}{c|}{CPLEX} & \multicolumn{2}{c|}{Auction}& \multicolumn{2}{c|}{FAQ}  \\
        },
    after row=\hline},
    columns/0/.style={
        assign cell content/.code={%
        \pgfmathparse{int(Mod(\pgfplotstablerow,3)}%
        \ifnum\pgfmathresult=0%
            \pgfkeyssetvalue{/pgfplots/table/@cell content}%
            {\multirow{3}{*}{##1}}%
        \fi%
        },
    },
    every row no 2/.style={after row = \hline},
    every row no 5/.style={after row = \hline},
    every row no 8/.style={after row = \hline},
    every first column/.style={column type/.add={|}{} },
    columns/1/.style = {column type/.add={|}{}},
	every odd column/.style={column type/.add={}{|}},
    ]{SimulationData/Simulation_Time_S1_O2.csv}

    \caption{Mean run-times in seconds for the first scenario as well as the $5\%$- and $95\%$-quantiles (in brackets) observed over $100$ samples. Note that CPLEX runs by default on multiple cores, while the other two algorithms only use a single core for such small problems.
}
    \label{tab: S1TimeO2}
\end{table}

\subsection{Performance on larger graphs}

Consider pairs of graphs of equal size $m_1\in \{20,50,100\}$ for the first scenario and $m_2\in \{30,50,100\}$ for the second scenario. We scale the parameters for the graph constructions according to the sizes. In the first scenario, set the edge probabilities to $p_1 = 3/m_1$ and $q_1 = 4/m_1$. In the second scenario set the edge probability of the first graph to $p_2 = 4/m_2$ and generate the perturbed graph by flipping every (non-)edge independently with probability $r_2=0.3$ and setting the variance $\sigma^2$ of the i.i.d.\ bivariate normal shifts of the vertices to $1/m_2^2$. 
The free metric parameter $K$ is chosen in $\{0.1, 0.4\}$, yielding $3 \times 2 = 6$ settings for either of the two scenarios. Again we generate $100$ samples for each combination.

We fix the parameters of the auction algorithm to $\varepsilon_1 = 4Km_1/10000$ in the first and $\varepsilon_2 = 2Km_2/10000$ in the first and second scenario, and use $\textsf{stop\_at} = 15$ and $\textsf{maxiter} = 10000$ for both scenarios. All these parameters were chosen heuristically based on the experience from smaller graphs and a few test runs. For the FAQ algorithm, we take as starting point the barycentre matrix $\frac{1}{n}\1\1^{\top}$ again.

\begin{table}[ht!]
    \centering
    \begin{filecontents*}{SimulationData/Simulation_Increase_S1_O2.csv}
"20";"0.1";"0.0017";"(-0.0101, 0.0114)";"0.59";"(0.27, 1.13)";"0.4";"(0.32, 0.66)"
" ";"0.4";"0.0027";"(-0.0031, 0.014)";"1.28";"(0.52, 1.71)";"0.33";"(0.27, 0.42)"
"50";"0.1";"-0.096";"(-0.1389, -0.0616)";"4.74";"(1.53, 11.22)";"0.45";"(0.35, 0.69)"
" ";"0.4";"-0.1724";"(-0.2285, -0.1111)";"20.33";"(20.13, 20.52)";"0.38";"(0.31, 0.69)"
"100";"0.1";"-0.1354";"(-0.1644, -0.1003)";"2.77";"(1.74, 3.77)";"0.5";"(0.37, 0.77)"
" ";"0.4";"-0.213";"(-0.266, -0.1661)";"164.09";"(163.68, 164.56)";"0.4";"(0.34, 0.5)"
\end{filecontents*}
\pgfplotstabletypeset[
    col sep= semicolon,
    string type,
    column type=l,
    string replace*={"}{},
    header=false,
    every head row/.style={ 
    output empty row,
    before row={%
        \hline
         \multicolumn{1}{|c|}{$m_1$} &\multicolumn{1}{c|}{$K$} & 
         \multicolumn{2}{c|}{relative deviation}&\multicolumn{4}{c|}{run-time}\\
        },
    after row=
             \multicolumn{1}{|c|}{} &\multicolumn{1}{c|}{} & 
         \multicolumn{2}{c|}{}&
         \multicolumn{2}{c|}{Auction}& \multicolumn{2}{c|}{FAQ}  \\
         \hline
         },
    columns/0/.style={
        assign cell content/.code={%
        \pgfmathparse{int(Mod(\pgfplotstablerow,2)}%
        \ifnum\pgfmathresult=0%
            \pgfkeyssetvalue{/pgfplots/table/@cell content}%
            {\multirow{2}{*}{##1}}%
        \fi%
        },
    },
    every row no 1/.style={after row = \hline},
    every row no 3/.style={after row = \hline},
    every row no 5/.style={after row = \hline},
    every first column/.style={column type/.add={|}{} },
    columns/1/.style = {column type/.add={|}{}},
	every odd column/.style={column type/.add={}{|}},
    ]{SimulationData/Simulation_Increase_S1_O2.csv}

    \caption{Relative deviation of the solution obtained with the FAQ algorithm from the solution obtained with the auction algorithm in the first scenario as well as the corresponding run-times for both algorithms. Shown are the means as well as the $5\%$- and $95\%$-quantiles (in brackets) observed over $100$ samples. Note that our implementation of the auction algorithm uses only a single thread, whereas some underlying problems in the FAQ-algorithm use multiple threads.}
    \label{tab: S1LargeO2}
\end{table}
The results with respect to accuracy and run-time in the first scenario (independent pairs of graphs) are depicted in Table~\ref{tab: S1LargeO2}. We observe that the relative performance of the auction algorithm compared to the FAQ algorithm decreases as the problems become larger. More precisely, up to a graph size $m_1$ of about $20$, the auction algorithm has on average better accuracy than the FAQ algorithm. For graphs of (much) larger size, the FAQ algorithm performs considerably better. A similar behaviour is observed for the run-times. For small graphs, the auction algorithm is faster or about as fast as the FAQ algorithm. However, the run-times of the auction algorithm substantially increase with increasing graph size, while the run-times of the FAQ algorithm increase only slowly. It is noticeable that the run-time of the auction algorithm for $K=0.1$ is much smaller than the run-time of the algorithm for $K=0.4$. However, 
this might be attributed to the (heuristic) choice of the $\varepsilon$ parameter, more precisely, choosing a smaller $\varepsilon$ value in the cases where $K=0.4$ would decrease the run-time but also the accuracy. Although we are not able to further specify neither the ratio of this increases nor the influence of the $\varepsilon$ parameter, it is still reasonable to presume that in this scenario the auction algorithm yields better results in the setting where $K=0.1$.

\begin{table}[ht!]
    \centering
    \begin{filecontents*}{SimulationData/Simulation_Increase_S2_O2.csv}
"30";"0.1";"0.0002";"(0, 0)";"1.63";"(0.31, 4.71)";"0.29";"(0.27, 0.32)"
" ";"0.4";"0.0005";"(0, 0.0022)";"2.27";"(0.93, 4.21)";"0.3";"(0.27, 0.33)"
"50";"0.1";"-0.0012";"(-0.0098, 0)";"17.97";"(12.18, 18.82)";"0.3";"(0.28, 0.32)"
" ";"0.4";"-0.0435";"(-0.1056, 0)";"19.08";"(18.88, 19.27)";"0.31";"(0.28, 0.35)"
"100";"0.1";"-0.0152";"(-0.0453, 0)";"160.84";"(159.42, 162.83)";"0.32";"(0.28, 0.35)"
" ";"0.4";"-0.0254";"(-0.0546, -0.0036)";"160.32";"(159.87, 160.76)";"0.3";"(0.28, 0.32)"
\end{filecontents*}
\pgfplotstabletypeset[
    col sep= semicolon,
    string type,
    column type=l,
    string replace*={"}{},
    header=false,
    every head row/.style={ 
    output empty row,
    before row={%
        \hline
         \multicolumn{1}{|c|}{$m_2$} &\multicolumn{1}{c|}{$K$} & 
         \multicolumn{2}{c|}{relative deviation}&\multicolumn{4}{c|}{run-time}\\
        },
    after row=
             \multicolumn{1}{|c|}{} &\multicolumn{1}{c|}{} & 
         \multicolumn{2}{c|}{}&
         \multicolumn{2}{c|}{Auction}& \multicolumn{2}{c|}{FAQ}  \\
         \hline
         },
    columns/0/.style={
        assign cell content/.code={%
        \pgfmathparse{int(Mod(\pgfplotstablerow,2)}%
        \ifnum\pgfmathresult=0%
            \pgfkeyssetvalue{/pgfplots/table/@cell content}%
            {\multirow{2}{*}{##1}}%
        \fi%
        },
    },
    every row no 1/.style={after row = \hline},
    every row no 3/.style={after row = \hline},
    every row no 5/.style={after row = \hline},
    every first column/.style={column type/.add={|}{} },
    columns/1/.style = {column type/.add={|}{}},
	every odd column/.style={column type/.add={}{|}},
    ]{SimulationData/Simulation_Increase_S2_O2.csv}

    \caption{Relative deviation of the solution obtained with the FAQ algorithm from the solution obtained with the auction algorithm in the second scenario as well as the corresponding run-times of both algorithms. Shown are the mean as well as the $5\%$- and $95\%$-quantiles (in brackets) observed over $100$ samples. Note that our implementation of the auction algorithm uses only a single thread, whereas some underlying problems in the FAQ-algorithm use multiple threads.}
    \label{tab: S2LargeO2}
\end{table}

The results for the second scenario (dependent pairs of graphs) are given in Table~\ref{tab: S2LargeO2}. We observe that the auction algorithm performs better on smaller graphs up to a size $m_2$ of about $30$, while the FAQ algorithm yields better results on larger graphs. However, the relative performance of the auction algorithm (compared to the FAQ algorithm) is considerably better here than in the first scenario. The current parameter settings allow the auction algorithm to take much more time, which could likely be used to further improve the solution of the FAQ algorithm by restarting it many times at random locations. Note that the run-times for the auction algorithm are consistent within the same problem size now and are also similar to the (maximal) run-times from the first scenario. In contrast, the FAQ algorithm is somewhat faster in the second scenario than in the first.

\section{Real data application: neuronal graphs}  \label{sec:realdata}

In this section, we apply the GOSPA2 metric to spatial trees derived from olfactory projection neurons in the brains of \textit{Drosophila} flies and use the results to perform statistical tests. 

Olfaction in \textit{Drosophila} is an active field of research in neuroscience with a long-standing history that helps understand neuronal processes in other insects and vertebrates \parencite{benton2022}. The operation of the olfactory system can be roughly described as follows. 
An odor is perceived by receptor neurons in the olfactory organs (antennae and maxillary palps). These  neurons are specialized to react to specific types of volatile small molecules and transmit electric signals to the antennal lobe. Receptor neurons of the same type converge at subregions of the antennal lobe, referred to as glomeruli. There they form synapses with local interneurons and projection neurons, which pass on processed signals to the higher olfactory centres of the brain (mushroom body and lateral horn).
See \textcite{masse2009} for a more detailed description of the olfactory processes in \textit{Drosophila}. 

We consider the Cell07PNs data set from the \textsf{R} package \textsf{nat} \parencite{nat}. The data set contains information about 40 projection neurons as well as their respective glomeruli and is based on the work of \textcite{jefferis2007}.  The projection neurons are traced by a fluorescent labeling technique called MARCM. Staining of the surrounding brain area revealed the brain structure that was then used to map the neurons onto a single reference brain for comparison using a nonrigid 3D image registration algorithm.
We consider trees constructed from the data points of a neuron by taking only start, end and branching points and adding edges according to their connectivity. In this way, we obtain a total of $40$ trees that characterise the spatial structure of their corresponding projection neurons. On average, these trees consist of $51$ vertices, with the smallest tree containing $12$ and the largest tree $172$ vertices. Moreover, the number of branching points varies between $4$ and $84$ with an average of $24$ and is mostly similar to the number of end points. 
For more information on the initial data processing, see \textcite{jefferis2007}. 

It has been shown that the spatial tree structure of a projection neuron is characteristic for its associated glomerulus, see for example~\textcite{wong2002}. Figure~\ref{Fig: Drosophila Neurons} shows the neurons of the Cell07PNs data grouped by their associated glomeruli. In what follows we use the GOSPA2 metric in connection with the ANOVA methods on metric spaces from \textcite{muller2023} to see if we can confirm this result. The recommended procedure in that paper to test for differences between groups of spatial objects is to perform two permutation tests that are evaluated separately, one derived from the ANOVA-like statistic in \textcite{anderson2001}, which evaluates differences in location of the objects, and one from the classic Levene statistic \parencite{levene1960}, which evaluates differences in how the objects are scattered.
\begin{figure}[ht!]
    \centering
    \begin{subfigure}[t]{0.48\linewidth}
    \raggedleft
    \includegraphics[width=0.95\linewidth]{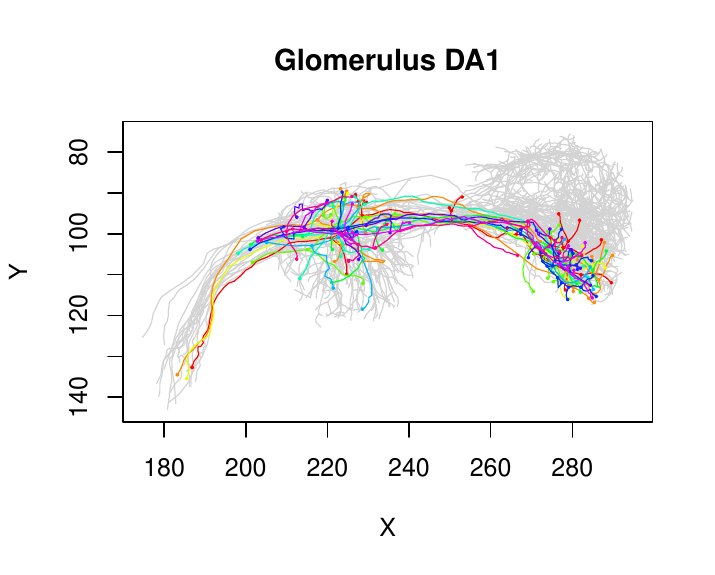}
    \end{subfigure}
 \begin{subfigure}[t]{0.48\linewidth}
    \raggedright
    \captionsetup{justification=raggedright, singlelinecheck=off}
    \includegraphics[width=0.95\linewidth]{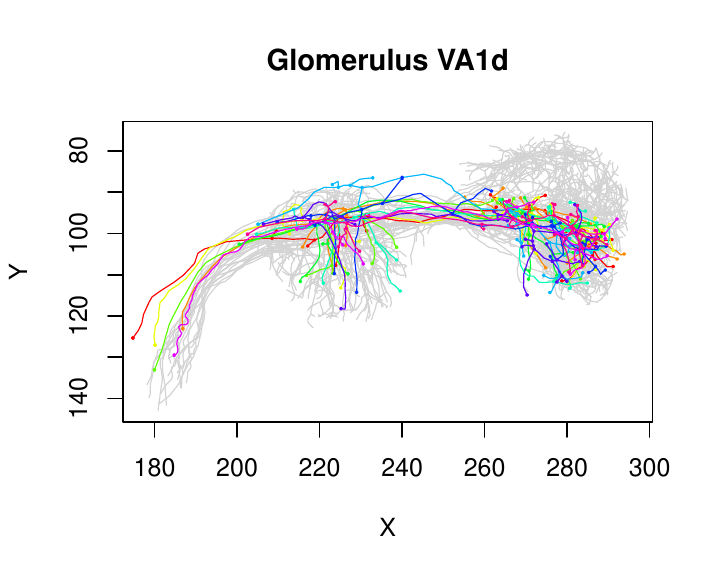}
    \end{subfigure}
        \begin{subfigure}[t]{0.48\linewidth}
    \raggedleft
    \includegraphics[width=0.95\linewidth]{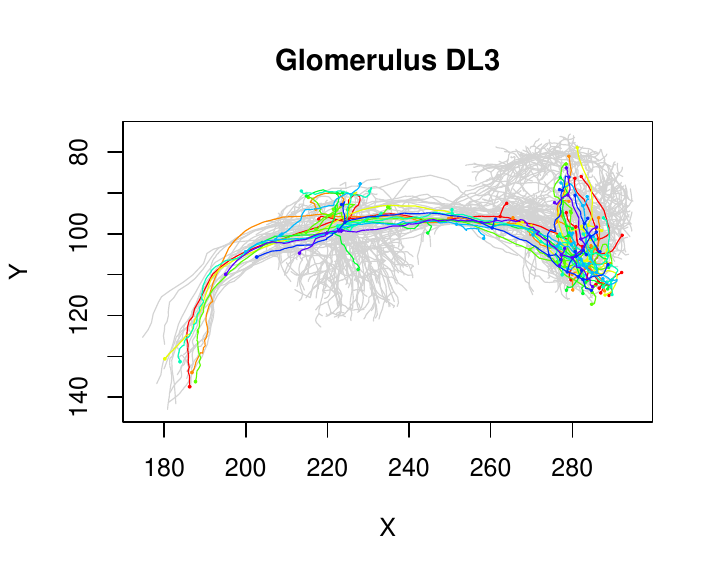}
    \end{subfigure}
 \begin{subfigure}[t]{0.48\linewidth}
    \raggedright
    \captionsetup{justification=raggedright, singlelinecheck=off}
    \includegraphics[width=0.95\linewidth]{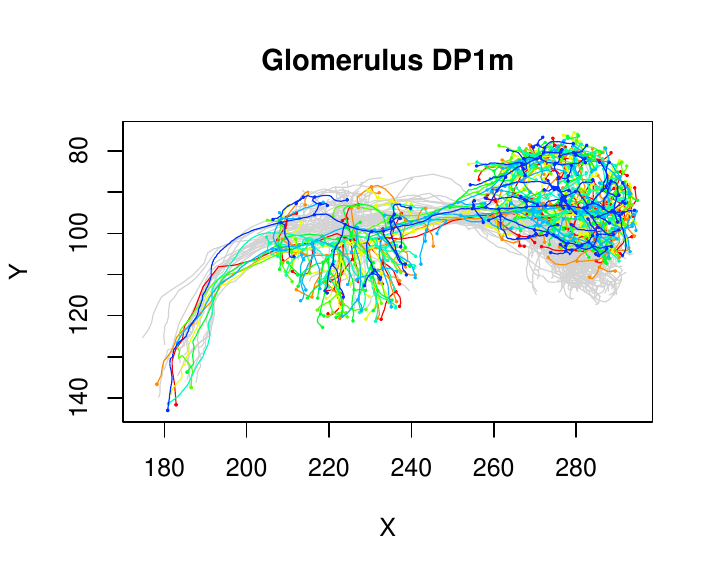}
    \end{subfigure}
    \vspace*{-3mm}
    
    \caption{Two-dimensional visualization of the 3D projection neurons from \textcite{jefferis2007} classified by their glomeruli. The structures in the middle show the synapses in the mushroom body, the structures to the right synapses in the lateral horn.}
  \label{Fig: Drosophila Neurons}
\end{figure}
To perform the tests, we calculate the graph distance of each pair of neuron graphs in the Cell07PNs data set with respect to the GOSPA2 metric with parameters $C_\mcx = C_\mcy = 100$ (and minimal constant $C_2$) using the FAQ algorithm. 
The p-values of the tests based on $n=10\,000$ permutations are shown in Table~\ref{tab: P-Vals Drosophila}.

\begin{table}[ht!]
    \centering
    \begin{filecontents*}{P_Vals_Drosophila_Final.csv}
"All","DA1-DL3","DA1-DP1m","DA1-VA1d","DL3-DP1m","DL3-VA1d","DP1m-VA1d"
"0.0001","0.1052","0.0001","0.0017","0.0001","0.0301","0.0001"
"0.0001","0.9496","0.0001","0.0018","0.0001","0.1075","0.0002"
\end{filecontents*}
\pgfplotstabletypeset[
    col sep= comma,
    string type,
    column type=l,
    string replace*={"}{},
    header=false,
    every head row/.style={ 
    output empty row,
    before row={%
        \hline
        \multicolumn{1}{|c|}{}&
        \multicolumn{7}{c|}{glomeruli}\\
        \multicolumn{1}{|c|}{}&
        \multicolumn{1}{c|}{\!\small{All}}&
        \multicolumn{1}{c|}{\!\small{DA1-DL3}\!}&
        \multicolumn{1}{c|}{\!\small{DA1-DP1m}\!}&
        \multicolumn{1}{c|}{\!\small{DA1-VA1d}\!}&
        \multicolumn{1}{c|}{\!\small{DL3-DP1m}\!}&
        \multicolumn{1}{c|}{\!\small{DL3-VA1d}\!}&
        \multicolumn{1}{c|}{\!\small{DP1m-VA1d}\!}
        \\
                \hline
        }},
    create on use/newCol/.style={
    create col/set list={,ANOVA, Levene}
    },
    skip rows between index = {0}{1},
    columns/newCol/.style={string type},
    columns = {newCol, [index] 0, [index] 1, [index] 2, [index] 3, [index] 4, [index] 5, [index] 6},
    every last row/.style={after row = \hline},
    columns/newCol/.style = {column type/.add={|}{}},
	every odd column/.style={column type/.add={}{|}},
	every even column/.style={column type/.add={}{|}},
    ]{P_Vals_Drosophila_Final.csv}

    \caption{p-values for metric space ANOVA and Levene tests for comparing the tree structure of projection neurons associated with different glomeruli. First column: omnibus test for all four glomeruli; then: tests between pairs of glomeruli.}
    \label{tab: P-Vals Drosophila}
\end{table}

The omnibus tests show significant differences in the tree structures of the neurons between glomeruli both with respect to location (here in terms of the ``usual'' localization of (sub)structures in the trees) and with respect to scatter (the overall variation from such localization). For the pairwise tests, the ANOVA statistic detects significant differences between most glomeruli except for the comparison of DA1 with DL3. Similarly, the Levene statistic sees significant differences except for the comparison of DL3 with DA1 and VA1d. These results are consistent with eyeball inspection of the data. For instance, the variance within the structures of DA1 and DL3 is quite large (maybe about the same, which explains the high p-value in the Levene test between the two) and even with the full spatial information visible from Figure~\ref{Fig: Drosophila Neurons}, it is not easy to distinguish the localization difference in arbitrary pairs of neurons associated with the two glomeruli.

It should be kept in mind that we do a rather course analysis here. The GOSPA2 distance between the tree structures sees far less than we do in Figure~\ref{Fig: Drosophila Neurons}, as only the location of nodes and ``straight'' edges between them are compared. Also, individual vertices of one tree that are far away from the other tree are only relatively mildly penalized due to the relative nature of the metric. A refined analysis based on graph metrics should consider finer segmentations of the neurons and include vertex attributes distinguishing the types of vertices (start, branching, end, other). Presumably, one would also use edge attributes and a finer edge metric.

We finally comment on the strong significance of any comparison involving glomerulus DP1m. As shown in Figure~\ref{Fig: Drosophila Neurons} the neurons associated to glomerulus DP1m are not only substantially larger than neurons of other glomeruli, but also show clear spatial differences. By applying the GOSPA2 metric we are able to detect differences in both aspects. This is because, although the GOSPA2 metric is a relative (spatial) metric, different graph sizes are penalized via the fixed cost $C_2$ of the added auxiliary nodes (see Proposition~\ref{prop: proof of GOSPA1 and GOSPA2 filling up to same size does not change}).
Moreover, by the relative construction of the metric, both the spatial distances and the differences in graph size can be balanced in a reasonable proportion to each other.  


\appendix

\section{Proofs left out in the main text}\label{sec: proofs left out}

\subsection{Proofs for the GTT metric} \label{ssec: proofs left out for gtt}

\begin{proof}[Proof of Proposition~\ref{prop: GTT as optimal matching (filled to n+m)}]
Let $I\subset [m]$ and $\pi_*\in S_n$ be minimizers according to the definition of $d_{\mathbb{G},A}(([m],v,e),([n],w,f))$. Set $l = \abs{I}$ and choose any permutation $\pi\in S_{n+m}$ satisfying
\begin{equation*}
    \pi(i) \begin{cases}
      =\pi_*(i) &\text{if $i\in I$} \\
      \in \{n+1,\ldots,n+m-l\} &\text{if $i \in [m]\setminus I$} \\
      \in [n]\setminus \pi_*(I) &\text{if $i \in \{m+1, \ldots, n+m-l\}$} \\
      = i &\text{if $i \in \{n+m-l+1, \ldots, n+m\}$}.
    \end{cases}
\end{equation*}
 \\
Then
\begin{align*}
    &d_{\mathbb{G},A}(([m],v,e),([n],w,f))\\
    &= \biggl[ (m-l)C^p + (n-l) C^p + \sum\limits_{i\in I} d_{\mcx}(v_i, w_{\pi_*(i)})^p\\
    & + \frac{1}{2} \sum\limits_{(i,i')\in I^2} d_{\mcy}(e_{ii'},f_{\pi_*(i)\pi_*(i')})^p + \frac{1}{2} \sum_{(i,i')\in [m]^2\setminus I^2 } d_{\mcy}(e_{ii'},y_0)^p + \frac{1}{2} \sum_{(j,j')\in [n]^2\setminus \pi_*(I)^2 } d_{\mcy}(y_0,f_{jj'})^p
    \biggr]^{1/p}\\
    &= \biggl[ \sum\limits_{i\in [m]\setminus I} d_{\mcx}(v_i,w_{\pi(i)})^p + \sum\limits_{i \in[n+m-l]\setminus[m]} d_{\mcx}(v_i,w_{\pi(i)})^p + \sum\limits_{i\in I} d_{\mcx}(v_i,w_{\pi(i)})^p\\
    & + \frac{1}{2} \sum\limits_{(i,i')\in I^2} d_{\mcy}(e_{ii'},f_{\pi(i)\pi(i')})^p +\frac{1}{2} \sum_{(i,i')\in [m]^2\setminus I^2} d_{\mcy}(e_{ii'},f_{\pi(i)\pi(i')})^p\\
    &+\frac{1}{2} \sum_{(i,i')\in [n+m-l]^2\setminus [m]^2} d_{\mcy}(e_{ii'},f_{\pi(i)\pi(i')})^p\biggr]\\
    &= \biggl[ \sum_{i\in[n+m]} d_{\mcx}(v_i,w_{\pi(i)})^p +\frac{1}{2} \sum\limits_{(i,i')\in [n+m]^2} d_{\mcy}(e_{ii'},f_{\pi(i)\pi(i')})^p\biggr]^{1/p}
\end{align*}
where the last equality follows as $v_i = x_* = w_{\pi(i)}$ for all $m+n-l+1\leq i\leq n+m$ and $e_{ii'} = y_0 = f_{\pi(i)\pi(i')}$ for all $m+n-l+1\leq i\vee j \leq n+m$.

On the other hand, let $\pi_*\in S_{n+m}$ be a minimizer of the expression in Proposition~\ref{prop: GTT as optimal matching (filled to n+m)}.
Set $I = \{i\in [m] : \pi_*(i) \in [n]\}$ and choose any permutation $\pi\in S_n$ satisfying $\pi(i) = \pi_*(i)$ for all $i\in I$ and $\pi([n]\setminus I)=[n]\setminus \pi_*(I)$.
Let furthermore $K = \{i\in [m] : \pi_*(i) \in [n+m]\setminus[n]\} = [m]\setminus I $, $L = \{i\in [n+m]\setminus [m] : \pi_*(i) \in [n]\} = \pi_*^{-1}([n]\setminus \pi_*(I))$ as well as $M = \{i\in [n+m]\setminus[m] : \pi_*(i) \in [n+m]\setminus[n]\}$. We remark that $w_{\pi_*(i)} = x_*$ and $f_{\pi_*(i)\pi_*(i')} = y_0$ for all $i\in K $ and $i'\in[n+m]$  and $v_i = x_*$ and $e_{ii'} = y_0$ for all $i\in[n+m]\setminus[m]$ and $i'\in[n+m]$. In particular we obtain $v_i = w_{\pi_*(i)}$ and $e_{ii'} = f_{\pi_*(i)\pi_*(i')}$ for all $i\in M$ and $i'\in[n+m]$. This together with the fact that $[n+m] = I   \overset{.}{\cup} K \overset{.}{\cup} L\overset{.}{\cup} M$ yields
\begin{align*}
    &\min_{\pi\in S_{n+m}} \biggl[ \sum_{i\in[n+m]} d_{\mcx}(v_i,w_{\pi(i)})^p +\frac{1}{2} \sum\limits_{(i,i')\in [n+m]^2} d_{\mcy}(e_{ii'},f_{\pi(i)\pi(i')})^p\biggr]^{1/p}\\
    &= \biggl[ \sum_{i\in I} d_{\mcx}(v_i,w_{\pi_*(i)})^p + \sum_{i\in K} d_{\mcx}(v_i,x_*)^p + \sum_{i\in L} d_{\mcx}(x_*,w_{\pi_*(i)})^p \\
    &\hspace*{6mm} +\frac{1}{2}\sum\limits_{(i,i')\in I^2}d_{\mcy}(e_{ii'},f_{\pi_*(i)\pi_*(i')})^p + \frac{1}{2} \sum_{i\in I}\biggl(\sum_{i'\in K} d_{\mcy}(e_{ii'},y_0)^p+ \sum_{i'\in L} d_{\mcy}(y_0,f_{\pi_*(i)\pi_*(i')})^p\biggr)\\
    &\hspace*{6mm} + \frac12 \sum_{i\in K}\biggl(\sum\limits_{i'\in [m] }d_{\mcy}(e_{ii'},y_0)^p + \sum_{i'\in [n+m]\setminus[m]} d_{\mcy}(y_0,y_0)^p\biggr) \\
    &\hspace*{6mm} + \frac12 \sum_{i\in L}\biggl(\sum\limits_{i'\in \pi_*^{-1}([n])} d_{\mcy}(y_0,f_{\pi_*(i)\pi_*(i')})^p + \sum_{i'\in \pi_*^{-1}([n+m]\setminus[n])} d_{\mcy}(y_0,y_0)^p\biggr)\biggr]^{1/p}\\
    &= \biggl[ \sum\limits_{i\in I} d_{\mcx}(v_{i},w_{\pi(i)})^p + (m-\abs{I})C^p + (n-\abs{I}) C^p\\
    &\hspace*{6mm} +\frac{1}{2} \sum\limits_{(i,i')\in I^2} d_{\mcy}(e_{ii'},f_{\pi(i)\pi(i')})^p + \frac{1}{2} \sum_{(i,i')\in [m]^2\setminus I^2 } d_{\mcy}(e_{ii'},y_0)^p + \frac{1}{2} \sum_{(j,j')\in [n]^2\setminus \pi(I)^2 } d_{\mcy}(y_0,f_{jj'})^p
    \biggr]^{1/p}\\
    &\geq d_{\mathbb{G},A}(([m],v,e),([n],w,f)).
\end{align*}
\vspace*{-12mm}

\hspace*{2mm}
\end{proof}

\begin{proof}[Proof of Theorem~\ref{thm: graph-TT is metric}]
We show the (pseudo-)metric properties of the map $d_{\mathbb{G},A}$ and start with the proof that the distance of a graph to itself is zero. Indeed, assume that $([m],v,e)=([n],w,f)$. Then $n=m$ and there exists an $\pi\in S_n$ such that $v_i = w_{\pi(i)}$ and $e_{ii'} = f_{\pi(i)\pi(i')}$ for all $i,i'\in [n]$. In particular, as $d_{\mcx}$, $d_{\mcy}$ are (pseudo-)metrics, $d_{\mcx}(x_{i},y_{\pi(i)}) = 0$ and $d_{\mcy}(e_{ii'}, f_{\pi(i)\pi(i')}) = 0$ for all $i,i'\in [n]$ and choosing $I = [n]$ yields $d_{\mathbb{G},A}(([m],v,e),([n],w,f)) = 0$.

In the case where $d_{\mcx}$, $d_{\mcy}$ are metrics, it remains to show that $d_{\mathbb{G},A}$ can distinguish between different graphs. Indeed, assume $d_{\mathbb{G},A}(([m],v,e),([n],w,f)) = 0$. As $C>0$ this implies $n = m = \abs{I}$ and thus $I = [n]$. Furthermore there exists a $\pi\in S_n$ such that $d_{\mcx}(x_{i},y_{\pi(i)}) = 0$ and $d_{\mcy}(e_{ii'}, f_{\pi(i)\pi(i')}) = 0$ for all $i,i'\in [n]$. This yields $([m],v,e)=([n],w,f)$ and thus the identity property for the case where $d_{\mcx}$, $d_{\mcy}$ are metrics.

The symmetry properties of $d_{\mcx}$ and $d_{\mcy}$ directly imply the symmetry property of $d_{\mathbb{G},A}$.

It remains to show the triangle inequality. Fill up two graphs $([m],v,e),([n],w,f)\in \mathbb{G}$ to a size $n_* \geq n+m$ by setting $v_i = x_*$ for $m+1\leq i\leq n_*$ and $e_{ij} = y_0$ for $m+1\leq i \vee j \leq n_*$ as well as $w_i = x_*$ for $n+1\leq i\leq n_*$ and $f_{ij} = y_0$ for $n+1\leq i \vee j \leq n_*$. Then the equality in Proposition~\ref{prop: GTT as optimal matching (filled to n+m)} is still true, i.e. 
\begin{align}\label{eq: Proof of GTT, larger size fill up }
   &d_{\mathbb{G},A}(([m],v,e),([n],w,f))= \min\limits_{\pi\in S_{n_*}} \biggl[ \sum_{i\in[n_*]} d_{\mcx}(v_i,w_{\pi(i)})^p +\frac{1}{2} \sum\limits_{(i,i')\in [n_*]^2} d_{\mcy}(e_{ii'},f_{\pi(i)\pi(i')})^p\biggr]^{1/p}.
  \end{align}
  
Indeed, let $\pi\in S_{n_*}$ be optimal. We aim to construct an optimal permutation $\pi'\in S_{n_*}$ that fulfills $\pi'([n_*]\setminus[n+m]) = [n_*]\setminus[n+m]$. Define $R_1 = [n]$, $R_2 = [n+m]\setminus[n]$ and $R_3 = [n_*]\setminus[n+m]$. For $j\in [3]$ consider $A_j = \{i\in [m]: \pi(i)\in R_j\}$, $D_j = \pi(A_j)$, $B_j = \{i\in [n+m]\setminus[m]: \pi(i)\in R_j\}$, $E_j = \pi(B_j)$ as well as $C_j = \{i\in [n_*]\setminus[n+m]: \pi(i)\in R_j\}$ and $F_j = \pi(C_j)$. 
Assume that $\abs{C_2} = \abs{B_3} + k_1$ for some $k_1\in \N_0$ and note that the construction of $\pi'$ in the case $\abs{C_2} \leq \abs{B_3}$ is analogous. Before constructing $\pi'$ we need to interchange $\pi(i)$ for $i \in C_2$ and $\pi(j)$ for $j \in B_3$ for as many indices as possible. This yields $\abs{C_2} = k_1$ and $B_3 = \emptyset$ as well as $\abs{F_2} = k_1$ and $E_3 = \emptyset$.

Let $k_2 = \abs{C_1}$. As $\pi$ is a permutation we have $\abs{C_3} = \abs{F_3}$ as well as $\abs{C_1\overset{.}{\cup} C_2\overset{.}{\cup} C_3} = \abs{D_3\overset{.}{\cup} E_3 \overset{.}{\cup} F_3}$. From this we obtain $\abs{A_3} = \abs{D_3} = \abs{C_1} + \abs{C_2} = k_1 + k_2$. We now construct $\pi'$ by mapping $C_2$ to the first $k_1$ elements of $D_3$ and $C_1$ to the remaining $k_2$ elements of $D_3$. Further we map the first $k_1$ elements of $A_3$ to the $k_1$ elements of $F_2$. The remaining $k_2$ elements of $A_3$ are mapped to the first $k_2$ elements of $E_2$, whose preimages in turn are mapped to $F_1$. 
For this last exchange we require $\abs{E_2}\geq k_2$. Indeed, we have $\abs{A_1} + \abs{A_2} + \abs{A_3} = m = \abs{D_2} + \abs{E_2} + \abs{F_2}$. By construction $\abs{A_2} = \abs{D_2}$ and $\abs{F_2} = k_1$ and therefore $\abs{E_2} = \abs{A_1} + \abs{A_3} - \abs{F_2} = \abs{A_1} + k_1 + k_2 - k_1 \geq k_2$.

Defining $\pi' = \pi$ on all remaining sets, we obtain a permutation $\pi'\in S_{n_*}$ that fulfills $\pi'([n_*]\setminus[n+m]) = [n_*]\setminus[n+m]$. Note that as $v_i = x_* $ for $m+1\leq i\leq n_*$ and $w_i = x_* $ for $n+1\leq i\leq n_*$ the permutation $\pi'$ is still optimal. Hence
\begin{align*}
   &\min\limits_{\pi\in S_{n+m}} \biggl[ \sum_{i\in[n+m]} d_{\mcx}(v_i,w_{\pi(i)})^p +\frac{1}{2} \sum\limits_{(i,i')\in [n+m]^2} d_{\mcy}(e_{ii'},f_{\pi(i)\pi(i')})^p\biggr]^{1/p}\\
   &= \min\limits_{\pi\in S_{n_*}} \biggl[ \sum_{i\in[n_*]} d_{\mcx}(v_i,w_{\pi(i)})^p +\frac{1}{2} \sum\limits_{(i,i')\in [n_*]^2} d_{\mcy}(e_{ii'},f_{\pi(i)\pi(i')})^p\biggr]^{1/p},
  \end{align*}
 which together with Proposition~\ref{prop: GTT as optimal matching (filled to n+m)} implies Equation~\ref{eq: Proof of GTT, larger size fill up }.
 
Now consider three graphs $([m],v,e), ([n],w,f),([l],u,g)\in\mathbb{G}$. Fill up all graphs to size $n_* = \max\{m+n,m+l,n+l\}$ by setting $v_i = x_*$ for $m+1\leq i\leq n_*$ and $e_{ij} = y_0$ for $m+1\leq i \vee j \leq n_*$, as well as $w_i = x_*$ for $n+1\leq i\leq n_*$ and $f_{ij} = y_0$ for $n+1\leq i \vee j \leq n_*$, and also $u_i = x_*$ for $l+1\leq i\leq n_*$ and $g_{ij} = y_0$ for $l+1\leq i \vee j \leq n_*$. As seen above this does not change the metric.

Let $\pi_1\in S_{n_*}$ and $\pi_2\in S_{n_*}$ be solutions of the optimal matching problems corresponding to $d_{\mathbb{G},A}(([m],v,e), ([l],u,g))$ and $ d_{\mathbb{G},A}(([l],u,g),([n],w,f))$, respectively. Take $\pi = \pi_2\circ \pi_1$. Then Equation~\ref{eq: Proof of GTT, larger size fill up } and the triangle inequalities for $d_{\mcx}$, $d_{\mcy}$ and the Minkowski inequality yield
\begin{align*}
    &d_{\mathbb{G},A}(([m],v,e), ([n],w,f))\\
    &\leq \biggl[ \sum_{i\in[n_*]} d_{\mcx}(v_i,w_{\pi(i)})^p +\frac{1}{2} \sum\limits_{(i,i')\in [n_*]^2} d_{\mcy}(e_{ii'},f_{\pi(i)\pi(i')})^p\biggr]^{1/p}\\
    &\leq \biggl[ \sum_{i\in[n_*]} \bigl(d_{\mcx}(v_i,z_{\pi_1(i)})+d_{\mcx}(z_{\pi_1(i)},y_{\pi_2(\pi_1(i))})\bigr)^p\\
    &\hspace{10mm}+\frac{1}{2} \sum\limits_{(i,i')\in [n_*]^2} \bigl(d_{\mcy}(e_{ii'},g_{\pi_1(i)\pi_1(i')})+d_{\mcy}(g_{\pi_1(i)\pi_1(i')},f_{\pi_2(\pi_1(i))\pi_2(\pi_1(i'))})\bigr)^p\biggr]^{1/p}\\
    &\leq 
    d_{\mathbb{G},A}(([m],v,e), ([l],u,g))+ d_{\mathbb{G},A}(([l],u,g), ([n],w,f)).
\end{align*}
\vspace*{-12mm}

\hspace*{2mm}
\end{proof}

\begin{proof}[Proof of Propostion~\ref{prop: GTT is special case fo GED }]
The graph edit distance seeks for an optimal inexact match between to graphs, i.e. a collection of edit operations that transforms one graph into the other with minimal costs. Valid edit operations are the deletion, insertion and substitution of vertices and edges or their labels and attributes.
By defining the cost of these edit operations below, we show that the GTT metric is a special case of the GED distance. 

Let $([n],v,e),$ $([m],w,f)\in\mathbb{G}$ be two simple attributed graphs and consider the corresponding graphs $([m+n],v,e),$ $([m+n],w,f)$ that are filled up with respect to the GTT metric (see Propostion~\ref{prop: GTT as optimal matching (filled to n+m)}). 
Then the optimal permutation $\pi_*\in S_{m+n}$ corresponds to an optimal inexact match between the initial graphs. 
More precisely, matching a vertex $v_i$ to an auxiliary vertex, i.e. $\pi_*(i)\in [m+n]\setminus [n]$ for $i\in[m]$, corresponds to the deletion of the vertex $v_i$. Analogously, matching a vertex $w_j$ to an auxiliary vertex, i.e. $\pi^{-1}_*(j)\in [m+n]\setminus [m]$ for $j\in[n]$, corresponds to the insertion of the vertex $w_j$. Matching a vertex with an auxiliary vertex results in a constant cost in the GTT metric, given by the vertex penalty $C$. Hence, we set the cost for deleting or inserting a vertex to $\text{DELNODE}(v_i) = C$ and $\text{INSNODE}(w_j) = C$, respectively. 
The matching of two non-auxiliary vertices $v_i$ and $w_j$, i.e. $\pi_*(i) = j\in [n]$ for $i\in[m]$, complies with the substitution of $v_i$ by $w_j$. The cost of matching of two non-auxiliary vertices is given by the underlying metric $d_{\mcx}$. Consequently, we define the cost for substituting a vertex~$v_i$ with a vertex~$w_j$ to be given by $\text{SUBNODE}(v_i, w_j) = d_{\mcx}(v_i,w_j)$. 
On the other hand, matching two auxiliary vertices with each other is equivalent to substituting them. 
Since both vertices coincide, the cost of such a matching (edit operation) is zero.

It remains to define the cost of edge related edit operations. We distinguish two different types. First, edge related edit operations that are independent of vertex transformations and second, edge related edit operations that are induced by vertex related edit operations.
In the GTT metric only edit operations of the second type are considered. Hence, we set the cost of edge edit operations of the first type to be infinity, i.e. $\text{DELBRANCH}(e_{ii'}) = \text{INSBRANCH}(f_{jj'}) =  \infty$ for all $i,i',j,j'\in [m+n]$ and $\text{SUBBRANCH}(e_{ii'}, f_{jj'}) = \infty$ whenever $\pi_*(i)\neq j$ or $\pi_*(i')\neq j'$. 
We now consider edge edit operations of the second type. The "matching" of an edge $e_{ii'}\neq y_0$ with an no-edge element $f_{\pi_*(i)\pi_*(i')} = y_0$ corresponds to the deletion of the edge $e_{ii'}$. This results in cost $d_{\mcy}(e_{ii'}, y_0)$ in the GTT metric. Hence, we set the cost for deleting an edge $e_{ii'}\neq y_0$ to $\text{DELBRANCH}'(e_{ii'}) = d_{\mcy}(e_{ii'}, y_0)$. Analogously, "matching" an no-edge element $e_{ii'} = y_0$ with an edge $f_{\pi_*(i)\pi_*(i')} \neq y_0$ corresponds to the insertion of an edge $f_{\pi_*(i)\pi_*(i')}$. Again, the cost of such an edit operation is defined in a consistent way by setting $\text{INSBRANCH}'(f_{\pi_*(i)\pi_*(i')}) = d_{\mcy}(y_0,f_{\pi_*(i)\pi_*(i')})$. 
Finally, "matching" of two edges $e_{ii'}\neq y_0$ and $f_{\pi_*(i)\pi_*(i')}\neq y_0$ complies with the substitution of an edge $e_{ii'}$ by an edge $f_{\pi_*(i)\pi_*(i')}$. Again, the cost in the GTT metric is given by the underlying metric $d_{\mcy}$. Consequently, we define the cost for substituting an edge~$e_{ii'}$ with an edge~$f_{\pi_*(i)\pi_*(i')}$ to be given by $\text{SUBBRANCH}(e_{ii'}, f_{\pi_*(i)\pi_*(i')}) = d_{\mcx}(e_{ii'},f_{\pi_*(i)\pi_*(i')})$. 

This yields the claim since an optimal permutation complies with an optimal inexact match and the cost of the inexact match are defined in a GTT metric consistent way. 
\end{proof}

\subsection{Proofs for the GOSPA metrics} \label{ssec: proofs left out for gospa}

\begin{proof}[Proof of Proposition~\ref{prop: proof of GOSPA1 and GOSPA2 filling up to same size does not change}]
We start with the statement for the GOSPA1 metric.
Recall that we filled up the graph $([m],v,e)$ with auxiliary elements $x_*$ and $y_*$ in such a way that $d_{\mcx}(v_i,w_{\pi(i)})^p = d_{\mcx}(x_*,w_{\pi(i)})^p = C_1^p -\frac12 C_{\mcy}^p =: \overline{C}^p_{\mcx}$ for all $m +1 \leq i\leq n$ and
$d_{\mcy}(e_{ii'}, f_{\pi(i)\pi(i')}) = d_{\mcy}(y_*, f_{\pi(i)\pi(i')}) = C_{\mcy}$ for $m+1\leq i \vee i' \leq n$ with $i\neq i'$. Hence we obtain
\begin{align*}
&d_{\mathbb{G},R_1}(([m],v,e),([n],w,f))\\
&= \frac{1}{n^{1/p}}\min\limits_{\pi \in S_{n}} \biggl[ (n-m) \overline{C}^p_{\mcx} + \sum_{i\in [m]} d_{\mcx}(v_i, w_{\pi(i)})^p\\
&\hspace*{6mm} {}+ \frac12 \frac{1}{n-1} \biggl(\sum_{i\in [m]} \biggl(  \sum\limits_{\substack{i'\in [m] \\ i\neq i'}} d_{\mcy}(e_{ii'},f_{\pi(i)\pi(i')})^p +  (n-m) C_{\mcy}^p \biggr) + (n-m)\,(n-1)\, C_{\mcy}^p\biggr)\biggr]^{1/p}\\
&=\frac{1}{n^{1/p}}\min\limits_{\pi \in S_{n}} \biggl[ \sum_{i\in [n]} d_{\mcx}(v_i, w_{\pi(i)})^p
+ \frac12 \frac{1}{n-1} \biggl( \sum_{(i,i')\in[n]^2} d_{\mcy}(e_{ii'}, f_{\pi(i)\pi(i')})^p \biggr)\biggr]^{1/p}\\
&= d_{\mathbb{G},R_1}(([n],v,e),([n],w,f)).
\end{align*}
Now consider the GOSPA2 metric. We filled-up the graph $([m],v,e)$ such that $d_{\mcx}(v_i,w_{\pi(i)})^p = d_{\mcx}(x_*,w_{\pi(i)})^p = C_2^p$ for all $m +1 \leq i\leq n$ and
$d_{\mcy}(e_{ii'}, f_{\pi(i)\pi(i')}) = d_{\mcy}(y_0, f_{\pi(i)\pi(i')})$ for $m+1\leq i \vee i' \leq n$. 
This implies
\begin{align*}
&d_{\mathbb{G},R_2}(([m],v,e),([n],w,f))\\
 &= \frac{1}{n^{1/p}}\min\limits_{\pi \in S_{n}} \biggl[ (n-m) C_2^p + \sum_{i\in[m]} d_{\mcx}(v_i, w_{\pi(i)})^p \\
&\hspace*{6mm} {} + \frac12 \frac{1}{n-1} \biggl( \sum_{(i,i')\in [m]^2} d_{\mcy}(e_{ii'}, f_{\pi(i)\pi(i')})^p +   \sum_{(i,i')\in [n]^2\setminus[m]^2} d_{\mcy}(y_0, f_{\pi(i)\pi(i')})^p \biggr) \biggr]^{1/p}\\
&=\frac{1}{n^{1/p}}\min\limits_{\pi \in S_{n}} \biggl[ \sum_{i\in[n]} d_{\mcx}(v_i, w_{\pi(i)})^p
+ \frac12 \frac{1}{n-1}  \sum_{(i,i')\in [n]^2} d_{\mcy}(e_{ii'}, f_{\pi(i)\pi(i')})^p \biggr]^{1/p}\\
&= d_{\mathbb{G},R_2}(([n],v,e),([n],w,f)).
\end{align*}
\vspace*{-12mm}

\hspace*{2mm}
\end{proof}

\begin{proof}[Proof of Theorem~\ref{thm: graph-OSPA are metric}]
We first show the identity and symmetry properties for both GOSPA metrics together.

For the identity, assume that $([m],v,e) = ([n],w,f)$. Then $n=m$ as $C_1,C_2>0$ and there exists a $\pi\in S_n$ such that $v_i = w_{\pi(i)}$ and $e_{ii'} = f_{\pi(i)\pi(i')}$ for all $i,i'\in [n]$. This yields $d_{\mcx}(v_i, w_{\pi(i)}) =0$ and $d_{\mcy}(e_{ii'},f_{\pi(i)\pi(i')})=0$ for all $i,i'\in [n]$ and hence $d_{\mathbb{G},R_k}(([m],v,e),([n],w,f))=0$ for $k \in \{1,2\}$.

In the case where $d_{\mcx}$, $d_{\mcy}$ are metrics, it remains to show that $d_{\mathbb{G},R_k}$ can distinguish between different graphs. Indeed, $d_{\mathbb{G},R_k}(([m],v,e),([n],w,f)) = 0$ implies $n=m$ as $C_k > 0$, $k \in \{1,2\}$. It implies further that there exists a $\pi\in S_n$ such that $d_{\mcx}(v_i, w_{\pi(i)}) =0$ and $d_{\mcy}(e_{ii'},f_{\pi(i)\pi(i')})=0$ for all $i,i'\in [n]$. This yields $([m],v,e)=([n],w,f)$ and thus the identity property for the case where $d_{\mcx}$, $d_{\mcy}$ are metrics.

The symmetry property follows from the construction and the symmetry properties of $d_{\mcx}$ and~$d_{\mcy}$. 

To show the triangle inequality, consider spatial graphs $([m],v,e)$, $([n],w,f)$, $([l],u,g)\in \mathbb{G}$. We start with the special case where $m=n=l$. We can still treat both GOSPA metrics together as their values agree if both graphs have the same number of vertices. Since $n=0$ is trivial, assume that $n\geq 1$ and recall that we interpret $0/0$ as $0$, so there is no contribution of an edge sum if $n=1$. 
Let $\pi_1$ and $\pi_2$ be the optimal solutions of $d_{\mathbb{G},R_1}(([m],v,e),([l],u,g))$ and $d_{\mathbb{G},R_1}(([l],u,g),([n],w,f))$, respectively, and take $\pi = \pi_2\circ \pi_1$. Then the metric properties of $d_{\mcx}$ and $d_{\mcy}$ and the Minkowski inequality yield
    \begin{align*}
        &d_{\mathbb{G},R_1}(([m],v,e),([n],w,f)) = d_{\mathbb{G},R_2}(([m],v,e),([n],w,f))\\
        &\leq \frac{1}{n^{1/p}} \biggl[ \sum_{i\in[n]} d_{\mcx}(v_i, w_{\pi(i)})^p
        + \frac12 \frac{1}{n-1} \sum_{(i,i')\in [n]^2} d_{\mcy}(e_{ii'}, f_{\pi(i)\pi(i')})^p \biggr]^{1/p}\\
        &\leq \frac{1}{n^{1/p}} \biggl[ \sum_{i\in[n]} \bigl(d_{\mcx}(v_i,z_{\pi_1(i)}) + d_{\mcx}( z_{\pi_1(i)},y_{\pi_2(\pi_1(i))})\bigr)^p\\
        &\hspace*{6mm} {} + \frac12 \frac{1}{n-1} \sum_{(i,i')\in [n]^2}  \bigl(d_{\mcy}(e_{ii'}, g_{\pi_1(i)\pi_1(i')}) + d_{\mcy}(g_{\pi_1(i)\pi_1(i')},f_{\pi_2(\pi_1(i))\pi_2(\pi_1(i'))})\bigr)^p \biggr]^{1/p}\\
        &\overset{}{\leq} \frac{1}{n^{1/p}} \biggl[ \sum_{i\in[n]} d_{\mcx}(v_i,z_{\pi_1(i)})^p + \frac12 \frac{1}{n-1} \sum_{(i,i')\in [n]^2} d_{\mcy}(e_{ii'}, g_{\pi_1(i)\pi_1(i')})^p \biggr]^{1/p}\\
        &\hspace*{6mm} {} + \frac{1}{n^{1/p}} \biggl[ \sum_{i\in[n]} d_{\mcx}( z_{\pi_1(i)},y_{\pi_2(\pi_1(i))})^p+ \frac12 \frac{1}{n-1} \sum_{(i,i')\in [n]^2}  d_{\mcy}(g_{\pi_1(i)\pi_1(i')},f_{\pi_2(\pi_1(i))\pi_2(\pi_1(i'))})^p\biggr]^{1/p}\\
        &= d_{\mathbb{G},R_2}(([m],v,e),([l],u,g)) + d_{\mathbb{G},R_2}(([l],u,g),([n],w,f))\\
        &= d_{\mathbb{G},R_1}(([m],v,e),([l],u,g)) + d_{\mathbb{G},R_1}(([l],u,g),([n],w,f)).
    \end{align*}

We now turn to the general case, i.e.\ we allow the graphs to be of different size. Due to the symmetry of the triangle inequality in $([m],v,e)$ and $([n],w,f)$, we assume w.l.o.g.\ that $m\leq n$. We assume further that at most one of the graphs is empty (otherwise the triangle inequality is trivial), which can only happen if $m=0$ or $l=0$ and $l<n$. These cases are included below, as well as any cases involving graphs of size 1 if notation is interpreted in a benevolent way, as detailed in Definitions~\ref{def: graph-OSPA}. 

In the four parts that follow, we separate the two GOSPA metrics and distinguish the cases $n>l$ and $n\leq l$. Each time we fill up the two smaller graphs to the size of the largest graph according to Proposition~\ref{prop: proof of GOSPA1 and GOSPA2 filling up to same size does not change} in a way that is consistent with the respective penalty.
\bigskip

\noindent
\emph{GOSPA1, $n>l$.} We fill up the graphs $([m],v,e)$ and $([l],u,g)$ to size $n$ based on the extended (pseudo-)metric spaces from~Proposition~\ref{prop: proof of GOSPA1 and GOSPA2 filling up to same size does not change}(a). Recall in particular that $d_{\mcx}(x_*,x)^p = d_{\mcx}(x,x_*)^p = \overline{C}^p_{\mcx} = C_1^p -\frac12 C_{\mcy}^p$ for all $x \in \mcx$ and $d_{\mcy}(y_*,y) = d_{\mcy}(y,y_*) = C_{\mcy}$ for all $y \in \mcy$. Set $v_i = x_*$, $e_{ii} = y_0$ for $m+1\leq i\leq n$ and $e_{ii'} = y_*$ for $m+1\leq i \vee i' \leq n$ with $i\neq i'$, as well as $u_i = x_*$, $g_{ii} = y_0$ for $l+1\leq i\leq n$ and $g_{ii'} = y_*$ for $l+1\leq i \vee i' \leq n$ with $i\neq i'$. 
Lemma~\ref{le: proof of GOSPA1 extending both graphs reduces size} below states for two graphs $([n_1],v,e)$ and $([n_1],u,g)$ of the same size $n_1 \leq n$ that this extension does not increase the metric, i.e.
\begin{equation} \label{eq:GOSPA1_lem1}
  d_{\mathbb{G},R_1}(([n],v,e),([n],u,g)) \leq d_{\mathbb{G},R_1}(([n_1],v,e), ([n_1],u,g)).
\end{equation}

Setting $n_1 = \max\{m,l\}$, we apply Proposition~\ref{prop: proof of GOSPA1 and GOSPA2 filling up to same size does not change} (three times) and Inequality~\eqref{eq:GOSPA1_lem1} to obtain
\begin{align*}
    d_{\mathbb{G},R_1}(([m],v,e),([n],w,f)) &=d_{\mathbb{G},R_1}(([n],v,e),([n],w,f)) \\
    &\leq d_{\mathbb{G},R_1}(([n],v,e),([n],u,g)) + d_{\mathbb{G},R_1}(([n],u,g),([n],w,f))\\
    &\leq d_{\mathbb{G},R_1}(([n_1],v,e),([n_1],u,g)) + d_{\mathbb{G},R_1}(([n],u,g),([n],w,f))\\
    &=d_{\mathbb{G},R_1}(([m],v,e),([l],u,g)) + d_{\mathbb{G},R_1}(([l],u,g),([n],w,f)),
\end{align*}
where the first inequality follows from the special case where all graphs are of the same size.
\bigskip
        
\noindent
\emph{GOSPA1, $n \leq l$.} Here we extend the (pseudo-)metric spaces $\mcx$ and $\mcy$ by adding two points each and extending the corresponding maps $d_{\mcx}$, $d_{\mcy}$. Add $x_*, x'_*$ to $\mcx$ satisfying $d_{\mcx}(x_*,x)^p = d_{\mcx}(x'_*,x)^p = d_{\mcx}(x_*,x'_*)^p = \overline{C}^p_{\mcx} = C_1^p-\frac{1}{2}C_{\mcy}^p$ for all $x \in \mcx$ and add $y_*, y'_*$ to $\mcy$ satisfying $d_{\mcy}(y_*,y) = d_{\mcy}(y'_*,y) = d_{\mcy}(y_*, y'_*) = C_{\mcy}$ for all $y \in \mcy$. It is readily checked that the extended maps  $d_{\mcx}$ and $d_{\mcy}$ are still (pseudo-)metrics.

We now fill up the graph $([m],v,e)$ in the analogous way to Proposition~\ref{prop: proof of GOSPA1 and GOSPA2 filling up to same size does not change}(a). Set $v_i = x_*$ and $e_{ii} = y_0$ for $m+1\leq i\leq l$ and $e_{ii'} = y_*$ for $m+1\leq i \vee i' \leq l$ with $i\neq i'$. Analogously, we fill up the graph $([n],w,f)$ but now using the auxiliary elements $x'_*$ and $y'_*$, i.e.\ we let $w_i = x'_*$ and $f_{ii}=y_0$ for $n+1\leq i\leq l$ and $f_{ii'} = y'_*$ for $n+1\leq i \vee i' \leq l$ with $i\neq i'$.

Lemma~\ref{le:  proof of GOSPA1 extending graphs with different elements increases size} below yields that extending the larger graph to a larger size does not decrease the metric, i.e.
\begin{equation} \label{eq:GOSPA1_lem2}
  d_{\mathbb{G},R_1}(([m],v,e),([n],w,f)) \leq d_{\mathbb{G},R_1}(([m],v,e),([l],w,f)).
\end{equation}

Note that Proposition~\ref{prop: proof of GOSPA1 and GOSPA2 filling up to same size does not change} still holds true when the original spaces are $\mcx \cup \{x_*\}$ and $\mcy \cup \{y_*\}$ (or $\mcx \cup \{x'_*\}$ and $\mcy \cup \{y'_*\}$) and the extensions are performed on $\mcx \cup \{x_*, x'_*\}$ and $\mcy \cup \{y_*, y'_*\}$. The triangle inequality for $n\leq l$ is therefore obtained by applying Inequality~\eqref{eq:GOSPA1_lem2}, Proposition~\ref{prop: proof of GOSPA1 and GOSPA2 filling up to same size does not change} (three times), as well as the triangle inequality for the special case $m=n=l$, as follows:
\begin{align*}
    d_{\mathbb{G},R_1}(([m],v,e),([n],w,f)) 
    &\leq d_{\mathbb{G},R_1}(([m],v,e),([l],w,f))\\
    &= d_{\mathbb{G},R_1}(([l],v,e),([l],w,f))\\
    &\leq d_{\mathbb{G},R_1}(([l],v,e),([l],u,g)) + d_{\mathbb{G},R_1}(([l],u,g),([l],w,f))\\
    &=d_{\mathbb{G},R_1}(([m],v,e),([l],u,g)) + d_{\mathbb{G},R_1}(([l],u,g),([n],w,f)).
\end{align*}

\noindent
\emph{GOSPA2, $n>l$.}
We fill up the graphs $([m],v,e)$ and $([l],u,g)$ to $n$ based on the extended (pseudo-)metric space of vertex attributes from Proposition~\ref{prop: proof of GOSPA1 and GOSPA2 filling up to same size does not change}(b). Recall in particular that $d_{\mcx}(x_*,x) = d_{\mcx}(x,x_*) = C_2$ for all $x \in \mcx$. Set $v_i = x_*$ for $m+1\leq i\leq n$ and $e_{ii'} = y_0$ for $m+1\leq i \vee i' \leq n$, as well as $u_i = x_*$ for $l+1\leq i\leq n$ and $g_{ii'} = y_0$ for $l+1\leq i \vee i' \leq n$. 

We show in Lemma~\ref{le: proof of GOSPA2 extending both graphs reduces size} that corresponding extensions of two graphs of the same size $n_1$ to a larger size $n$ do not increase the metric, i.e.
\begin{equation} \label{eq:GOSPA2_lem1}
  d_{\mathbb{G},R_2}(([n],v,e),([n],u,g)) \leq d_{\mathbb{G},R_2}(([n_1],v,e), ([n_1],u,g)).
\end{equation}

Applying Proposition~\ref{prop: proof of GOSPA1 and GOSPA2 filling up to same size does not change} (three times), the triangle inequality for the special case $m = n = l$ and Inequality~\eqref{eq:GOSPA2_lem1} with $n_1 = \max\{m,l\}$ yields 
\begin{align*}
    d_{\mathbb{G},R_2}(([m],v,e),([n],w,f)) &= d_{\mathbb{G},R_2}(([n],v,e),([n],w,f)) \\
    &\leq d_{\mathbb{G},R_2}(([n],v,e),([n],u,g)) + d_{\mathbb{G},R_2}(([n],u,g),([n],w,f))\\ &\leq d_{\mathbb{G},R_2}(([n_1],v,e),([n_1],u,g)) + d_{\mathbb{G},R_2}(([n],u,g)),([n],w,f)) \\
    &= d_{\mathbb{G},R_2}(([m],v,e),([l],u,g)) + d_{\mathbb{G},R_2}(([l],u,g),([n],w,f)).
        \end{align*}

\noindent
\emph{GOSPA2, $n\leq l$.}
Here we add two points $x_*,x'_*$ to the space $\mcx$ and extend the corresponding (pseudo-)metric $d_{\mcx}$. Compared to the corresponding GOSPA1 case, we have to set up the new distances more carefully. Define $d_{\mcx}(x_*,x)^p = d_{\mcx}(x,x_*)^p = \overline{C}^p_{\mcx} := C_2^p -C_{\mcy}^p$ and $d_{\mcx}(x'_*,x)^p= d_{\mcx}(x,x'_*)^p = d_{\mcx}(x_*,x'_*)^p= C_2^p$ for all $x\in\mcx$. Clearly the extended $d_{\mcx}$ is still a (pseudo-)metric.

We fill up the graph $([m],v,e)$ by setting $v_i = x_*$ for $m+1\leq i\leq l$ and $e_{ii'} = y_0$ for $m+1\leq i \vee i' \leq l$. Analogously we extend the graph $([n],w,f)$ by letting $w_i = x'_*$ for $n+1\leq i\leq l$ and $f_{ii'} = y_0$ for $n+1\leq i \vee i' \leq l$. \\ 

With regard to Proposition~\ref{prop: proof of GOSPA1 and GOSPA2 filling up to same size does not change}(b), the point $x_*$ is too close to allow an extension from $\mcx \cup \{x'_*\}$ that maintains the metric. However, it is easily seen from the proof of this proposition that the amount by which it is too close goes into an additive constant, so that
\begin{equation}\label{eq: Equality filling up smaller to larger size with constant in OSPACE}
\biggl(d_{\mathbb{G},R_2}(([l],v,e),([l],u,g))^p + C_{\mcy}^p \frac{(l-m)}{l}\biggr)^{1/p} = d_{\mathbb{G},R_2}(([m],v,e),([l],u,g)).    
\end{equation}
Since $x'_*$ is not too close to $\mcx \cup \{x_*\}$ such a constant is not needed for the extension of $([n],w,f)$ i.e.\ we have
\begin{equation} \label{eq: more filling}
  d_{\mathbb{G},R_2}(([l],u,g),([l],w,f)) = d_{\mathbb{G},R_2}(([l],u,g),([n],w,f)).
\end{equation}

By Lemma~\ref{le: GOSPA2 extending with different elements} below, we obtain
\begin{equation} \label{eq:GOSPA2_lem2}
  d_{\mathbb{G},R_2}(([m],v,e),([n],w,f)) \leq \biggl(d_{\mathbb{G},R_2}(([l],v,e),([l],w,f))^p + C_{\mcy}^p \frac{(n-m)}{n}\biggr)^{1/p},
\end{equation}
i.e.\ extending \emph{both} graphs to a larger size does not decrease the metric by more than the above mentioned constant (adapted to new graph cardinalities).

Combining Inequality~\eqref{eq:GOSPA2_lem2}, Equations~\eqref{eq: Equality filling up smaller to larger size with constant in OSPACE} and~\eqref{eq: more filling}, and the triangle inequality for the special case $m =n = l$, we obtain 
\begin{align*}
    d_{\mathbb{G},R_2}(&([m],v,e),([n],w,f))\\
    &\leq \biggl(d_{\mathbb{G},R_2}(([l],v,e),([l],w,f))^p + C_{\mcy}^p \frac{(n-m)}{n}\biggr)^{1/p}\\
    &\leq \biggl(d_{\mathbb{G},R_2}(([l],v,e),([l],u,g))^p + C_{\mcy}^p \frac{(n-m)}{n}\biggr)^{1/p} + d_{\mathbb{G},R_2}(([l],u,g),([l],w,f)) \\
    &\leq \biggl(d_{\mathbb{G},R_2}(([l],v,e),([l],u,g))^p + C_{\mcy}^p \frac{(l-m)}{l} \biggr)^{1/p}+ d_{\mathbb{G},R_2}(([l],u,g),([l],w,f))\\
    &= d_{\mathbb{G},R_2}(([m],v,e),([l],u,g)) + d_{\mathbb{G},R_2}(([l],u,g),([n],w,f)).
\end{align*}
\vspace*{-12mm}

\hspace*{2mm}
\end{proof}

In what follows we present the four lemmas used in the proof of Theorem~\ref{thm: graph-OSPA are metric}.
\begin{lemma}\label{le: proof of GOSPA1 extending both graphs reduces size} Let $([n_1],v,e), ([n_1],u,g) \in \mathbb{G}$. Construct extensions $([n],v,e)$ and $([n],u,g)$ with $n_1 \leq n$ as in the proof of Theorem~\ref{thm: graph-OSPA are metric} for the case GOSPA1, $n > l$. Then
$$d_{\mathbb{G},R_1}(([n],v,e),([n],u,g)) \leq d_{\mathbb{G},R_1}(([n_1],v,e), ([n_1],u,g)).$$
\end{lemma}
\begin{proof}
Since both graphs are filled up with the same elements, i.e.\ $v_i = u_i$ for $n_1+1 \leq i\leq n$ and $e_{ii'} = g_{ii'}$ for $n_1+1 \leq i \vee i' \leq n$, we obtain
        \begin{align*}
            d_{\mathbb{G},R_1}(&([n],v,e),([n],u,g))\\
            &\leq  \frac{1}{n^{1/p}}\min\limits_{\substack{\pi \in S_{n}\\ \pi([n_1]) = [n_1]}} \biggl[\sum_{i\in[n]} 
             d_{\mcx}(v_i, u_{\pi(i)})^p  + \frac12 \frac{1}{n-1} \sum_{(i,i')\in[n]^2} d_{\mcy}(e_{ii'},g_{\pi(i)\pi(i')})^p   \biggr]^{1/p}\\
            &= \frac{1}{n^{1/p}}\min\limits_{\pi \in S_{n_1}} \biggl[\sum_{i\in[n_1]} 
             d_{\mcx}(v_i, u_{\pi(i)})^p  + \frac12 \frac{1}{n-1} \sum_{(i,i')\in[n_1]^2} d_{\mcy}(e_{ii'},g_{\pi(i)\pi(i')})^p   \biggr]^{1/p}\\
            &\leq d_{\mathbb{G},R_1}(([n_1],v,e),([n_1],u,g)).
        \end{align*}
\vspace*{-12mm}

\hspace*{2mm}
\end{proof}

\begin{lemma}\label{le:  proof of GOSPA1 extending graphs with different elements increases size} Let $([m],v,e), ([n],w,f) \in \mathbb{G}$. Construct the extension $([l],w,f)$ of $([n],w,f)$ as in the proof of Theorem~\ref{thm: graph-OSPA are metric} for the case GOSPA1, $n \leq l$. Then
    $$d_{\mathbb{G},R_1}(([m],v,e),([n],w,f)) \leq d_{\mathbb{G},R_1}(([m],v,e),([l],w,f)).$$
\end{lemma}
\begin{proof}
We remark that $a\leq nC$ implies $\frac{a}{n} \leq \frac{(l-n) C +a }{l}$. Using this together with the fact that the metrics $d_{\mcx}^p$ and $d_{\mcy}^p$ are bounded by
$\overline{C}^p_1$ and $C_{\mcy}^p$, respectively, we obtain
        \begin{align*}
            d_{\mathbb{G},R_1}(&([m],v,e),([n],w,f))\\
            &=\frac{1}{n^{1/p}}\min\limits_{\pi \in S_{n}} \biggl[ (n-m) C_1^p + \sum_{i\in[m]}  d_{\mcx}(v_i, w_{\pi(i)})^p\\
            & \hspace*{6mm} + \frac12  \frac{1}{n-1}\sum_{i\in[m]}\biggl(\sum_{i'\in[m]}  d_{\mcy}(e_{ii'}, f_{\pi(i)\pi(i')})^p + (n-m) C_{\mcy}^p\biggr)    \biggr]^{1/p}\\
            &\leq \frac{1}{n^{1/p}}\min\limits_{\pi \in S_{n}} \biggl[ (n-m) C_1^p + \sum_{i\in[m]}  d_{\mcx}(v_i, w_{\pi(i)})^p\\
            & \hspace*{6mm}+  
             \frac12 \frac{1}{l-1}\sum_{i\in[m]} \biggl(\sum_{i'\in[m]} d_{\mcy}(e_{ii'}, f_{\pi(i)\pi(i')})^p +(l-n)C_{\mcy}^p + (n-m) C_{\mcy}^p \biggr)\biggr]^{1/p}\\
            &\leq \frac{1}{l^{1/p}}\min\limits_{\pi \in S_{n}} \biggl[ (l- n) C_1^p + (n-m) C_1^p + \sum_{i\in[m]} 
             d_{\mcx}(v_i, w_{\pi(i)})^p\\
            & \hspace*{6mm}+ 
             \frac12 \frac{1}{l-1} \sum_{i\in[m]} \biggl( \sum_{i'\in[m]} d_{\mcy}(e_{ii'}, f_{\pi(i)\pi(i')})^p  +(l-m) C_{\mcy}^p\biggr)  \biggr]^{1/p}\\
            &=\frac{1}{l^{1/p}}\min\limits_{\pi \in S_{l}} \biggl[ (l-m) C_1^p + \sum_{i\in[m]} 
             d_{\mcx}(v_i, w_{\pi(i)})^p\\
            & \hspace*{6mm}  + \frac12 \frac{1}{l-1}\sum_{i\in[m]}  \biggl(\sum_{i'\in[m]}   d_{\mcy}(e_{ii'}, f_{\pi(i)\pi(i')})^p+(l-m) C_{\mcy}^p  \biggr)  \biggr]^{1/p}\\
            &= d_{\mathbb{G},R_1}(([m],v,e),([l],w,f)).
        \end{align*}
The second last equality follows from the fact that there is a minimizer $\pi_0\in S_l$ for the brackets on the right hand side that satisfies $\pi_0([n]) \subset [n]$. To see this, let $\pi_*\in S_l$ be any minimizer for these brackets. Assume that there exists an $i\in[n]$ such that $\pi_*(i)\in[l]\setminus [n]$. By construction of the extended graph we have $w_{\pi_*(i)} = x'_*$ and $f_{\pi_*(i)\pi_*(i')} = y'_*$ for all $i'\in[n]$, $i' \neq i$, and thus 
        \begin{align*}
            d_{\mcx}(&v_i, w_{\pi_*(i)})^p  + \frac12 \frac{1}{l-1} \biggl(  \sum\limits_{\substack{i'\in[m] \\ i \neq i'}} d_{\mcy}(e_{ii'}, f_{\pi_*(i)\pi_*(i')})^p  + (l-m) C_{\mcy}^p \biggr)\\
            &= d_{\mcx}(v_i, x'_*)^p  + \frac12 \frac{1}{l-1} \biggl(   \sum\limits_{\substack{i'\in[m] \\ i \neq i'}} d_{\mcy}(e_{ii'},y'_*)^p  + (l-m) C_{\mcy}^p\biggr)\\
            &\geq d_{\mcx}(v_i, w_{\pi(i)})^p  + \frac12 \frac{1}{l-1} \biggl(  \sum\limits_{\substack{i'\in[m] \\ i \neq i'}} d_{\mcy}(e_{ii'}, f_{\pi(i)\pi(i')})^p  + (l-m) C_{\mcy}^p \biggr)
        \end{align*}
        for any $\pi\in S_{n}$. Thus there exists an optimal solution $\pi_0\in S_l$ such that $\pi_0([n]) \subset [n]$.
\end{proof}

\begin{lemma}\label{le: proof of GOSPA2 extending both graphs reduces size}
Let $([n_1],v,e), ([n_1],u,g) \in \mathbb{G}$. Construct extensions $([n],v,e)$ and $([n],u,g)$ with $n \geq n_1$ as in the proof of Theorem~\ref{thm: graph-OSPA are metric} for the case GOSPA2, $n>l$. Then
        \begin{equation}
            d_{\mathbb{G},R_2}(([n],v,e),([n],u,g)) \leq d_{\mathbb{G},R_2}(([n_1],v,e), ([n_1],u,g)).
        \end{equation}
\end{lemma}
\begin{proof}
As the extended graphs are filled up with the same auxiliary element, we obtain
        \begin{align*}
            &d_{\mathbb{G},R_2}(([n],v,e),([n],u,g))\\
            &\leq \frac{1}{n^{1/p}}\min\limits_{\substack{\pi \in S_{n}\\ \pi([n_1]) = [n_1]}} \biggl[
             \sum_{i\in[n]} d_{\mcx}(v_i, u_{\pi(i)})^p
            + \frac12 \frac{1}{n-1} \sum_{(i,i')\in[n]^2} d_{\mcy}(e_{ii'}, g_{\pi(i)\pi(i')})^p \biggr)\biggr]^{1/p}\\
            &= \frac{1}{n^{1/p}}\min\limits_{\pi \in S_{n_1}} \biggl[ \sum_{i\in[n_1]} d_{\mcx}(v_i, u_{\pi(i)})^p
            + \frac12 \frac{1}{n-1} \sum_{(i,i')\in[n_1]^2} d_{\mcy}(e_{ii'}, g_{\pi(i)\pi(i')})^p \biggr]^{1/p}\\
            &\leq d_{\mathbb{G},R_2}(([n_1],v,e), ([n_1],u,g)).
        \end{align*}
        \vspace*{-12mm}

\hspace*{2mm}
\end{proof}

\begin{lemma} \label{le: GOSPA2 extending with different elements}
Suppose that $p=1$ or $p>1$ and $d_{\mcy}(y_1,y_2) \leq \max\{d_{\mcy}(y_1,y_0),d_{\mcy}(y_0,y_2)\}$ for all $y_1,y_2\in\mcy$. Let $([m],v,e), ([n],w,f) \in \mathbb{G}$. Construct extensions $([l],v,e)$ and $([l],w,f)$ as in the proof of Theorem~\ref{thm: graph-OSPA are metric} for the case GOSPA2, $n \leq l$. Then
$$d_{\mathbb{G},R_2}(([m],v,e),([n],w,f))  \leq \biggl(d_{\mathbb{G},R_2}(([l],v,e),([l],w,f))^p + C_{\mcy}^p \frac{(n-m)}{n}\biggr)^{1/p}.$$
\end{lemma}
\begin{proof}
Let $\overline{C}_{\mcy}^p := C_{\mcy}^p \frac{(n-m)}{n}$. By Equation~\eqref{eq: Equality filling up smaller to larger size with constant in OSPACE} and the fact that $a\leq nC$ implies $\frac{a}{n} \leq \frac{(l-n) C + a}{l}$, we obtain
        \begin{align*}
        &d_{\mathbb{G},R_2}(([m],v,e),([n],w,f))= \biggl(d_{\mathbb{G},R_2}(([n],v,e),([n],w,f))^p + \overline{C}^p_{\mcy} \biggr)^{1/p} \\
        &=\min\limits_{\pi \in S_{n}} \biggl[ \frac{1}{n} \sum_{i\in[n]} d_{\mcx}(v_i, w_{\pi(i)})^p 
        + \frac12\frac{1}{n} \sum_{i\in[n]} \frac{1}{n-1} \sum_{\substack{i'\in[n] \\ i \neq i'}} \underbrace{ d_{\mcy}(e_{ii'}, f_{\pi(i)\pi(i')})^p}_{\leq \, C_{\mcy}^p} + \overline{C}^p_{\mcy} \biggr]^{1/p}\\
        &\leq\min\limits_{\pi \in S_{n}} \biggl[ \frac{1}{n} \sum_{i\in[n]} \underbrace{d_{\mcx}(v_i, w_{\pi(i)})^p}_{\leq  \overline{C}^p_{\mcx}}
        + \frac12\frac{1}{n} \sum_{i\in[n]} \underbrace{\frac{1}{l-1} \biggl( \sum_{\substack{i'\in[n] \\ i \neq i'}} d_{\mcy}(e_{ii'}, f_{\pi(i)\pi(i')})^p + (l-n) C_{\mcy}^p \biggr)}_{\leq \, C_{\mcy}^p} + \overline{C}^p_{\mcy} \biggr]^{1/p}\\
        &\leq \min\limits_{\pi \in S_{n}} \biggl[ \frac{1}{l} \biggl(\sum_{i\in[n]} d_{\mcx}(v_i, w_{\pi(i)})^p + (l-n) \overline{C}^p_{\mcx} \biggr) &\\
        &\hspace*{15mm} + \frac12\frac{1}{l} \biggl( \sum_{i\in[n]} \frac{1}{l-1} \biggl( \sum_{\substack{i'\in[n] \\ i \neq i'}} d_{\mcy}(e_{ii'}, f_{\pi(i)\pi(i')})^p + (l-n)  C_{\mcy}^p \biggr) + (l-n) C_{\mcy}^p \biggr) + \overline{C}^p_{\mcy} \biggr]^{1/p} \\
        &\leq \min\limits_{\pi \in S_{n}} \biggl[ \frac{1}{l} \biggl(\sum_{i\in[n]} d_{\mcx}(v_i, w_{\pi(i)})^p + (l-n) C_2^p \biggr)  + \frac12 \frac{1}{l (l-1)}  \sum_{i\in[n]}   \sum_{\substack{i'\in[n] \\ i \neq i'}} d_{\mcy}(e_{ii'}, f_{\pi(i)\pi(i')})^p  + \overline{C}^p_{\mcy} \biggr]^{1/p} \\
        &\overset{(\dagger)}{=} \min\limits_{\pi \in S_{l}} \biggl[ \underbrace{ \frac{1}{l} \sum_{i\in[l]} d_{\mcx}(v_i, w_{\pi(i)})^p   + \frac12 \frac{1}{l(l-1)}\sum_{i\in[l]} \sum_{\substack{i'\in[l] \\ i \neq i'}}  d_{\mcy}(e_{ii'}, f_{\pi(i)\pi(i')})^p }_{=: \, F_{\pi}} + \overline{C}^p_{\mcy} \biggr]^{1/p}\\
        &=  \biggl(d_{\mathbb{G},R_2}(([l],v,e),([l],w,f))^p + C_{\mcy}^p \frac{(n-m)}{n}\biggr)^{1/p}.
\end{align*}
It remains to show Equation $(\dagger)$. We require that there is a minimizer $\pi_0 \in S_l$ of $F_{\pi}$ such that $\pi_0([n]) = [n]$. Once this is established, it follows
by $d_{\mcx}(v_i, w_{\pi_0(i)}) = C_2$ and $d_{\mcy}(e_{ii'}, f_{\pi_0(i)\pi_0(i')}) = 0$ for all $i, i' \in [l] \setminus [n]$ 
that the large bracket term in the last line is equal to the large bracket term in the second last line for any such $\pi_0$. Hence it is enough to minimize over $S_n$.

In order to show that there is a minimizer $\pi_0 \in S_l$ of $F_{\pi}$ which satisfies $\pi_0([n]) = [n]$, we start from any minimizer $\pi_* \in S_l$ of $F_{\pi}$. Let $A = \{i \in [n] \mvert \pi_*(i) \not\in [n] \}$ and $B = \{j \in [n] \mvert \exists i \in [l] \setminus [n] \text{ such that } \pi_*(i) = j \}$. Since $\pi_*$ is a permutation, $A$ and $B$ must have the same cardinality $r$, say. Define $\pi_0$ by swapping values of $i \in [n]$ for values from $B$ in an arbitrary way, i.e.\ enumerating $A = \{a_1, \ldots, a_r\}$, $B = \{b_1, \ldots, b_r\}$, $C = \pi_*^{-1}(B) = \{c_1, \ldots, c_r\}$, $D = \pi_*(A) = \{d_1, \ldots, d_r\}$ set
$$\pi_0(k) = \begin{cases}
  b_i &\text{if $k=a_i$ for some $i \in [r]$,} \\
  d_i &\text{if $k=c_i$ for some $i \in [r]$,} \\
  \pi_*(k) &\text{otherwise}.
\end{cases}
$$ 
Note that $d_{\mcx}(v_i, w_{\pi_0(i)}) = d_{\mcx}(v_i, w_{\pi_*(i)})$ for all $i\in[l]\setminus(A\cup C)$ and $d_{\mcy}(e_{ii'}, f_{\pi_0(i)\pi_0(i')}) = d_{\mcy}(e_{ii'}, f_{\pi_*(i)\pi_*(i')})$ for all $i,i'\in[l]\setminus(A\cup C)$ with $i\neq i'$. Due to the way we filled up the graphs we obtain $d_{\mcx}(v_i, w_{\pi_*(i)})^p = C_2^p$ and $d_{\mcx}(v_i, w_{\pi_0(i)})^p \leq C_{\mcx}^p$ for all $i\in A$, while $d_{\mcx}(v_i, w_{\pi_*(i)})^p =\overline{C}^p_{\mcx} = C_2^p-C_{\mcy}^p$ and $d_{\mcx}(v_i, w_{\pi_0(i)})^p = C_2^p$ for all $i\in C$. Further note that $e_{ij} = f_{ij} = y_0$ for all $n+1\leq i\vee j\leq l$. In particular we have $f_{\pi_0(i)\pi_0(j)}= e_{ij} = y_0$ for $\{i,j\}\cap C \neq \emptyset$ as well as $f_{\pi_*(i)\pi_*(j)} = y_0$ for $\{i,j\}\cap A \neq \emptyset$. This yields
\begin{align*}
  &F_{\pi_0} - F_{\pi_*}\\
  &= \frac{1}{l} \sum_{i \in A} \bigl[ \underbrace{d_{\mcx}(v_i, w_{\pi_0(i)})^p - d_{\mcx}(v_i, w_{\pi_*(i)})^p}_{\leq C_{\mcx}^p -C_2^p\leq \, -C_{\mcy}^p} \bigr] + \frac{1}{l} \sum_{i \in C} \bigl[ \underbrace{d_{\mcx}(v_i, w_{\pi_0(i)})^p - d_{\mcx}(v_i, w_{\pi_*(i)})^p}_{=\,C_2^p -\overline{C}^p_{\mcx} \, \leq \, C_{\mcy}^p} \bigr] \\
  &\hspace*{12mm} + \frac{1}{l(l-1)} \biggl( \sum_{i \in A} \sum_{\substack{i'\in[l] \\ i' \neq i}} \bigl[ \tfrac12 d_{\mcy}(e_{ii'}, f_{\pi_0(i)\pi_0(i')})^p - \tfrac12 d_{\mcy}(e_{ii'}, f_{\pi_*(i)\pi_*(i')})^p \bigr] \\[-1mm] 
  &\hspace*{30mm} + \sum_{i \in C} \sum_{\substack{i'\in[l] \\ i' \neq i}} \bigl[ \tfrac12 d_{\mcy}(e_{ii'}, f_{\pi_0(i)\pi_0(i')})^p - \tfrac12 d_{\mcy}(e_{ii'}, f_{\pi_*(i)\pi_*(i')})^p \bigr]\\[-1mm] 
  &\hspace*{30mm} +  \sum_{i\in [l]\setminus (A\cup C)} \sum_{\substack{i'\in (A\cup C)}} \bigl[ \tfrac12 d_{\mcy}(e_{ii'}, f_{\pi_0(i)\pi_0(i')})^p - \tfrac12 d_{\mcy}(e_{ii'}, f_{\pi_*(i)\pi_*(i')})^p \bigr]\biggr) \\
  &\leq \frac{1}{l(l-1)} \biggl( 2\cdot \sum_{i \in A} \sum_{\substack{i'\in[l] \\ i' \neq i}} \bigl[ \tfrac12 d_{\mcy}(e_{ii'}, f_{\pi_0(i)\pi_0(i')})^p - \tfrac12 d_{\mcy}(e_{ii'}, f_{\pi_*(i)\pi_*(i')})^p \bigr] \\[-1mm] 
  &\hspace*{8mm} + 2\cdot\sum_{i \in C} \sum_{\substack{i'\in[l] \\ i' \neq i}} \bigl[ \tfrac12 d_{\mcy}(e_{ii'}, f_{\pi_0(i)\pi_0(i')})^p - \tfrac12 d_{\mcy}(e_{ii'}, f_{\pi_*(i)\pi_*(i')})^p \bigr]\\[-1mm] 
  &\hspace*{8mm} -   \sum_{\substack{i,i'\in (A\cup C)\\ i' \neq i}} \bigl[ \tfrac12 d_{\mcy}(e_{ii'}, f_{\pi_0(i)\pi_0(i')})^p - \tfrac12 d_{\mcy}(e_{ii'}, f_{\pi_*(i)\pi_*(i')})^p \bigr]\biggr) \\
  &=  \frac{1}{l(l-1)} \biggl( \sum_{i \in A} \sum_{\substack{i'\in[n] \\ i' \neq i}} d_{\mcy}(e_{ii'}, f_{\pi_0(i)\pi_0(i')})^p + \sum_{i \in A} \sum_{i'\in[l]\setminus[n]} \underbrace{d_{\mcy}(e_{ii'}, f_{\pi_0(i)\pi_0(i')})^p}_{= \, 0} \\[-1mm] 
  &\hspace*{8mm} - \sum_{i \in A} \sum_{\substack{i'\in[n] \\ i' \neq i}} d_{\mcy}(e_{ii'}, f_{\pi_*(i)\pi_*(i')})^p - \sum_{i \in A} \sum_{i'\in[l]\setminus[n]}\underbrace{ d_{\mcy}(e_{ii'}, f_{\pi_*(i)\pi_*(i')})^p}_{= 0} \\[-1mm] 
  &\hspace*{8mm} + \sum_{i \in C} \sum_{i'\in[n]} \underbrace{d_{\mcy}(e_{ii'}, f_{\pi_0(i)\pi_0(i')})^p}_{= \, 0} + \sum_{i \in C} \sum_{\substack{i'\in[l]\setminus[n] \\ i' \neq i}} \underbrace{d_{\mcy}(e_{ii'}, f_{\pi_0(i)\pi_0(i')})^p}_{= \, 0} \\[-1mm] 
  &\hspace*{8mm} - \sum_{i \in C} \sum_{i'\in A} \underbrace{d_{\mcy}(e_{ii'}, f_{\pi_*(i)\pi_*(i')})^p}_{=0}- \sum_{i \in C} \sum_{\substack{i'\in[n] \\ i'\notin A}} d_{\mcy}(e_{ii'}, f_{\pi_*(i)\pi_*(i')})^p\\
  &\hspace*{8mm}- \sum_{i \in C} \sum_{\substack{i'\in C \\ i' \neq i}} d_{\mcy}(e_{ii'}, f_{\pi_*(i)\pi_*(i')})^p- \sum_{i \in C} \sum_{\substack{i'\in[l]\setminus[n]\\ i'\notin C}} \underbrace{d_{\mcy}(e_{ii'}, f_{\pi_*(i)\pi_*(i')})^p}_{= 0}\\
  &\hspace*{8mm}- \biggl[\sum_{\substack{i,i'\in A \\ i\neq i'}}\frac12 d_{\mcy}(e_{ii'}, f_{\pi_0(i)\pi_0(i')})^p - \sum_{\substack{i,i'\in A \\ i\neq i'}}\frac12 d_{\mcy}(e_{ii'}, f_{\pi_*(i)\pi_*(i')})^p\\
  &\hspace*{8mm} + \sum_{\substack{i,i'\in C \\ i\neq i'}}\frac12 \underbrace{d_{\mcy}(e_{ii'}, f_{\pi_0(i)\pi_0(i')})^p}_{=0} - \sum_{\substack{i,i'\in C \\ i\neq i'}}\frac12 d_{\mcy}(e_{ii'}, f_{\pi_*(i)\pi_*(i')})^p \\
  &\hspace*{8mm}+ 2\cdot \sum_{i\in A} \sum_{i'\in C}\frac12 \underbrace{d_{\mcy}(e_{ii'}, f_{\pi_0(i)\pi_0(i')})^p}_{=0} -  2\cdot \sum_{i\in A} \sum_{i'\in C}\frac12\underbrace{ d_{\mcy}(e_{ii'}, f_{\pi_*(i)\pi_*(i')})^p}_{=0} \biggr] \biggr)\\
  &\leq  \frac{1}{l(l-1)} \biggl( \sum_{\substack{i,i'\in A \\ i' \neq i}} d_{\mcy}(e_{ii'}, f_{\pi_0(i)\pi_0(i')})^p + \sum_{i \in A} \sum_{\substack{i'\in[n] \\ i' \notin A}} d_{\mcy}(e_{ii'}, f_{\pi_0(i)\pi_0(i')})^p \\
  &\hspace*{8mm}- \sum_{\substack{i,i'\in A \\ i' \neq i}} d_{\mcy}(e_{ii'}, f_{\pi_*(i)\pi_*(i')})^p- \sum_{i \in A} \sum_{\substack{i'\in[n] \\ i' \notin A}} d_{\mcy}(e_{ii'}, f_{\pi_*(i)\pi_*(i')})^p\\
  &\hspace*{8mm}-  \sum_{\substack{i,i'\in C \\ i\neq i'}}d_{\mcy}(e_{ii'}, f_{\pi_*(i)\pi_*(i')})^p - \sum_{i \in C} \sum_{\substack{i'\in[n]\\ i'\notin A}} d_{\mcy}(e_{ii'}, f_{\pi_*(i)\pi_*(i')})^p\\
  &\hspace*{8mm}- \biggl[\sum_{\substack{i,i'\in A \\ i\neq i'}}\frac12 d_{\mcy}(e_{ii'}, f_{\pi_0(i)\pi_0(i')})^p - \sum_{\substack{i,i'\in A \\ i\neq i'}}\frac12 d_{\mcy}(e_{ii'}, f_{\pi_*(i)\pi_*(i')})^p - \sum_{\substack{i,i'\in C \\ i\neq i'}}\frac12 d_{\mcy}(e_{ii'}, f_{\pi_*(i)\pi_*(i')})^p \biggr]\biggr)\\
%
  &=\frac{1}{l(l-1)} \biggl( \sum_{\substack{i,i'\in A \\ i' \neq i}} \frac12d_{\mcy}(e_{ii'}, f_{\pi_0(i)\pi_0(i')})^p + \sum_{i \in A} \sum_{\substack{i'\in[n] \\ i' \notin A}} d_{\mcy}(e_{ii'}, f_{\pi_0(i)\pi_0(i')})^p \\
  &\hspace*{8mm}- \sum_{\substack{i,i'\in A \\ i' \neq i}} \frac12d_{\mcy}(e_{ii'}, y_0)^p- \sum_{i \in A} \sum_{\substack{i'\in[n] \\ i' \notin A}} d_{\mcy}(e_{ii'}, y_0)^p\\
  &\hspace*{8mm}-  \sum_{\substack{i,i'\in C \\ i\neq i'}}\frac12d_{\mcy}(y_0, f_{\pi_*(i)\pi_*(i')})^p - \sum_{i \in C} \sum_{\substack{i'\in[n]\\ i'\notin A}} d_{\mcy}(y_0, f_{\pi_*(i)\pi_*(i')})^p \biggr)\\
  &= \frac{1}{l(l-1)} \biggl(\frac12\biggl[ \sum_{\substack{i,i'\in A \\ i' \neq i}} d_{\mcy}(e_{ii'}, f_{\pi_0(i)\pi_0(i')})^p- \sum_{\substack{i,i'\in A \\ i' \neq i}} d_{\mcy}(e_{ii'}, y_0)^p- \sum_{\substack{i,i'\in C \\ i\neq i'}}d_{\mcy}(y_0, f_{\pi_*(i)\pi_*(i')})^p\biggr]\\
  &\hspace*{8mm}+ \sum_{i \in A} \sum_{\substack{i'\in[n] \\ i' \notin A}} d_{\mcy}(e_{ii'}, f_{\pi_0(i)\pi_0(i')})^p - \sum_{i \in A} \sum_{\substack{i'\in[n] \\ i' \notin A}} d_{\mcy}(e_{ii'}, y_0)^p - \sum_{i \in C} \sum_{\substack{i'\in[n]\\ i'\notin A}} d_{\mcy}(y_0, f_{\pi_*(i)\pi_*(i')})^p \biggr)\\
  &= \frac{1}{l(l-1)} \biggl( \frac12\biggl[\sum_{\substack{i,i'\in A \\ i' \neq i}} d_{\mcy}(e_{ii'}, f_{\pi_0(i)\pi_0(i')})^p- \sum_{\substack{i,i'\in A \\ i' \neq i}} d_{\mcy}(e_{ii'}, y_0)^p- \sum_{\substack{i,i'\in A \\ i\neq i'}}d_{\mcy}(y_0, f_{\pi_0(i)\pi_0(i')})^p\biggr]\\
  &\hspace*{8mm}+ \sum_{i \in A} \sum_{\substack{i'\in[n] \\ i' \notin A}} d_{\mcy}(e_{ii'}, f_{\pi_0(i)\pi_0(i')})^p - \sum_{i \in A} \sum_{\substack{i'\in[n] \\ i' \notin A}} d_{\mcy}(e_{ii'}, y_0)^p - \sum_{i \in A} \sum_{\substack{i'\in[n]\\ i'\notin A}} d_{\mcy}(y_0, f_{\pi_0(i)\pi_0(i')})^p \biggr)\\
  &\leq 0.
\end{align*}
The last equality follows because by construction of $\pi_0$ we have $\{\pi_*(i)\,\vert\,i\in C\} = \{\pi_0(i)\,\vert\,i\in A\}$, which implies that summing $d_{\mcy}(y_0, f_{\pi_*(i)\pi_*(i')})^p$ over $i,i'\in C$ is the same as summing $d_{\mcy}(y_0, f_{\pi_0(i)\pi_0(i')})^p$ over $i,i'\in A$. As further $\pi_0\vert_{[n]\setminus A} = \pi_*\vert_{[n]\setminus A} $ this also yields that summing $d_{\mcy}(y_0, f_{\pi_*(i)\pi_*(i')})^p$ over $i\in C$ and $i'\in[n]\setminus A$ is the same as summing $d_{\mcy}(y_0, f_{\pi_0(i)\pi_0(i')})^p$ over $i\in A$ and $i' \in[n]\setminus A$.
The last inequality holds by the triangle inequality for $d_{\mcy}$ in the case $p=1$ and by the assumption that $d_{\mcy}(y_1,y_2) \leq \max\{d_{\mcy}(y_1,y_0),d_{\mcy}(y_0,y_2)\}$ for all $y_1,y_2\in\mcy$ in the case $p>1$.
\end{proof}

\subsection{Proofs left out in Chapter~\ref{ssec: convergence considerations}}

\begin{proof}[Proof of Theorem~\ref{thm: Metrization of weak convergence}]
We start by showing that convergence w.r.t. to one of the $d_{\G}$ metrics is equivalent to weak convergence in $\mfn'_{\mathbb{G}}$. For this, we build on the ideas in the proof of \textcite[Proposition 4.2]{xia2005}, which deals with sequences of individual counting measures.

Recall that ${([m_n],v_n,e_n), ([m],v,e) \in \mathbb{G}}$ are only fixed up to permutations of the underlying index set. To lighten the notation we may assume w.l.o.g.\ that for graphs of the same size $m$ the optimal permutation in the definition of $d_{\mathbb{G}}(([m],v_n,e_n),([m],v,e))$ is always the identity. As we will see below, both weak convergence and convergence w.r.t.\ $d_{\G}$ directly imply the existence of an $N\in\N$ such that $m_n = m$ and hence we may and do assume the identity is the optimal solution for all $n\geq N$. We transfer these enumerations to the graph measures $(\xi_n,\sigma_n) = (\sum_{i} \delta_{{v'_i}^n}, \sum_{j < j'} \delta_{{e'_{jj'}\hspace*{-7pt}}^n})$ and $(\xi,\sigma) = (\sum_{i} \delta_{{v'_i}}, \sum_{j < j'} \delta_{{e'_{jj'}}})$ for $n\geq N$. 

We note in the case where $d_{\G}$ is the GTT metric $d_{\G,A}$ that the graphs ${([m],v_n,e_n), ([m],v,e)}$ are filled up to the size of $2m$ with auxiliary elements according to Proposition~\ref{prop: GTT as optimal matching (filled to n+m)}. Similarly, we fill up the corresponding measures $(\xi_n,\sigma_n), (\xi,\sigma)$ using the bijection to the graph objects. For simplicity, we continue to use $m$ to denote the total number of vertices. Clearly, this does not change the equivalence of the convergences. 

Assume first that $d_{\mathbb{G}}(([m_n],v_n,e_n),([m],v,e))\to 0$. As $C>0$, there exists an $N\in \N$ such that $m_n = m$ and hence $ \abs{\xi_n} = \abs{\xi}$ for all $n\geq N$. 

Since $\mfn(\mcx') \times \mfn(\mcy')$ and $\mfn(\mcx' \times \mcy')$ are homeomorphic, we have for $(\xi_n,\sigma_n),(\xi, \sigma)\in \mfn'_\mathbb{G}$ that $(\xi_n,\sigma_n) \to (\xi, \sigma)$ weakly in $\mfn'_\mathbb{G}$ iff for every bounded continuous function $f \colon \mcx'\times\mcy' \to \R$ we have $\int f \, d(\xi_n \otimes \sigma_n) \to \int f \, d(\xi \otimes \sigma)$.
Now consider a bounded continuous function $f$ on $\mcx'\times\mcy'$. Then for $n\geq N$
\begin{align}
    &\bigg\vert \int_{\mcx'\times\mcy'} f(x',y')  \; \xi_n \otimes \sigma_n (\diff(x',y')) - \int_{\mcx'\times\mcy'} f(x',y') \; \xi\otimes\sigma (\diff(x',y'))\bigg\vert \notag\\[1mm]
    &= \bigg\vert \sum_{i=1}^m \sum_{\substack{j,j'=1 \\ j<j'}}^m f({v'_i}^n,{e'_{jj'}\hspace*{-7pt}}^n) - \sum_{i=1}^m \sum_{\substack{j,j'=1 \\ j<j'}}^m f(v'_i,e'_{jj'})\bigg\vert \notag\\
    &\leq  \sum_{i=1}^m \sum_{\substack{j,j'=1 \\  j<j'}}^m \big\vert f({v'_i}^n,{e'_{jj'}\hspace*{-7pt}}^n) - f(v'_i,e'_{jj'})\big\vert \notag\\
    &\leq m^3 \sup_{\substack{(v',\tilde{v}') \, \in \, \mathcal{S}^2 (\mcx'), \, (e',\tilde{e}') \, \in \, \mathcal{S}^2 (\mcy') \\
    d_{\mcx}(v,\tilde{v}) \, \leq \, C^{1}_m\, d_{\mathbb{G}}(([m_n],v_n,e_n),([m],v,e)) \\ d_{\mcy}(e,\tilde{e}) \, \leq \, C^{2}_m\, d_{\mathbb{G}}(([m_n],v_n,e_n),([m],v,e))}} \abs{f(v',e')-f(\tilde{v}',\tilde{e}')}, \label{eq:convergence_sup} 
\end{align}
where $v, \tilde{v}\in\mcx$ and $e,\tilde{e}\in\mcy$ denote the attribute components of $v', \tilde{v}'\in\mcx'$ and $e',\tilde{e}'\in\mcy'$, respectively, and $C^{1}_m = C^{2}_m = 1$ if $d_{\G} = d_{\G,A}$ is the GTT metric and $C^{1}_m = m^{1/p}$, ${C^{2}_m = (m \, (m-1))^{1/p}}$ if $d_{\G} = d_{\G,R_i}$, $i=1,2$, is one of the GOSPA metrics with order~$p$. Moreover, the sets $\mathcal{S}^2(\mcx')$ and $\mathcal{S}^2(\mcy')$ are defined as $\mathcal{S}^2(\mathcal{Z}') := \{(z',\tilde{z}')\,\mvert\,z',\tilde{z}'\in \mathcal{Z}'\text{ and } d_{\mathcal{Z}'}(z',\tilde{z}') = d_{\mathcal{Z}}(z,\tilde{z}) \}$ for $\mathcal{Z}= \mcx,\mcy$. Equivalently, the sets $\mathcal{S}^2(\mcx')$ and $\mathcal{S}^2(\mcy')$ contain pairs of vertex elements~$(v',\tilde{v}')$ and pairs of edge elements~$(e',\tilde{e}')$, respectively, that vary only in their attribute components, i.e.\ they have the same values in the extra labeling components (if they have any). 
In order to obtain the crucial two inequalities for $d_{\mcx}(v,\tilde{v})$ and $d_{\mcy}(e,\tilde{e})$ in the supremum, we have used the fact that the identity permutation gives the optimal pairing of vertex and edge indices (due to the initial choice of their enumerations). 

By Tychonoffs Theorem the space $\mcx'\times\mcy'$ is compact, which implies that $f$ is uniformly continuous. Thus the supremum in~\eqref{eq:convergence_sup} converges to zero, as  $d_{\mathbb{G}}(([m_n],v_n,e_n),([m],v,e))\to 0$ by assumption. Since $f$ was an arbitrary bounded continuous function on $\mcx'\times\mcy'$, we have shown weak convergence of $(\xi_n,\sigma_n)$ to $(\xi,\sigma)$. 

For the converse direction, assume that $(\xi_n,\sigma_n)$ converges weakly to $(\xi,\sigma)$. Hence $\xi_n \to \xi$ weakly and $\sigma_n \to \sigma$ weakly. By the former, choosing $f\equiv 1$, we obtain ${m_n = \xi_n(\mcx') \to \xi(\mcx') = m}$. Thus there exists an $N \in \N$ such that $m_n = m$ for all $n\geq N$. 

We derive the convergence of $([m_n],v_n,e_n)$ to $([m],v,e)$ in $d_{\G}$ by showing that every subsequence $(n_k)_k$ of $\{N,N+1,\ldots\}$ has a subsubsequence $(n_{k_j})_j$ such that $([m_{n_{k_j}}],v_{n_{k_j}},e_{n_{k_j}}) \to ([m],v,e)$.

Write $v_i^n = v_n(i)$ and $e_{ii'}^n = e_n(i,i')$ for the vertex and edge attributes. 
Let $(n_k)_k$ be a subsequence of $\{N,N+1,\ldots\}$ and consider the corresponding sequences of vertex and edge attributes. Recall that their indices are aligned with the indices of the vertex and edge attributes of $([m],v,e)$ so that they realize the $d_{\G}$-optimal matching.
As $\mcx$ is compact, we may pick a convergent subsubsequence for every subsequence of $(v_i^{n_k})_{k}$ and this for any $i \in [m]$. Picking subsubsequences one after the other for $i=1, \ldots, m$, we arrive at a preliminary subsubsequence $(n_{k_j})_j$ of $(n_k)_k$ and $\tilde{v}_i \in \mcx$, $i \in [m]$, such that $d_{\mcx}(v_i^{n_{k_j}},\tilde{v}_i)\to 0$ for every $i \in [m]$. As $\mcy$ is also compact, we may likewise pick a convergent subsubsequence for every subsequence of $(e_{ii'}^{n_{k_j}})_j$ and this for any $i,i' \in [m]$ with $i<i'$. Repeated picking of subsubsequences yields a final subsubsequence of $(n_k)_k$, which for convenience we refer to as $(n_{k_j})_j$ again, and $\tilde{e}_{ii'} \in \mcx$, $i,i' \in [m]$ with $i<i'$, such that both $d_{\mcx}(v_i^{n_{k_j}},\tilde{v}_i)\to 0$ and $d_{\mcx}(e_{ii'}^{n_{k_j}},\tilde{e}_{ii'})\to 0$ for all ${i,i'\in[m]}$ (setting $\tilde{e}_{ii}=0$ and $\tilde{e}_{i'i}=\tilde{e}_{ii'}$ for $i<i'$). Denoting by $([m],\tilde{v},\tilde{e})$ the graph constructed from the limit attributes, i.e.\ $\tilde{v}:[m] \to\mcx, \tilde{v}(i) = \tilde{v}_i$ and $\tilde{e}:[m]^2 \to\mcy, \tilde{e}(i,i') = \tilde{e}_{ii'}$, we obtain $d_{\mathbb{G}}(([m_{n_{k_j}}],v_{n_{k_j}},e_{n_{k_j}}),([m],\tilde{v},\tilde{e}))\to 0$ by the definition of~$d_{\mathbb{G}}$.

The forward direction then implies the convergence of the corresponding graph measures, i.e.\  $(\xi_{n_{k_j}},\sigma_{n_{k_j}})_j$ converges weakly to $(\tilde{\xi},\tilde{\sigma})$. By the uniqueness of the weak limit and the fact that $(\xi_{n},\sigma_{n})$ converges weakly to $(\xi,\sigma)$, we obtain $(\tilde{\xi},\tilde{\sigma}) = (\xi,\sigma)$. By the bijection to the graph space, this implies $([m],\tilde{v},\tilde{e}) = ([m],v,e)$. Since metric limits are unique, we have thus shown $([m_{n_{k_j}}],v_{n_{k_j}},e_{n_{k_j}}) \to ([m],v,e)$ in $d_{\G}$.

A crucial ingredient of this proof was that the graph measures ``see'' the assignment of their edge attributes to the (uniquely identifiable) vertex attributes, a fact which is implicitly used every time we use the bijection to go from $\mfn'_{\G}$ to $\G$.

It remains to show that the spaces $(\mfn'_{\mathbb{G}},d_{\mathbb{G}})$ are c.s.m.s.

We start with the separability property and remark that this property depends only on the topology. The separability of the spaces $\mfn(\mcx')$ and $\mfn(\mcy')$ (endowed with the weak topology) follows by \textcite[Proposition~2.2]{SXia2008}. Hence, the space $\mfn(\mcx')\times \mfn(\mcx')$ is separable as it is the product space of two separable metric spaces. As subspaces of separable metric spaces are separable again, this yields the separability of $\mfn'_{\mathbb{G}}$.

Finally, we show the completeness of the spaces $(\mfn'_{\mathbb{G}},d_{\mathbb{G}})$ in a somewhat similar way as in \textcite[Proposition~2.2]{SXia2008}. Let $(\xi_n,\sigma_n)_n$ be a Cauchy sequence in $\mfn'_{\mathbb{G}}$. Analogous to the proof for metrizing weak convergence, one can show that this implies the existence of an $N\in\N$ and $k \in \N_0$ such that $\abs{\xi_n} = \abs{\xi_n'}=k$ and hence $\abs{\sigma_n} = \abs{\sigma_n'} = \binom{k}{2}$ for all ${n,n'\geq N}$. 
Note that $\mfn^k(\mcx') = \{\xi \in \mfn(\mcx') \mvert \abs{\xi}=k\}$ and $\mfn^k(\mcy') = \{\sigma \in \mfn(\mcy') \mvert \abs{\sigma} = \binom{k}{2} \}$ are compact by the compactness of $\mcx'$ and $\mcy'$, hence so is their Cartesian product. Using that edge attributes can be continuously mapped to their unique vertex set, it can be shown that ${{\mfn'}^{k}_{\mathbb{G}} = \{(\xi,\sigma)\in\mfn'_{\mathbb{G}}\,\vert\, \abs{\xi} = k\}}$ is a closed subset of $\mfn^k(\mcx') \times \mfn^k(\mcy')$ and hence compact as well. Since the tail of $(\xi_n,\sigma_n)_n$ lies in this space, we have a Cauchy sequence with a convergent subsequence, hence $(\xi_n,\sigma_n)_n$ converges in ${\mfn'}_{\mathbb{G}}^k \subset {\mfn'}_{\mathbb{G}}$.
\end{proof}

\begin{proof}[Proof of Theorem~\ref{thm: Convergence Example}]
We construct a coupling of the graph measure pair such that the expected $p$-th power of the $d_{\G,R_k}$-distance is small. For this let $(\tXi_n,\tXi)$ be an optimal coupling for the $W_{\varrho,C_k,p}$-metric, i.e.\ we have $\tXi_n \eqinlaw \Xi_n$, $\tXi \eqinlaw \Xi$ and
$$
  \E[\varrho(\tXi_n,\tXi)^p] = W_{\varrho,C_k,p}(\mcl(\Xi_n),\mcl(\Xi)). 
$$
That such a minimizing coupling exists, follows from \textcite[Theorem~4.1]{villani2009} and the fact that $(\mfn(\mcx), \varrho)$ is a c.s.m.s. Let $\Pi=\Pi_n$ be the random permutation that attains the minimum in the definition of $\varrho(\tXi_n,\tXi)$ (if there are several, a systematic one is picked according to a fixed rule) 
In order to have a fixed domain, we interpret $\Pi$ as random element of the (countable) set $S$ of permutations $\pi \colon \N \to \N$ that satisfy $\pi(i) = i$ for all but finitely many $i$. The measurability of $\Pi$ follows by the fact that $\{\Pi = \pi\}$ can be represented by a measurable function of $(\tXi_n,\tXi)$, since the objective function for $\varrho(\tXi_n,\tXi)$ is itself such a function.

To couple the edge measures, consider $(U_{ii'})_{i,i'\in \N, i<i'}$ i.i.d.\ uniformly distributed on~$[0,1]$. Set $\tE_{n,ii'} = \tE_{n,i'i} = \1_{\{U_{ii'} \leq q_n\}}$ and $\tE_{\Pi(i)\Pi(i')} = \tE_{\Pi(i')\Pi(i)} = \1_{\{U_{ii'} \leq q\}}$ for $i < i'$. Note that
\begin{equation}
  \E \bigl[ d_{\mcy}(E_{n,ii'}, E_{\Pi(i)\Pi(i')})^p \bigr] 
  = C_{\mcy}^p \, \Prob \bigl( \min\{q_n,q\} < U_{ii'} \leq \max\{q_n,q\} \bigr) = C_{\mcy}^p \, \abs{q_n-q}.
\end{equation}
Also set $\tE_{n,ii} = \tE_{ii} = 0$ for all $i \in \N$.

Writing $\tXi_n = \sum_{i=1}^{\tM_n} \delta_{\tX_{n,i}}$ and $\tXi = \sum_{i=1}^{\tM} \delta_{\tX_{i}}$, we set
$$
  \tSigma_n = \sum_{\substack{i,i'=1\\ i<i'}}^{\tM_n} \delta_{(\tE_{n,{ii'}}, \{\tX_{n,i}, \tX_{n,i'}\})}, \quad \tSigma = \sum_{\substack{i,i'=1\\ i<i'}}^{\tM} \delta_{(\tE_{ii'}, \{\tX_{i}, \tX_{i'}\})}.
$$
By the independence of the $\Xi$- and their respective $E$-processes as well as the exchangeability of the $E$-processes in their double indices, it is easily seen that $\mcl(\tSigma_n \mvert \tXi_n) = \mcl(\Sigma_n \mvert \Xi_n)$ and $\mcl(\tSigma \mvert \tXi) = \mcl(\Sigma \mvert \Xi)$ and therefore $\mcl(\tXi_n, \tSigma_n) = \mcl(\Xi_n, \Sigma_n)$ and $\mcl(\tXi, \tSigma) = \mcl(\Xi, \Sigma)$.

As a consequence, writing $\tL_{-} = \min\{\tM_n,\tM\}$, $\tL_{+} = \max\{\tM_n,\tM\}$ and recalling that we use the parameter $p$ for the order of both the $d_{\G,R_k}$ and the Wasserstein metric, we have
\begin{align}
    &W_{\mathbb{G},R_k,p}(\mcl(\Xi_n, \Sigma_n),\mcl(\Xi, \Sigma)) \notag\\[1mm]
    &\hspace*{2mm} \leq \E \bigl[ d_{\G,R_k}\bigl( (\tXi_n, \tSigma_n), (\tXi, \tSigma) \bigr)^p \bigr] \notag\\
    &\hspace*{2mm} = \E [\varrho(\tXi_n,\tXi)^p] + \frac12 \, \E \biggl[\frac{1}{\tL_{+}(\tL_{+}-1)}  \sum_{i,i'\in[\tL_{-}]^2} d_{\mcy}(\tE_{n,ii'},\tE_{n,\Pi(i)\Pi(i')})^p \biggr]
    + \frac12 \, \E \tR_k \notag\\
    &\hspace*{2mm} \leq W_{\varrho,C_k,p}(\mcl(\Xi_n),\mcl(\Xi))^p + C_{\mcy}^p \frac{\abs{q_n-q}}{2} + \frac12 \, \E \tR_k, 
    \label{eq:wassergospa}
\end{align}
where $\tR_k = \tR_k(\tM_n, \tM, (\tE_{n,ii'}), (\tE_{ii'}))$ stands for the remainder term in Definition~\ref{def:gttmetric} that depends on the GOSPA metric used. Recall that we use the convention $0/0=0$, which is relevant whenever $\tL_{+}$ takes a value $\in \{0,1\}$ above. 

Noting that $\frac{l_{-}(l_{+}-l_{-})}{l_{+} (l_{+}-1)} \leq \frac{l_{+}-l_{-}}{l_{+}}$ for $l_{-},l_{+} \in \N_0$ with $l_{-} \leq l_{+}$ (again using $0/0 = 0$), we obtain for GOSPA1
\begin{align}
    \frac12 \, \E \tR_1
    = \frac12 \E \biggl[\frac{1}{\tL_+(\tL_+-1)} \tL_- (\tL_+-\tL_-) C_{\mcy}^p \biggr]
    \leq \frac12 C_{\mcy}^p \, \E \biggl[\frac{\tL_+-\tL_-}{\tL_{+}} \biggr]
    \leq W_{\varrho,C_1,p}(\mcl(\Xi_n),\mcl(\Xi))^p,
    \label{eq:remainder_gospa1}
\end{align}
using $\frac12 C_{\mcy}^p \leq C_1^p$,
and for GOSPA2 by the independence of $\Xi$- and $E$-processes
\begin{align}
    \frac12 \, \E \tR_2
    &= \frac12 \, \E \biggl[\frac{\1 \{\tM_n > \tM\}}{\tM_n(\tM_n-1)} 2 \sum_{i=1}^{\tM} \sum_{i'=\tM + 1}^{\tM_n} C_{\mcy}^p \tE_{n,ii'}  +  \frac{\1 \{\tM_n \leq \tM\}}{\tM(\tM-1)} 2 \sum_{i=1}^{\tM_n} \sum_{i'=\tM_n + 1}^{\tM} C_{\mcy}^p \tE_{ii'} \biggr] \notag\\
    &= \max\{q_n,q\} \, C_{\mcy}^p \, \E \biggl[\frac{1}{\tL_+(\tL_+-1)} \tL_- (\tL_+-\tL_-) \biggr] \notag\\
    &= \max\{q_n,q\} \, C_{\mcy}^p \, \E \biggl[\frac{\tL_+-\tL_-}{\tL_{+}} \biggr] \notag\\[1mm]
    &\leq \max\{q_n,q\} \, W_{\varrho,C_2,p}(\mcl(\Xi_n),\mcl(\Xi))^p, \label{eq:remainder_gospa2}
\end{align}
using $C_{\mcy}^p \leq C_2^p$. Note that the last inequalities in \eqref{eq:remainder_gospa1} and \eqref{eq:remainder_gospa2} are rather coarse in general and something better may be possible for specific point process distributions.

Plugging \eqref{eq:remainder_gospa1} and \eqref{eq:remainder_gospa2} into \eqref{eq:wassergospa} yields the claimed upper bound. Its convergence to zero follows from the fact that $W_{\varrho, C_k, p}$ metrizes convergence in distribution of point processes; see \textcite[Proposition~2.3 and Remark~5.1]{SXia2008}.
\end{proof}

\section{Algorithms}\label{sec: algorith}

\IncMargin{0em}
\begin{algorithm}[!ht]
\footnotesize
  \caption{\small \FuncSty{AuctionExt}: find good permutation for GTT/GOSPA distances. $\FuncSty{bidfor}(i, j, \DataSty{bidder2obj})$ makes the changes in assignment \DataSty{bidder2obj} if $i$ bids for object~$j$. \FuncSty{which.max} and \FuncSty{which.max.2} return indices of a largest and second largest value in a vector, $\varepsilon > 0$ ensures that \DataSty{bidincrement} is always positive.}\label{algo:AuctionExt} 
\DontPrintSemicolon
\SetKwData{penalties}{penalties}\SetKwData{iter}{iter}\SetKwData{maxiter}{maxiter}\SetKwData{stopat}{stop\_at}\SetKwData{change}{change}\SetKwData{fullmatch}{fullmatch}\SetKwData{fullmatchnum}{fullmatchnum}\SetKwData{persvalue}{persvalue}\SetKwData{bidfor}{bidfor}\SetKwData{perstoobj}{bidder2obj}\SetKwData{perstoobjnew}{bidder2obj\_new}\SetKwData{prices}{prices}\SetKwData{persvalue}{persvalue}\SetKwData{persvalues}{persvalues}\SetKwData{persvalue}{persvalue}\SetKwData{bidincrement}{bidincrement}\SetKwData{cost}{cost}\SetKwData{bestcost}{bestcost}\SetKwData{bestmatch}{bestmatch}
\SetKwData{true}{TRUE}\SetKwData{false}{FALSE}\SetKwData{na}{NA}\SetKwData{inf}{INF}
\SetKwFunction{rep}{rep}\SetKwFunction{comppersvals}{compute\_persvals}\SetKwFunction{whichmax}{which.max}\SetKwFunction{whichmaxtwo}{which.max.2}\SetKwFunction{bidfor}{bidfor}
\SetKwInOut{input}{Input}\SetKwInOut{output}{Output}
\input{$n \times n$ matrix $d$ of vertex distances ($d[i,j] = d^{(V)}_{i,j}$);\\
       $n \times n$ matrices $e,f$ of edge attributes $e[i,i']$, $f[j,j']$ for the two graphs;\\
       \penalties for the metric ($C$ for GTT, $C_{\mcy}$ and $C_1$ for GOSPA1, $C_2$ for GOSPA2);\hspace*{-5mm}\\
       termination parameters $\maxiter$ and $\stopat$.}   
\BlankLine
$\change \leftarrow \true$; $\iter \leftarrow 1$; $\fullmatch \leftarrow \false$; $\fullmatchnum \leftarrow 0$\;
$\perstoobj \leftarrow \rep(\na, n)$; $\prices \leftarrow \rep(0,n)$\;
$\bestcost \leftarrow \inf$; $\bestmatch \leftarrow \rep(\na, n)$\;
\While{$(\change = \true \ \&\& \ \fullmatchnum < \stopat \ \&\& \ \iter \leq \maxiter)$}{
  $\change \leftarrow \false$\;
  \For(\tcp*[f]{$i$ is index of current bidder}){$i \leftarrow 1$ \KwTo $n$}{
    $\persvalues \leftarrow \comppersvals(i, d, e, f, \penalties)$ \tcp*[r]{vector of personal values}
    $j \leftarrow \whichmax(\persvalues)$ \tcp*[r]{index of the new object $i$ bids for}
    $\perstoobjnew \leftarrow \bidfor(i,j, \perstoobj)$\;
    \lIf{$\perstoobjnew \neq \perstoobj$}{$\change \leftarrow \true$}
    $\perstoobj \leftarrow \perstoobjnew$\;
        $\bidincrement \leftarrow \whichmax(\persvalues)-\whichmaxtwo(\persvalues) + \varepsilon$\;
        $\prices[j] \leftarrow \prices[j] + \bidincrement$\;
        \eIf{everything is assigned}{
            \If{$\fullmatch = \false$}{
              $\fullmatch \leftarrow \true$\;
              $\fullmatchnum \leftarrow \fullmatchnum+1$\;
            }
            $\cost \leftarrow$ compute cost of assignment using penalties\;
            \If{$\cost < \bestcost$}{
              $\bestcost \leftarrow \cost$\;
              $\bestmatch \leftarrow \perstoobj$\;
            }
        }{
          $\fullmatch \leftarrow \false$
        }
  }
  $\iter \leftarrow \iter + 1$\;
}
\KwRet{\bestcost, \bestmatch}\;
\end{algorithm}\DecMargin{0em}

\IncMargin{0em}
\begin{algorithm}[!ht]
\footnotesize
  \caption{\small \FuncSty{compute\_persvals}: compute the personal values of the objects for a given bidder, assuming she must pay for externalities. \FuncSty{sum\_edge\_dist} computes the sum of $d_{\mcy}$-distances between two vectors of edge attributes (assigned in the order provided), \FuncSty{best\_sum\_edge\_dist} performs an optimal assignment between two vectors of edge attributes before taking the sum. \DataSty{vfact} and \DataSty{efact} are the normalizing factors for the vertex and edge sum terms, respectively, in the used metric, i.e.\ $\DataSty{vfact} = \frac1n$ and $\DataSty{efact} = \frac{1}{n(n-1)}$ for GOSPA and $\DataSty{vfact} = \DataSty{efact} = 1$ for GTT.  \DataSty{maxcost} is an arbitrary constant that does not change the result.}\label{algo:compute_persvals} 
\DontPrintSemicolon
\SetKwData{penalties}{penalties}\SetKwData{edgedists}{edgedists}\SetKwData{externality}{externality}\SetKwData{perstoobj}{bidder2obj}\SetKwData{objtopers}{obj2bidder}\SetKwData{perstoobjhypo}{bidder2obj\_new}\SetKwData{objtopershypo}{obj2bidder\_new}\SetKwData{prices}{prices}\SetKwData{assigned}{assigned}\SetKwData{persvals}{persvals}\SetKwData{persvalue}{persvalue}\SetKwData{bidincrement}{bidincrement}\SetKwData{cost}{cost}\SetKwData{bestcost}{bestcost}\SetKwData{bestmatch}{bestmatch}\SetKwData{abidnew}{abidnew}\SetKwData{abid}{abid}\SetKwData{aobj}{aobj}\SetKwData{aobjnew}{aobjnew}\SetKwData{externalityone}{externality1}\SetKwData{externalitytwo}{externality2}\SetKwData{compensations}{compensations}\SetKwData{distbefore}{distbefore}\SetKwData{distafter}{distafter}\SetKwData{vfact}{vfact}\SetKwData{efact}{efact}\SetKwData{maxcost}{maxcost}
\SetKwData{true}{TRUE}\SetKwData{false}{FALSE}\SetKwData{na}{NA}\SetKwData{inf}{INF}
\SetKwFunction{rep}{rep}\SetKwFunction{compvalext}{compute\_persval\_with\_ext}\SetKwFunction{whichmax}{which.max}\SetKwFunction{whichmaxtwo}{which.max.2}\SetKwFunction{bidfor}{bidfor}\SetKwFunction{sumedgedist}{sum\_edge\_dist}\SetKwFunction{bsumedgedist}{best\_sum\_edge\_dist}\SetKwFunction{compemin}{comp\_min\_ecost}\SetKwFunction{codesum}{sum}\SetKwFunction{which}{which}
\SetKwInOut{input}{Input}\SetKwInOut{output}{Output}
\input{index $i$ of the current bidder;\\
       $n \times n$ matrix $d$ of vertex distances $\smash{d[i,j] = d^{(V)}_{i,j}}$;\\
       $n \times n$ matrices $e,f$ of edge attributes $e[i,i']$, $f[j,j']$ for the two graphs;\\
       \penalties for the metric ($C$ for GTT, $C_{\mcy}$ and $C_1$ for GOSPA1, $C_2$ for GOSPA2)\\
       \hspace*{2mm} used implicitly in the functions $\sumedgedist$ and $\bsumedgedist$ below.}
\BlankLine
$\edgedists \leftarrow \rep(0, n)$; $\compensations \leftarrow \rep(\na, n)$\;
\For(\tcp*[f]{investigate changes if $i$ bids for object $j$}){$j \leftarrow 1$ \KwTo $n$}{
    $\perstoobjhypo \leftarrow \bidfor(i,j, \perstoobj)$\;
     $\abidnew \leftarrow \which(\perstoobjhypo \neq \na)$ \tcp*[r]{indices of newly assigned bidders}
     $\aobjnew \leftarrow \perstoobjhypo[\abidnew]$          \tcp*[r]{indices of newly assigned objects} 
     \tcp*[l]{actual edgedists for assigned vertices:}
     $\edgedists[j] \leftarrow \sumedgedist(e[i,\abidnew], f[j, \aobjnew])$\;
     \tcp*[l]{minimal possible edgedist for unassigned vertices (${}^c$ is complement):\hspace*{-5mm}}
     $\edgedists[j] \leftarrow \edgedists[j] + \bsumedgedist(e[i,\abidnew^c], f[j, \aobjnew^c])$\;
     \tcp*[l]{from here on:\ compensations for externalities}
     $\compensations[j] \leftarrow 0$\;
     $\abid \leftarrow \abidnew \setminus \{i\}$ \tcp*[r]{indices of bidders that stay assigned to same objects}
     $\aobj \leftarrow \aobjnew \setminus \{j\}$ \tcp*[r]{$=\perstoobj[\abid] = \perstoobjhypo[\abid]$}
     $j_0 \leftarrow \perstoobj[i]$\;
     \If(\tcp*[f]{compensation for unassigning $j_0$}){$j_0 \neq \na$}{
        $\distbefore \leftarrow \sumedgedist(e[i, \abid], f[j_0, \aobj])$\;
        $\distafter \leftarrow \sumedgedist(e[i, \abid], f[j, \aobj])$\;
        $\compensations[j] \leftarrow \distafter - \distbefore$\;
     }
     $i_0 \leftarrow \which(\perstoobj = j)$\;
     \If(\tcp*[f]{compensation for unassigning $i_0$}){$i_0 \neq \na$}{
        $\distbefore \leftarrow \sumedgedist(e[i_0, \abid], f[j, \aobj])$\;
        $\distafter \leftarrow \sumedgedist(e[i, \abid], f[j, \aobj])$\;
        $\compensations[j] \leftarrow \compensations[j] + \distafter - \distbefore$\;
     }
}
%
  $\persvals \leftarrow \rep(\maxcost, n) - d[i,] * \vfact - (\edgedists+\compensations) * \efact / 2 - \prices$\;
\KwRet{\persvals}\;
\end{algorithm}\DecMargin{0em}

\setlength{\bibhang}{2em}
\setlength{\bibitemsep}{0.2\baselineskip}
\printbibliography

\end{document}